\documentclass{amsart}

\usepackage{latexsym,amsfonts,amssymb,exscale,enumerate}
\usepackage{amsmath,amsthm,amscd}
\usepackage{hyperref}

\input xy
\usepackage[all]{xy}
\xyoption{line}
\xyoption{arrow}
\xyoption{color}
\SelectTips{cm}{}
\usepackage{tikz}
\usetikzlibrary{decorations.markings}
\usetikzlibrary{decorations.pathreplacing}
\newcommand{\hackcenter}[1]{
 \xy (0,0)*{#1}; \endxy}

\tikzset{->-/.style={decoration={
  markings,
  mark=at position #1 with {\arrow{>}}},postaction={decorate}}}

\tikzset{middlearrow/.style={
        decoration={markings,
            mark= at position 0.5 with {\arrow{#1}} ,
        },
        postaction={decorate}
    }
}

\usepackage{graphicx}
\usepackage{color}
\newcommand{\onen}{{\mathbf 1}_{n}}
\newcommand{\onenn}[1]{{\mathbf 1}_{#1}}

\theoremstyle{plain}
\newtheorem{theorem}{Theorem}

\newtheorem{proposition}[theorem]{Proposition}
\newtheorem{lemma}[theorem]{Lemma}

\newtheorem{exercise}[theorem]{Exercise}

\newtheorem{problem}[theorem]{Problem}

\theoremstyle{definition}

\newtheorem{conjecture}[theorem]{Conjecture}
\theoremstyle{remark}
\newtheorem{remark}[theorem]{Remark}

\theoremstyle{definition}
\newtheorem{thm}{Theorem}[section]
\newtheorem{cor}[thm]{Corollary}

\newtheorem{lem}[thm]{Lemma}

\newtheorem{prop}[thm]{Proposition}
\newtheorem{defn}[thm]{Definition}
\newtheorem{example}{Example}

\newcommand{\U}{\dot{{\bf U}}}
\newcommand{\Ucat}{\cal{U}}

\newcommand{\UcatD}{\dot{\cal{U}}}

\newcommand{\B}{\dot{\mathbb{B}}}

\newcommand{\UA}{{_{\cal{A}}\dot{{\bf U}}}}
\newcommand{\sym}{{\rm Sym}}
\newcommand{\Mat}{{\rm Mat}}
\newcommand{\maps}{\colon}

\newcommand{\xsum}[2]{
  \xy
  (0,.4)*{\sum};
  (0,3.7)*{\scs #2};
  (0,-2.9)*{\scs #1};
  \endxy
}

\newcommand{\refequal}[1]{\xy {\ar@{=}^{#1}
(-1,0)*{};(1,0)*{}};
\endxy}
\newcommand{\cat}[1]{\ensuremath{\mbox{\bfseries {\upshape {#1}}}}}

\newcommand{\To}{\Rightarrow}

\newcommand{\Hom}{{\rm Hom}}
\newcommand{\HOM}{{\rm HOM}}
\renewcommand{\to}{\rightarrow}

\def\lra{{\longrightarrow}}
\def\NH{{\mathrm{NH}}}

\def\Id{\mathrm{Id}}

\def\mf{\mathfrak}
\def\shuffle{\,\raise 1pt\hbox{$\scriptscriptstyle\cup{\mskip
               -4mu}\cup$}\,}

\numberwithin{equation}{section}

\def\emph#1{{\sl #1\/}}

\let\hat=\widehat
\let\tilde=\widetilde

\let\theta=\vartheta
\let\epsilon=\varepsilon

\usepackage{bbm}
\def\C{{\mathbbm C}}
\def\N{{\mathbbm N}}

\def\Z{{\mathbbm Z}}
\def\Q{{\mathbbm Q}}

\def\cal#1{\mathcal{#1}}%
\def\1{\mathbbm{1}}%
\def\tr{\mathrm{tr}}%
\def\nn{\notag}

\def\la{\langle}
\def\ra{\rangle}

\newcommand{\ccbub}[1]{
\xybox{%
 (-6,0)*{};
  (6,0)*{};
  (-4,0)*{}="t1";
  (4,0)*{}="t2";
  "t2";"t1" **\crv{(4,6) & (-4,6)};
   ?(1)*\dir{>};
  "t2";"t1" **\crv{(4,-6) & (-4,-6)};
   ?(.3)*\dir{}+(0,0)*{\bullet}+(0,-3)*{\scs {#1}};
}}

\newcommand{\cbub}[1]{
\xybox{%
 (-6,0)*{};
  (6,0)*{};
  (-4,0)*{}="t1";
  (4,0)*{}="t2";
  "t2";"t1" **\crv{(4,6) & (-4,6)};
    ?(.95)*\dir{<};
  "t2";"t1" **\crv{(4,-6) & (-4,-6)};
   ?(.3)*\dir{}+(0,0)*{\bullet}+(0,-3)*{\scs {#1}};
}}
\newcommand{\bbe}[1]{\xybox{%
  (-2,0)*{};
  (2,0)*{};
  (0,0);(0,-18) **\dir{-}; ?(.5)*\dir{<}+(2.3,0)*{\scriptstyle{#1}};
}}

\newcommand{\bbf}[1]{\xybox{%
  (-2,0)*{};
  (2,0)*{};
  (0,0);(0,-18) **\dir{-}; ?(.5)*\dir{>}+(2.3,0)*{\scriptstyle{#1}};
}}

\newcommand{\bbpef}{\xybox{%
  (-6,0)*{};
  (6,0)*{};
  (-4,0)*{}="t1";
  (4,0)*{}="t2";
  "t1";"t2" **\crv{(-4,-6) & (4,-6)}; ?(.15)*\dir{>} ?(.9)*\dir{>};
}}
\newcommand{\bbpfe}{\xybox{%
  (-6,0)*{};
  (6,0)*{};
  (-4,0)*{}="t1";
  (4,0)*{}="t2";
  "t2";"t1" **\crv{(4,-6) & (-4,-6)}; ?(.15)*\dir{>} ?(.9)*\dir{>};
}}

\newcommand{\bbcfe}[1]{\xybox{%
  (-6,0)*{};
  (6,0)*{};
  (-4,0)*{}="t1";
  (4,0)*{}="t2";
  "t1";"t2" **\crv{(-4,6) & (4,6)}; ?(.15)*\dir{>} ?(.9)*\dir{>}
  ?(.5)*\dir{}+(0,2)*{\scriptstyle{#1}};
}}
\newcommand{\bbcef}[1]{\xybox{%
  (-6,0)*{};
  (6,0)*{};
  (-4,0)*{}="t1";
  (4,0)*{}="t2";
  "t2";"t1" **\crv{(4,6) & (-4,6)}; ?(.15)*\dir{>}
  ?(.9)*\dir{>} ?(.5)*\dir{}+(0,2)*{\scriptstyle{#1}};
}}

\newcommand{\lowrru}[1]{\xybox{%
  (-8,0)*{};
  (8,0)*{};
  (-6,-18)*{};(6,-9)*{} **\crv{(-6,-13) & (6,-15)} ?(1)*\dir{>};
  (6,-9)*{};(6,0)*{}  **\dir{-} ?(.3)*\dir{ }+(2,0)*{\scs {\bf j}};
}}

\newcommand{\lowllu}[1]{\xybox{%
  (-8,0)*{};
  (8,0)*{};
  (6,-18)*{};(-6,-9)*{} **\crv{(6,-13) & (-6,-15)} ?(1)*\dir{>};
  (-6,-9)*{};(-6,0)*{}  **\dir{-} ?(.3)*\dir{ }+(-2,0)*{\scs {\bf j}};
}}

\newcommand{\bbdl}[1]{\xybox{%
  (2,0);(0,-8) **\crv{(2,-2)&(0,-6)}; ?(.5)*\dir{>}
}}
\newcommand{\bbdlu}[1]{\xybox{%
  (2,0);(0,-8) **\crv{(2,-2)&(0,-6)}; ?(.5)*\dir{<}
}}
\newcommand{\bbdr}[1]{\xybox{%
  (-2,0);(0,-8) **\crv{(-2,-2)&(0,-6)}; ?(.5)*\dir{>}
}}
\newcommand{\bbdru}[1]{\xybox{%
  (-2,0);(0,-8) **\crv{(-2,-2)&(0,-6)}; ?(.5)*\dir{<}
}}

\newcommand{\lccbub}[1]{
\xybox{%
 (-12,0)*{};
  (12,0)*{};
  (-12,0)*{}="t1";
  (12,0)*{}="t2";
  "t2";"t1" **\crv{(12,14) & (-12,14)}; ?(0)*\dir{>} ?(1)*\dir{>};
  "t2";"t1" **\crv{(12,-14) & (-12,-14)}; ?(.3)*\dir{}+(2,-1)*{\scs #1};
}}
\newcommand{\lcbub}[1]{
\xybox{%
 (-12,0)*{};
  (12,0)*{};
  (-12,0)*{}="t1";
  (12,0)*{}="t2";
  "t2";"t1" **\crv{(12,14) & (-12,14)}; ?(0)*\dir{<} ?(1)*\dir{<};
  "t2";"t1" **\crv{(12,-14) & (-12,-14)}; ?(.3)*\dir{}+(2,-1)*{\scs #1};
}}

\newcommand{\scap}{
\xybox{%
(-6,0)*{};
  (6,0)*{};
 (-4,0)*{};(4,0)*{}; **\crv{(4,5) & (-4,5)};
 }}

\newcommand{\lcap}{
\xybox{%
 (-12,2)*{};(12,2)*{}; **\crv{(12,14) & (-12,14)};
 }}
\newcommand{\xlcap}{
\xybox{%
 (-20,0)*{};(20,0)*{}; **\crv{(20,22) & (-20,22)};
 }}

\newcommand{\scupfe}{\xybox{%
 (-4,0)*{};(4,0)*{}; **\crv{(4,-5) & (-4,-5)} ?(0)*\dir{<} ?(.95)*\dir{<};
 }}

\newcommand{\lcupfe}{\xybox{%
 (-12,0)*{};(12,0)*{}; **\crv{(12,-14) & (-12,-14)} ?(0)*\dir{<} ?(.95)*\dir{<};
 }}

\newcommand{\scupef}{\xybox{%
 (-4,0)*{};(4,0)*{}; **\crv{(4,-5) & (-4,-5)} ?(0)*\dir{>} ?(1)*\dir{>};
 }}

\newcommand{\lcupef}{\xybox{%
 (12,0)*{};(-12,0)*{} **\crv{(12,-12) & (-12,-12)} ?(0)*\dir{>} ?(1)*\dir{>};
 }}
\newcommand{\xlcupef}{\xybox{%
 (-20,0)*{};(20,0)*{}; **\crv{(20,-22) & (-20,-22)} ?(0)*\dir{>} ?(1)*\dir{>};
 }}
\newcommand{\ecross}{\xybox{%
(-6,0)*{};
  (6,0)*{};
 (-4,-4)*{};(4,4)*{} **\crv{(-4,-1) & (4,1)}?(1)*\dir{>};
 (4,-4)*{};(-4,4)*{} **\crv{(4,-1) & (-4,1)}?(1)*\dir{>};
 }}
\newcommand{\fcross}{\xybox{%
 (-4,-4)*{};(4,4)*{} **\crv{(-4,-1) & (4,1)}?(1)*\dir{<};
 (4,-4)*{};(-4,4)*{} **\crv{(4,-1) & (-4,1)}?(1)*\dir{<};
 }}
\newcommand{\fecross}{\xybox{%
 (-4,-4)*{};(4,4)*{} **\crv{(-4,-1) & (4,1)}?(1)*\dir{>};
 (4,-4)*{};(-4,4)*{} **\crv{(4,-1) & (-4,1)}?(1)*\dir{<};
 }}
\newcommand{\efcross}{\xybox{%
 (-4,-4)*{};(4,4)*{} **\crv{(-4,-1) & (4,1)}?(1)*\dir{<};
 (4,-4)*{};(-4,4)*{} **\crv{(4,-1) & (-4,1)}?(1)*\dir{>};
 }}
\newcommand{\seline}{\xybox{%
 (0,-4)*{};(0,4)*{} **\dir{-} ?(1)*\dir{>};
}}
\newcommand{\sfline}{\xybox{%
 (0,-4)*{};(0,4)*{} **\dir{-} ?(1)*\dir{<};
}}
\newcommand{\meline}{\xybox{%
 (0,-8)*{};(0,8)*{} **\dir{-} ?(1)*\dir{>};
}}
\newcommand{\mfline}{\xybox{%
 (0,-8)*{};(0,8)*{} **\dir{-} ?(1)*\dir{<};
}}

\newcommand{\lfline}{\xybox{%
 (0,-12)*{};(0,12)*{} **\dir{-} ?(1)*\dir{<};
}}

\DeclareGraphicsExtensions{.pdf,.eps}
\usepackage{psfrag}
\usepackage{bbm}

\def\cal#1{\mathcal{#1}}

\def \k {\mathbbm{k}}
\def \Z {\mathbbm{Z}}

\def \N {\mathbbm{N}}
\def \Q {\mathbbm{Q}}
\def \E {\mathcal{E}}

\def \U {\mathcal{U}}

\def \B {\mathcal{B}}
\def \C {\mathcal{C}}
\def \Tr{\operatorname{Tr}}

\def \Span{\operatorname{Span}}
\def \Ob{\operatorname{Ob}}

\def \HH{\operatorname{HH}}

\def \Id {{\rm Id}}

\def\k{\mathbbm{k}}

\newcommand{\xto}[1]{{\overset{#1}{\longrightarrow}}}

\newcommand\nc{\newcommand}
\nc\rnc{\renewcommand}
\nc\Kar{\operatorname{Kar}}
\nc\End{\operatorname{End}}
\nc\modQ {{\mathbb Q}}
\nc\modZ {{\mathbb Z}}
\nc\simeqto{\overset{\simeq}{\longrightarrow }}
\nc\modC {{\mathcal C}}
\nc\modD {{\mathcal D}}
\nc\K{\mathcal {K}}
\nc\CC{\mathbf{C}}
\newcommand{\scs}{\scriptstyle}
\nc\calU{\mathcal{U}}
\nc\cU{\calU}
\nc\col{\colon\thinspace}
\nc\calA{\mathcal{A}}
\nc\Ab{\mathbf{Ab}}
\nc\Ko{K_0}
\nc\TrhorCC{\Tr^{\mathrm{hor}}(\CC)}
\nc\AdCat{\mathbf{AdCat}}
\nc\TrCC{\Tr(\CC)}
\nc\Udot{\dot{\mathcal{U}}}
\nc\diag{\mathrm{d}}
\nc\modU {\mathcal{U}}
\nc\bfU{\mathbf{U}}
\nc\dU{\dot{\mathbf U}}
\nc\dUZ{{_\modZ\dot{\mathbf U}}}
\nc\UZ{{_\modZ \mathbf U} }
\nc\fsl{\mathfrak{sl}}
\nc\Uaa{{\bf U} (\mathfrak{sl}_2\otimes \Q[t,t^{-1}])}
\nc\UZslt{{_\modZ\mathbf{U}} (\mathfrak{sl}_2\otimes \Q[t])}
\nc\UdZslt{{_\modZ\dot{\mathbf{U}}} (\mathfrak{sl}_2\otimes \Q[t])}
\nc\LL{L^+\fsl_2}
\nc\UL{\mathbf U(\LL)}
\nc\UZL{\UZ(\LL)}
\nc\dUZL{\dUZ(\LL)}
\nc\dUL{\dU(\LL)}
\nc{\im}{\rm im}
\nc\Kom{\rm Kom}
\nc\GL{\rm{GL}}
\nc\g{\mathfrak{g}}

\nc\tG{\tilde{G}} \nc\tE{\tilde{E}}
\nc\Vect{\rm Vect}
\nc{\Gras}{{\rm {Gr}}}
\nc\FMod{\rm FMod}
\nc\yto[1]{\underset{#1}{\to}}
\nc\Ear{\yto{E}}

\newcommand{\scup}{\xybox{%
 (-4,0)*{};(4,0)*{}; **\crv{(4,-5) & (-4,-5)} ;
 }}
 \newcommand{\sline}{\xybox{%
 (0,-4)*{};(0,4)*{} **\dir{-};
}}

\newcommand{\naecross}{\xybox{%
(-6,0)*{};
  (6,0)*{};
 (-4,-4)*{};(4,4)*{} **\crv{(-4,-1) & (4,1)};
 (4,-4)*{};(-4,4)*{} **\crv{(4,-1) & (-4,1)};
 }}
\newcommand{\lcup}{\xybox{%
 (-12,0)*{};(12,0)*{}; **\crv{(12,-14) & (-12,-14)} ;
 }}
 \newcommand{\mline}{\xybox{%
 (0,-8)*{};(0,8)*{} **\dir{-} ;
}}

\newcommand\sE{{\cal{E}}}
\newcommand\sF{{\cal{F}}}

\newcommand{\onelp}{{\mathbf 1}_{\lambda'}}
\def\l{\lambda}
\newcommand{\iccbub}[2]{
\xybox{%
 (-6,0)*{};
  (6,0)*{};
  (-4,0)*{}="t1";
  (4,0)*{}="t2";
  "t2";"t1" **\crv{(4,6) & (-4,6)}; ?(.7)*\dir{}+(-2,0)*{\scs #2}
  ?(.05)*\dir{>} ?(1)*\dir{>};
  "t2";"t1" **\crv{(4,-6) & (-4,-6)};
   ?(.3)*\dir{}+(0,0)*{\bullet}+(0,-3)*{\scs {#1}};
}}
\newcommand{\icbub}[2]{
\xybox{%
 (-6,0)*{};
  (6,0)*{};
  (-4,0)*{}="t1";
  (4,0)*{}="t2";
  "t2";"t1" **\crv{(4,6) & (-4,6)};?(.7)*\dir{}+(-2,0)*{\scs #2};
   ?(0)*\dir{<} ?(.95)*\dir{<};
  "t2";"t1" **\crv{(4,-6) & (-4,-6)};
   ?(.3)*\dir{}+(0,0)*{\bullet}+(0,-3)*{\scs {#1}};
}}

\allowdisplaybreaks

\begin{document}
% ==============================================================================
% ==============================================================================

%\title{Trace decategorification of categorified quantum groups}

\date{\today}

\title{Trace as an alternative decategorification functor}

\author{Anna Beliakova}
\address{Universit\"at Z\"urich, Winterthurerstr. 190
CH-8057 Z\"urich, Switzerland}
\email{anna@math.uzh.ch}

\author{Zaur Guliyev}
\address{Universit\"at Z\"urich, Winterthurerstr. 190
CH-8057 Z\"urich, Switzerland}
\email{zaur.guliyev@math.uzh.ch}

\author{Kazuo Habiro}
\address{Research Institute for Mathematical Sciences, Kyoto University, Kyoto, 606-8502, Japan}
\email{habiro@kurims.kyoto-u.ac.jp}

\author{Aaron D.~Lauda}
\address{University of Southern California, Los Angeles, CA 90089, USA}
\email{lauda@usc.edu}

\begin{abstract}
Categorification is a process of lifting structures to a higher
categorical level. The original  structure can then be recovered by means of
the so-called ``decategorification'' functor.  Algebras are typically categorified to additive categories with additional structure and
decategorification is usually given by the (split)
Grothendieck group.
In this article we study an  alternative  decategorification functor
given by the trace or the zeroth Hochschild--Mitchell homology.
We show that this form of decategorification endows any 2-representation of the categorified quantum $\fsl_n$ with an action of the current algebra $\bfU(\fsl_n[t])$ on its center.
\end{abstract}
%% \begin{abstract}
%% We compute the traces $\Tr(\Ucat)$ and $\Tr(\U^*)$ of the
%% $2$-categories $\Ucat$ and $\U^*$, resptctively, introduced by the
%% third author in the context of the categorification of quantum
%% $\mathfrak{sl}_2$.  Recall that $\Ucat$ embeds into $\U^*$ as a
%% subcategory of degree zero $2$-morphisms, and $\UcatD=\Kar(\Ucat)$ is
%% the Karoubi envelope of $\Ucat$.  It turns out that
%% $\Tr(\UcatD)=\Tr(\Ucat)$ coincides with its Grothendieck group
%% $K_0(\UcatD)$.  However, $\Tr(\U^*)$ is isomorphic to the integral
%% idempotented version of the loop algebra $\bfU(\mathfrak{sl}_2\otimes
%% \Q[t])$, whose zero degree part is $K_0(\Kar(\U^*))$.
%% %%Thus, $\Tr(\U^*)$ has  richer structure than $K_0(\U^*)$.
%% In addition, our method allows us to
%% show that all other Hochschild--Mitchell homology groups
%% of $\Ucat$ are zero.
%% \end{abstract}

\maketitle
%\tableofcontents

% ==============================================================================
%
\section{Introduction}
%
% ==============================================================================
\newcommand{\onel}{{\mathbf 1}_{\lambda}}

% -------------------------------------------------------------------------------
%
\subsection{What is categorification?}
%
% -------------------------------------------------------------------------------

Categorification is perhaps best explained by defining decategorification.  Quite broadly, {\em decategorification} is a rigorously defined procedure for forgetting information and reducing the complexity of a given mathematical structure.   A bit more formally, a {\em decategorification} can be thought of as a map
\[
 \xy
  (-20,0)*+{\txt{$(n-1)$-category}}="1";
  (20,0)*+{\txt{$n$-category}}="2";
  {\ar_-{\cal{D}} "2";"1" };
 \endxy
\]
for simplifying an $n$-categorical structure into an $(n-1)$-categorical structure.

Categorifcation asks the question ``given a  specific $(n-1)$-category $A$, find an $n$-category $B$ such that $A \cong \cal{D}(B)$".  In this case, we say that $A$ is the decategorification of $B$ and that $B$ categorifies $A$. Often it is the case that the existence of a categorification of $A$ reveals new insights and hidden structure that could not be seen without knowledge of $B$.  For more introductory material on categorification see~\cite{BaezDolan,KMS,Lau3,Savage}.

An example when $n=1$ can be obtained by considering the 1-category $\cat{Vect}_{\k}$ of finite dimensional $\k$-vector spaces.  We can define a decategorification
\[
 \xy
  (-20,0)*+{\N}="1";
  (20,0)*+{\cat{Vect}_{\k}}="2";
  {\ar_{\cal{D}} "2";"1" };
 \endxy
\]
from the 1-category $\cat{Vect}_{\k}$ to the $0$-category, or set, of natural numbers $\N$ by setting $\cal{D}=\dim$.  In this example a natural number $n$ is categorified by any vector space $V$ with $\dim(V)=n$.  This categorification lifts much of the structure of natural numbers, i.e. $\dim(V \oplus W) = \dim(V) + \dim (W)$ and $\dim (V \otimes W) = \dim (V) \times \dim (W)$.

Generalizing the previous example, let $\cat{gVect}_{\k}$ denote the category of graded vector spaces $V = \oplus_{k \in \Z} V_k$.  One way to decategorify a graded vector space is to take its graded dimension
\[
 \xy
  (-20,0)*+{\N[q,q]^{-1}}="1";
  (20,0)*+{\cat{gVect}_{\k}}="2";
  {\ar_-{\cal{D} = \dim_q} "2";"1" };
 \endxy
\]
where
\[
 \dim_q V := \sum_{k \in \Z} q^k \dim V_k.
\]
This example can be extended another direction by considering the category $\Kom(\cat{gVect}_{\k})$ of complexes of graded vector spaces. The graded Euler characteristic gives rise to a decategorification map taking a bounded complex of graded vector spaces into a Laurent polynomial $\chi(C_{\bullet})=\sum_i (-1)^i \dim_q C^i \in \Z[q,q^{-1}]$.  While these examples may seem somewhat trivial, these elementary ideas resurface in more sophisticated categorifications discussed below.

Several important observations are in order.  A categorification need not be unique, in the sense that there may be two different
$n$-categories $B$ and $B'$ with $A=\cal{D}(B)=\cal{D}(B')$.  Furthermore, there may be several different ways to simplify a given
$n$-categorical structure into an $(n-1)$-categorical structure.  That is to say, it may be possible to decategorify a given structure in more than one way.  In this article we will focus on two ways of decategorifying an (additive) category to produce an abelian group.  These are the ``Grothendieck group" and ``trace".

% -------------------------------------------------------------------------------
%
\subsection{The Grothendieck group}
%
% -------------------------------------------------------------------------------

For any additive category $\C$, its split Grothendieck group $K_0(\C)$ is the abelian group generated by isomorphisms classes of its objects $\{[X]_{\cong} \, |\, X\in \Ob(\C)\}$, modulo the relation that $[A \oplus B]_{\cong} = [A]_{\cong}  + [B]_{\cong}$.  Categories are organized into a 2-categorical structure $\cat{Cat}$ consisting of categories, functors, and natural transformations.  (Here and throughout this article we ignore issues of size of categories.) Put into the framework described above, the (split) Grothendieck group can be thought of as the procedure for turning the 2-categorical structure of an (additive) category into a 1-categorical structure of an abelian group.
\[
 \xy
  (-20,0)*+{\cat{Ab}}="1";
  (20,0)*+{\cat{AdCat} .}="2";
  {\ar_-{\cal{D} = K_0 } "2";"1" };
 \endxy
\]

We showed above that we can decategorify a vector space $V$ into a natural number by taking its dimension $\dim(V)$ as a decategorification map.  This map shows up again when we decategorify the {\em category} of vector spaces $\cat{Vect}_{\k}$ using the Grothendieck group decategorification map.  By choosing a basis, every vector space is isomorphic to a direct sum
 of copies of the ground field $[V]_{\cong}=[\k^{\dim(V)}]_{\cong}=\dim(V)[\k]_{\cong}$ for any $V\in \Ob(\cat{Vect}_{\k})$.   Hence, the Grothendieck group $K_0(\cat{Vect}_{\k})$ of the category of $\k$-vector spaces can be identified with the infinite cyclic abelian group $\Z$ after sending $[\k]_{\cong} \mapsto 1$.

If $\C$ is a graded additive category, then the Grothendieck group $K_0(\C)$ acquires the structure of a $\Z[q,q^{-1}]$-module by declaring that $[x \la t \ra]_{\cong} = q^t [x]_{\cong}$.  For example, the split Grothendieck group of the graded additive category of graded vector spaces $\cat{gVect}_{\k}$ can be identified with $\Z[q,q^{-1}]$ by sending $[\k]_{\cong}\to 1$ since
\[
[V]_{\cong} = [\bigoplus_{k \in \Z} V_k ]_{\cong} = \sum_{k \in \Z} q^k [V_k]_{\cong}
 = \sum_{k \in \Z} q^k \dim V_k[\k]_{\cong} = \dim_q V [\k]_{\cong}.
\]
Again, we see our previous decategorification map $\dim_q$ sending a graded vector spaces to $\N[q,q^{-1}]$ appearing in the context of the split Grothendieck group taking the category of graded vector spaces $\cat{gVect}_{\k}$ into the abelian group $K_0(\cat{gVect}_{\k})$.

% - - - - - - - - - - - - - - - - - - - - - - - - - - - - - - - - -
%
\subsubsection{The Grothendieck group and Euler characteristics}
%
% - - - - - - - - - - - - - - - - - - - - - - - - - - - - - - - - -

The Grothendieck group has a universal property making it the
universal way of producing an abelian group from an additive category.  This universal property makes the Grothendieck group into a universal receptacle for generalized Euler characteristics.

If we denote by $\Kom(\cat{gVect}_{\k})$ the category of bounded complexes of graded vector spaces, then the graded Euler characteristic gives rise to a decategorification map
\[
 \xy
  (-20,0)*+{K_0(\cat{gVect}_\k)}="1";
  (20,0)*+{\Kom(\cat{gVect}_\k)}="2";
  {\ar_-{ \chi} "2";"1" };
 \endxy
\]
given by taking a complex $C^{\bullet}$ of graded vector spaces into a Laurent polynomial
\[
 \chi(C^\bullet)= \sum_i (-1)^i [C^i]_{\cong} =
 \sum_i (-1)^i \dim_q (C^i) [\k] \in
  K_0(\cat{gVect}_\k) \cong \Z[q,q^{-1}]\, .
\]

More generally, for any additive category $\C$, the Euler characteristic $\chi(C^\bullet)= \sum_i (-1)^i [C^i]_{\cong}$ of a complex in $\Kom(\C)$ is element of the Grothendieck group $K_0(\C)$.  Thus, we see that the Euler characteristic fits nicely with the Grothendieck group decategorification map.

% -------------------------------------------------------------------------------
%
\subsection{The trace decategorification map}
%
% -------------------------------------------------------------------------------

The trace, or zeroth  Hochschild homology, is another procedure for turning a linear category $\C$ into an abelian group.  It is defined by
 \begin{gather*}
  \Tr(\C ):=
\bigoplus_{x\in \Ob(\modC )}\modC (x,x)/\Span\{fg-gf\}
\end{gather*}
where $f$ and $g$ run through all pairs of morphisms $f:x\to y$ and $g:y\to x$ with $x,y\in \Ob(\C)$.
We write $[f]$ for the class of an endomorphism $f \maps x \to x$ in $\Tr(\C)$.
In section \ref{sec_trace_lincat} we show that the trace satisfies the property that $[f \oplus g] = [f] + [g]$.

In the category $\cat{Vect}_{\k}$, the class $[\phi]$ of an endomorphism $\phi \maps V \to V$ is
equal to the class $\tr(\phi)[1_{\k}]$, where $\tr(\phi)$ is the usual trace of the linear endomorphism $\phi$ (see exercise \ref{ex-1}).    Hence, the categorical trace of $\cat{Vect}_{\k}$ is closely connected to the usual trace map sending a vector space to an element of the ground field $\k$.

Note that $\dim_V=\tr(1_V)$ so that the trace can be seen as a generalization of the dimension decategorification map.   In Section~\ref{sec_trace_lincat} we show that
the trace $\Tr(\cat{Vect}_{\k})$ of the category of vector spaces can be identified with the ground field $\k$ and that the trace $\Tr(\cat{gVect}_{\k})$ of
the category of graded vector spaces is isomorphic to $\k[q,q^{-1}]$.

% - - - - - - - - - - - - - - - - - - - - - - - - - - - - - - - - -
%
\subsubsection{Traced Euler characteristics}
%
% - - - - - - - - - - - - - - - - - - - - - - - - - - - - - - - - -

The Euler characteristic in  $\cat{Vect}_{\k}$ can be recast in the language of traces by observing that $\dim_V=\tr(1_V)$, which allows to define so called traced  Euler characteristic
$$\chi_{\tr}(C^\bullet)=\sum_i (-1)^i\;\; \tr (1_{C^i}) \, .$$
This notion can be extended to any linear category $\C$ as follows. Given a complex
$C^{\bullet}\in   \Kom(\C)$,
\[
 \xy
  (-20,0)*+{\Tr(\C)}="1";
  (20,0)*+{\Kom(\C)}="2";
  {\ar_-{ \chi_{tr}} "2";"1" };
 \endxy
 \]
is defined by setting  $\chi_{tr}(C^{\bullet})=\sum_i (-1)^i [1_{C^i}]$, where $[f]$ denotes the class of the endomorphism $f$ in $\Tr(\C)$.
In general, $\chi_{\tr}$ does not need to coincide with the usual Euler characteristic.

% -------------------------------------------------------------------------------
%
\subsection{Comparing Grothendieck group and trace}
%
% -------------------------------------------------------------------------------

We have now described two flavors of decategorification.  One which we will refer to as the {\em Grothendieck decategorification} which is closely connected to the notion of dimension, graded dimension, and Euler characteristic.
Our second notion of decategorification, which we will call the {\em trace decategorification}, seems to generalize the Grothendieck group decategorification in much the same way that the trace of linear map generalizes the dimension of a vector space $\dim(V) = \tr(1_V)$.   The trace decategorification is closely connected with the usual trace and it gives rise to a notion of traced Euler characteristic.

These two notions of decategorification can actually produce the same results in some very interesting examples.
One of the most spectacular examples is the categorification
of the Jones polynomial by the Khovanov complex of graded $\Z$-modules, where
the Jones polynomial can be recovered as the graded Euler characteristic~\cite{Kh1,Kh2}.
Bar-Natan generalized Khovanov's construction by defining the so-called Khovanov bracket:
a complex over an additive category of special 2-dimensional cobordisms.
In his setting, one recovers the Jones polynomial by means of the
traced Euler characteristic (Theorem 6 \cite{Dror}). Hence,   the
functor $\Tr$ or the zeroth Hochschild--Mitchell homology replaces $K_0$ in the Bar-Natan setting.

%These examples may give an impression that  $K_0$ and $\Tr$ can be used interchangeably
%and in interesting cases they do produce the same results.

This survey paper is devoted to the comparison of $K_0$ and $\Tr$ as decategorification functors.  The interplay between these two notions of decategorification is surprisingly rich.  It can be thought of as a categorical counterpart to the Chern character map, see for example~\cite{CW}.

%We recall the definitions of these functors and emphasis few general facts about them:
%\begin{itemize}
%\item
% $\Tr$ is defined over the ground ring,
%however $K_0$ is essentially an integral invariant.
%\item
%For any linear category $\C$,  there exists a homomorphism
%$$h_{\C}: K_0(\C) \to \Tr(\C)$$
%which is neither injective nor surjective in general.
%\item
%$\Tr(\C)=\Tr(\Kar(\C))$ for any linear category, however
%$K_0(\C)\neq K_0(\Kar(\C))$ in general. Here $\Kar(\C)$ is the idempotent completion or Karoubi envelop
%of $\C$.
%\end{itemize}
%

% -------------------------------------------------------------------------------
%
\subsection{Quantum group categorifications}
%
% -------------------------------------------------------------------------------

There is a strong interplay between quantum link invariants and algebraic objects known as quantum groups. Given the existence of link homology theories categorifying link invariants, it is perhaps not surprising to discover that quantum groups can also be categorified.  In this article we will examine how the two notions of decategorification given by the Grothendieck group and the trace can be applied to categorified quantum groups.

Quantum groups are Hopf algebras associated to Lie algebras $\mf{g}$.  They are obtained by $q$-deforming the universal enveloping algebra $\mathbf{U}(\mf{g})$.   Lusztig's modified, or idempotented, form $\dot{\mathbf{U}}(\mf{g})$ of the quantum group $\mathbf{U}_q(\mf{g})$ is well suited for categorification. These algebras are equipped with a family of orthogonal idempotents $1_{\lambda}$ indexed by the weight lattice $X$ of $\mf{g}$.  Any $\k$-algebra $A$ equipped with orthogonal idempotents indexed by a set $X$ can be regarded as an additive category whose objects are the elements of $X$.  Given two elements $\lambda, \mu \in X$,  $\k$-vector space of maps from $\lambda$ to $\mu$ can be defined by $1_{\mu} A 1_{\lambda}$.  In this framework, the idempotents $1_{\lambda}$ are thought of as identity morphisms from the object $\lambda$ to itself.  Composition in the category
\[
1_{\mu'} A 1_{\lambda'} \times 1_{\mu} A 1_{\lambda}\to \delta_{\mu,\lambda'} 1_{\mu'} A 1_{\lambda}
\]
is given by multiplication in $A$.

By thinking of a quantum group $\mathbf{U}_q(\mf{g})$ as a linear category in this way, we are naturally led to consider a categorification of $\mathbf{U}_q(\mf{g})$ given by the structure of a 2-category and our decategorification map must reduce a linear 2-category to a linear 1-category.
%\[
% \xy
%  (-20,0)*+{\cat{LinCat}}="1";
%  (20,0)*+{\text{2-\cat{LinCat}}}="2";
%  {\ar_-{\cal{D} } "2";"1" };
% \endxy
%\]
%where $\cat{LinCat}$ denotes the 2-category of $\k$-linear categories, and $\text{2-\cat{LinCat}}$ denotes the 3-category of $\k$-linear 2-categories.

Associated to a field $\k$ and a symmetrizable Kac-Moody algebra $\mf{g}$, a 2-category $\UcatD(\mf{g})$ was defined in \cite{KL3} generalizing the $\mf{sl}_2$ case from \cite{Lau1}.   The 2-category $\UcatD(\mf{g})$ is the idempotent completion, or Karoubi envelope, of a 2-category $\Ucat(\mf{g})$ defined in terms of a graphical calculus.  More precisely, the objects of the category $\Ucat(\mf{g})$ are indexed by the weight lattice $X$ of $\mf{g}$.  The morphisms are formal direct sums of composites of maps $\onel \maps \lambda \to \lambda$, $\cal{E}_i \onel \maps \lambda \to \lambda + \alpha_i$, and $\cal{F}_i\onel \maps \lambda \to \lambda-\alpha_i$ corresponding to the Chevalley basis of $\mf{g}$.  For each 1-morphism $x$ there is also a grading shifted 1-morphism $x\la t\ra$ in $\Ucat(\mf{g})$.  The 2-morphisms are given by $\Bbbk$-linear degree preserving combinations of certain planar diagrams modulo local relations. See Figure~\ref{fig_Ucat-diagram} for an example of planar diagram in $\Ucat(\mf{g})$.

\begin{figure}[h]
 $
 \vcenter{
    \xy 0;/r.16pc/:
  (-8,-4)*{\fecross};(8,-4)*{\fecross};
  (0,2)*{\scap};(0,4)*{\lcap};
  (0,-10)*{\scupfe};(0,-14)*{\lcupef};
  (13,-12)*{\scs i}; (13,5)*{\scs k};
   (-18,-18)*{\lambda'};
  \endxy} \;\;
  \vcenter{\xy 0;/r.16pc/:
 (-4,-15)*{}; (-20,25) **\crv{(-3,-6) & (-20,4)}?(0)*\dir{<}?(.6)*\dir{}+(0,0)*{\bullet};
 (-12,-15)*{}; (-4,25) **\crv{(-12,-6) & (-4,0)}?(0)*\dir{<}?(.6)*\dir{}+(.2,0)*{\bullet};
 ?(0)*\dir{<}?(.75)*\dir{}+(.2,0)*{\bullet};?(0)*\dir{<}?(.9)*\dir{}+(0,0)*{\bullet};
 (-28,25)*{}; (-12,25) **\crv{(-28,10) & (-12,10)}?(1)*\dir{>};
  ?(.2)*\dir{}+(0,0)*{\bullet}?(.35)*\dir{}+(0,0)*{\bullet};
 (-36,-15)*{}; (-36,25) **\crv{(-34,-6) & (-35,4)}?(1)*\dir{>};
 (-28,-15)*{}; (-42,25) **\crv{(-28,-6) & (-42,4)}?(1)*\dir{>};
 (-42,-15)*{}; (-20,-15) **\crv{(-42,-5) & (-20,-5)}?(1)*\dir{>};
 (6,10)*{\cbub{m}{}};
 (-23,0)*{\cbub{}{}};
 (-2,-12)*{\scs i}; (-14,-12)*{\scs k};  (-19,-12)*{\scs i};
    (-46,22)*{\scs \ell}; (-33,22)*{\scs i}; (-10,22)*{\scs j};
    (8,18)*{\scs j}; (-20,7)*{\scs \ell};
 \endxy}
 \vcenter{ \xy 0;/r.16pc/: (-14,8)*{\xybox{
 (0,-10)*{}; (-16,10)*{} **\crv{(0,-6) & (-16,6)}?(.5)*\dir{};
 (-16,-10)*{}; (-8,10)*{} **\crv{(-16,-6) & (-8,6)}?(1)*\dir{}+(.1,0)*{\bullet};
  (-8,-10)*{}; (0,10)*{} **\crv{(-8,-6) & (-0,6)}?(.6)*\dir{}+(.2,0)*{\bullet}?
  (1)*\dir{}+(.1,0)*{\bullet};
  (0,10)*{}; (-16,30)*{} **\crv{(0,14) & (-16,26)}?(1)*\dir{>};
 (-16,10)*{}; (-8,30)*{} **\crv{(-16,14) & (-8,26)}?(1)*\dir{>};
  (-8,10)*{}; (0,30)*{} **\crv{(-8,14) & (-0,26)}?(1)*\dir{>}?(.6)*\dir{}+(.25,0)*{\bullet};
 (-18,-8)*{\scs i}; (-6,-8)*{\scs i}; (2,-8)*{\scs k};
   }};
 (5,0)*{\lambda};
 \endxy}$
\caption{An example of a diagram representing a 2-morphism in $\Ucat(\mf{g})$ from
$
\cal{E}_i^2\cal{E}_{\ell}\cal{F}_i \cal{F}_k \cal{F}_i \cal{E}_i\cal{E}_i \cal{E}_k \onel
$
to $
 \cal{E}_{\ell} \cal{E}_i \cal{F}_j \cal{F}_j \cal{F}_i \cal{E}_j \cal{E}_i \cal{E}_k \cal{E}_i\onel\la t\ra
$.  The degree $\la t \ra$ of this 2-morphism is determined by the relationship between the simple roots $i,j,k, \ell$ in the Cartan datam associated to $\mf{g}$.} \label{fig_Ucat-diagram}
\end{figure}

Extending the split Grothendieck group decategorification map for linear categories, it is possible to define the split Grothendieck group $K_0$ of a linear 2-category.  It was shown in \cite{KL3} for $\mf{g}=\mf{sl}_n$ that applying this notion of decategorification to the Karoubi envelope $\UcatD(\mf{sl}_n)$ of the 2-category $\Ucat(\mf{sl}_n)$ produces the integral idempotented form of the quantum enveloping algebra $\mathbf{U}_q(\mf{sl}_n)$.  Hence, with this notion of decategorification the 2-category $\UcatD(\mf{sl}_n)$ can be viewed as a categorification of the quantum group $\mathbf{U}_q(\mf{sl}_n)$.  The $n=2$ case of this result was proven in~\cite{Lau1} and the case of general $\mf{g}$ appears in \cite{Web5}.  For a closely related construction see ~\cite{Rou2}.

More recently it has been shown that using a 2-categorical analog of the trace as the decategorification map produces some very rich structures~\cite{BHLZ,Marko,BGHL}.   The trace relation can be realized diagrammatically by considering the diagrams for $\Ucat(\mf{g})$ on an annulus.
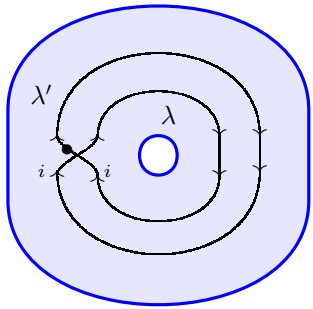
\[
 \vcenter{ \xy 0;/r.16pc/:
 (8,0)*{
 \hackcenter{\begin{tikzpicture}
  \path[draw,blue, very thick, fill=blue!10]
   (-2,-.6) to (-2,.6) .. controls ++(0,1.85) and ++(0,1.85) .. (2,.6)
   to (2,-.6)  .. controls ++(0,-1.85) and ++(0,-1.85) .. (-2,-.6);
    \path[draw, blue, very thick, fill=white]
    (-0.25,0) .. controls ++(0,.35) and ++(0,.35) .. (0.25,0)
            .. controls ++(0,-.35) and ++(0,-.35) .. (-0.25,0);
\end{tikzpicture}}
 };
(10,8)*{\lambda}; (-15,12)*{\lambda'};
  (8,8)*{\lcap};
  (8,14)*{\xlcap};
% (8,6)*{\scap};
 %(8,-6)*{\scupef}; (15,-3)*{\scs j};
 (8,-14)*{\xlcupef};
 (-8,0)*{\ecross};
 (20,0)*{\sfline};
 (28,0)*{\sfline};
 (8,-10)*{\lcupef};
(-2,-3)*{\scs i}; (-15,-3)*{\scs i};
 (-10,1)*{\bullet};
 \endxy}
\]
The invariance of the graphical calculus of $\Ucat(\mf{g})$ under planar deformations allows us to slide portions of the diagram around the annulus imposing the trace relation.

% -------------------------------------------------------------------------------
%
\subsection{Traces of categorified quantum groups}
%
% -------------------------------------------------------------------------------

In this article we hope to emphasize that the trace decategorification map $\Tr$ has various
advantages to $K_0$.
\begin{itemize}
  \item Using $\Tr$ we can work with the 2-category $\Ucat(\mf{g})$ rather than its Karoubi envelope, since $\Tr$ does not change under the passage to the Karoubi envelope.  This is a nice advantage as the passage to the Karoubi envelope often takes us out of the purely diagrammatic world.
\item Trace is defined for linear categories; in order for $K_0$ to be defined the category needs to be additive.
  \item The trace can have a richer structure than $K_0$ as is demonstrated by Theorems \ref{qqq} and \ref{thm_main}.
%
%  \item In diagrammatic settings where we are able to regard a diagram on an annulus as a diagram in the original calculus it is possible to use the trace to define on action on the center.
\end{itemize}

We outline the methods and tools used in studying $K_0$ and traces for categorified quantum groups.

% - - - - - - - - - - - - - - - - - - - - - - - - - - - - - - - - -
%
\subsubsection{Trace and $K_0$ of $\UcatD$}
%
% - - - - - - - - - - - - - - - - - - - - - - - - - - - - - - - - -

Just as the Euler characteristic and traced Euler characteristic agree for Khovanov homology,  we will explain that the 2-categories $\Ucat(\mf{g})$ simultaneously categorify the quantum group $\mathbf{U}(\mf{sl}_n)$ via the Grothendieck group decategorification functor $K_0$ and the trace decategorification functor $\Tr$.

\begin{thm}[\cite{KL3,BGHL}]
Let $\k$ be a field of characteristic $0$ and $\g=\fsl_n$.
Then there are isomorphisms
\[
 \Tr(\Ucat) \cong K_0(\UcatD) \otimes_\Z \k \cong \dot{\mathbf{U}}.
\]
and $\HH_i(\Ucat)=0$ for $i>0$.
\end{thm}

Again, this result is somewhat surprising since these two decategorification procedures need not agree in general.  We will explain that there is a large class of examples closely connected to geometric settings where this phenomenon is likely to occur.

Very recently a direct connection was established between categorified quantum groups and $\mf{sl}_n$ link homology theories~\cite{LQR,QR}.  Working in an enhanced foam category introduced by Christian Blanchet, one can realize all the foam relations in a Bar-Natan like setting as arising from the relations in the categorified quantum group $\Ucat(\mf{sl}_n)$ via so called foamation functors first studied in \cite{MackFoam}.   Once translated through these foamation functors, the theorem above immediately implies that the
traced and usual Euler characteristics in the  $\mf{sl}_n$ link homology theories coincide.

\begin{cor}
In $\fsl_n$ link homology theories, traced and usual Euler characteristics do coincide.
\end{cor}
% - - - - - - - - - - - - - - - - - - - - - - - - - - - - - - - - -
%
\subsubsection{Trace and $K_0$ of $\Ucat^{\ast}$}
%
% - - - - - - - - - - - - - - - - - - - - - - - - - - - - - - - - -

By allowing homogeneous, but not necessarily grading preserving 2-morphisms in $\Ucat(\mf{g})$ we can define a version of the 2-category with larger 2-hom spaces.  Let $\Ucat^{\ast}(\mf{g})$ denote the 2-category with the same objects and one morphisms as $\Ucat(\mf{g})$, but with 2-hom spaces between one morphisms $X \onel$ and $Y\onel$ given by
\[
 \Ucat^{\ast}(\mf{g})(X\onel, Y\onel) := \bigoplus_{t\in \Z} \Ucat(\mf{g})(X\onel, Y\onel \la t \ra).
\]
One can alternatively think of $\Ucat^{\ast}$ as the result of adding isomorphisms $X\onel \to X\onel \la t\ra $ for all $t\in \Z$.  This essentially kills the grading making Grothendieck ring $K_0(\Ucat^{\ast})$ only a $\Z$-module, rather than a $\Z[q,q^{-1}]$-module.  While this version of the 2-category is less interesting from a $K_0$ perspective, it has interesting consequences for the trace.

The trace of the graded version $\U^*$ of the categorified quantum $\fsl_2$ is isomorphic
to  the idempotent form of the current algebra $\bfU(\fsl_2[t]):=\bfU(\fsl_2\otimes \Q[t])$ \cite{BHLZ}. In \cite{BHLZ}
it is also shown that for the integral version $\U_\Z(\fsl_2)$ ($\k$ is replaced by $\Z$ in the definition), we have $$\Tr(\U_\Z)=K_0(\Udot_\Z)=K_0(\Udot)\, .$$
It is interesting that one loses the $q$ by considering the graded category $\U^{\ast}$, but one sees the positive part of the loop algebra $\mf{sl}_2 \otimes \Q[t,t^{-1}]$ appear.

We expect that these results generalize to finite type simply-laced Kac-Moody algebras $\mf{g}$.  However, to ease the exposition in this article, we set $\mf{g}=\fsl_n$ and describe evidence for the following conjecture in section~\ref{current-sln}.

\begin{conjecture}
For $\g=\fsl_n$, we have
$$K_0(\Udot_\Z)=
\Tr(\U_\Z)=K_0(\Udot)$$
and $\Tr(\U^\ast_\Z)$ coincides with the integral idempotented version of the current algebra
$\bfU(\fsl_n[t]) $ defined in Section \ref{current-sln}.
\end{conjecture}

In this paper we  construct a homomorphism $\bfU(\fsl_n[t]) \to \Tr(\Ucat^{\ast}(\fsl_n))$ utilizing the graphical
calculus and verifying relations directly  (see Proposition \ref{prop-cur}).  This implies that
%for any object $\lambda\in \Ob(\Udot)$,
the center of objects
$Z(\Udot):=\bigoplus_{\lambda \in X}\End_{\Udot}(\onel)$ of $\Udot$ (see section~\ref{action} for more details) is a module over the current algebra $\bfU(\fsl_n[t])$.
An immediate consequence of this result is the following theorem.

\begin{thm}\label{2-reps}
Any 2-representation $\U(\fsl_n) \to \K$ gives rise to an action of the current algebra
$\bfU(\fsl_n[t])$ on $\Tr(\K)$ and the
center $Z(\K)$ of objects in $\K$.
\end{thm}

This theorem is proven in Section \ref{action}. Let us discuss one implication of this result.   Brundan made the surprising discovery that one could define an action of the Lie algebra $\hat{\mf{g}}:=\mf{gl}_{\infty}(\mathbb{C})$ on the center $Z(\cal{O}) = \bigoplus_{\nu} Z(\cal{O}_{\nu})$ of all integral blocks $\cal{O}_{\nu}$ of category $\cal{O}$~\cite{Brundan}.
In this action, the Chevalley generators of $\hat{\mf{g}}$ act as certain trace maps associated to canonical adjunction maps between special translation functors that arise from tensoring with a $\mf{g}$-module and its dual.
In \cite{BO} it is explained that this action is closely connected to Ginzburg's geometric construction of representations of the general linear group~\cite{Gin}.
Theorem \ref{2-reps} provides a context  for understanding this surprising action by traces on category $\cal{O}$.  Indeed, Brundan's action is part of a broad phenomenon that occurs within the context of higher representation theory:  any categorified representation
of the quantum group $\mathbf{U}_q(\mf{sl}_n)$ immediately gives rise to an action of $\bfU(\mf{sl}_n[t])$ on the center of objects in the 2-representation via categorical traces. In particular,
an action of $\UcatD(\mf{sl}_n)$ on graded category $\cal{O}$ would automatically imply and generalize Brundan's result.  More details about this action will be given in \cite{BGHL}.

% - - - - - - - - - - - - - - - - - - - - - - - - - - - - - - - - -
%
\subsubsection{Organization of this paper}
%
% - - - - - - - - - - - - - - - - - - - - - - - - - - - - - - - - -
This paper is organized as follows. After recalling the
 diagrammatics for the trace  in pivotal categories, taking as example
 $\cat{Vect}_\k$,  we define the trace functor for linear and additive 1- and 2-categories and study its properties.
In Section 3 we recall the notion of (strongly) upper-triangular category from \cite{BHLZ}, which is
particularly convenient for computing all Hochschild--Mitchell homology groups.
In the next sections  different versions of the categorified
 quantum  $\fsl_n$ are defined and their decategorifications by means
of $K_0$ and $\Tr$ are compared.
The last chapter is devoted to the proof of Theorem \ref{2-reps}.

To make this  paper more  accessible for graduate students,
each section is supplied with many examples, exercises, pictures and references to textbooks.
Places where special knowledge  is required
are marked and can be ignored in the first reading.

\bigskip
%%%%%%%%%%%%%%%%%%%%%%%%%%%%%%%%%%%%%%%%%%%%%%%%%%%%%%%%%%%%%%%%%%
\noindent {\bf Acknowledgments:}
A.B. and Z.G. were  supported by Swiss National Science Foundation under Grant PDFMP2-141752/1.
K.H. was partially supported by JSPS Grant-in-Aid for Scientific Research (C) 24540077.
A.D.L  was partially supported by NSF grant DMS-1255334 and by the John Templeton Foundation.
A.D.L is extremely grateful to Mikhail Khovanov for sharing his insights and vision about higher representation theory.
Some of the ideas and calculations appearing in this article were done as part of this collaboration.  He is also grateful to Arun Ram for helpful decategorification discussions.
A.B. would like to thank Benjamin Cooper  for helpful conversations.

% =====================================================================
%
\section{Elementary notions of trace}
%
% =====================================================================

% ---------------------------------------------------------------------
%
\subsection{Traces of linear maps}
%
% ---------------------------------------------------------------------

In this section we recall the definition and properties of the trace
for any linear endomorphism. Throughout this paper $\k$ is assumed to be a field.

The trace of an $n \times n$ matrix $M=\{M_{ij}\}_{1 \leq i,j,\leq n}$ is the sum of its diagonal entries $\sum_{i} M_{ii}$.
Given two square matrices
$A,B$, the trace satisfies the {\em trace property}
\begin{equation} \label{trace-property}\tr(AB)=\tr(BA)\, .
\end{equation}

\begin{exercise}\label{ex-1} Show that
 property \eqref{trace-property} uniquely characterize the trace
in the following sense:
Two linear functions on the space of matrices satisfying
\eqref{trace-property} are proportional. Hint: any non-diagonal elementary
matrix   can be written as a commutator, i.e. $E_{ij}=[E_{ii}, E_{ij}]$.
\end{exercise}

More generally, for any finite-dimensional $\k$-vector space $V$, trace is a linear map
 $$ \tr: \End_{\k}(V) \to \k$$
such that the {\it trace relation} holds:
$$
\tr (fg)=\tr (gf)\quad\text{for any}\quad f,g \in \End_{\k}(V)\, .$$
%where $\GL(V)=\End_\k(V)$ is the space of $k$-linear endomorphisms of $V$.

\begin{exercise}
Show that this (basis independent) definition of the trace coincides with the usual one, where
 the trace of a matrix representing $f$ in some basis is taken.
Hint: use Exercise \ref{ex-1}.
\end{exercise}

As a simple  invariant of any linear endomorphism, the trace has a variety of
important applications. Let us just mention few of them.
\begin{itemize}
 \item  For any representation $\phi: G\to {\rm Aut}_{\k}(V)$ of a group $G$,
the character $\chi_{\phi} \maps G \to \k$ of the representation $\phi$ is the function sending each group element $g\in G$ to the trace of $\phi(g)$.  The character of a representation carries much of the essential information about a representation.  In particular, for $G$ a finite group and $\k=\mathbb{C}$, the complex representations of $G$ are determined up to isomorphism by their characters.

%form the set $\{ \tr(A(g))| g\in G\}$. These characters are an important aspect of representation theory in that two representations are  equivalent if all their characters coincide.

 \item  A Lie algebra $\g$ defined over $\k$ acts on itself via the adjoint action sending $x \in \g$ to the endomorphism ${\rm ad}_x$ defined by ${\rm ad}_x(y)=[x,y]$ for all $y \in \g$.  When $\g$ is finite dimensional, there is a symmetric bilinear form (or Killing form) $B: \g \times \g \to \k$ defined by $B(x,y):= \tr({\rm ad}_x{\rm ad}_y)$.  The non-degeneracy of this form is used to determine if a Lie algebra is semisimple.

\item In topology Markov trace is used to construct link invariants out of
representations of braid groups.
\end{itemize}

% - - - - - - - - - - - - - - - - - - - - - - - - - - - - - - - - -
%
\subsubsection{Diagrammatics for the trace}\label{tr-diag}
%
% - - - - - - - - - - - - - - - - - - - - - - - - - - - - - - - - -

The category $\cat{Vect}_{\k}$ has the structure of a pivotal monoidal category.   What this means is that this category admits a surprisingly rich diagrammatic calculus.  In this calculus, edges are labelled by finite dimensional vector spaces and planar diagrams  represent linear maps between the tensor product of vector spaces labelling the bottom strands to the tensor product of vector spaces labelling the top strands.
\[
\hackcenter{\begin{tikzpicture}
    \draw[very thick] (-1,1.75) [out=-90, in=150] to (0,.75);
    \draw[very thick] (1,1.75) [out=-90, in=30] to (0,.75);
    \draw[very thick] (-.5,-.25) [out=90, in=210] to (0,.75);
    \draw[very thick] (.5,-.25) [out=90, in=-30] to (0,.75);
    \draw[very thick] (0,1.75) -- (0,.75);
    \node[draw, thick, fill=blue!20,rounded corners=4pt,inner sep=3pt] () at (0,.75) {$f$};
    \node () at (.75,-.15) {$V_2$};
    \node () at (-.75,-.15) {$V_1$};
    \node () at (-1.25,1.6) {$W_1$};
    \node () at (1.25,1.6) {$W_2$};
    \node () at (.3,1.6) {$W_3$};
\end{tikzpicture}}
\qquad f \maps V_1 \otimes V_2 \longrightarrow W_1 \otimes W_2 \otimes W_3.
\]
In this graphical calculus the ground field $\k$ is depicted by the empty strand.  The ability to compose linear maps and to tensor linear maps is reflected in the graphical calculus by stacking the diagrams on top of each other and juxtaposing them side by side, respectively.

The ``pivotal" structure on the category $\cat{Vect}_{\k}$ of finite-dimensional vector spaces indicates the existence of certain diagrams in the graphical calculus.  To be concrete, fix a basis $\{ e_1, \dots , e_n\}$ of the vector space $V$.  Denote by $\{e_1^{\ast}, \dots, e_n^{\ast} \}$ the basis for the dual space $V^{\ast} = \Hom(V, \k)$, with $e_i^{\ast}(e_j) = \delta_{ij}$. Then there are linear maps
\[
 \hackcenter{\begin{tikzpicture}
    \draw [very thick] (0,.5) to (0,1) .. controls ++(0,.35) and ++(0,.35) .. (1,1) to (1,.5);
    \node[draw, thick, fill=blue!20,rounded corners=4pt,inner sep=3pt] () at (.5,1.25) {$ev$};
      \node () at (-.35,.65) {$V^{\ast}$};
        \node () at (1.25,.65) {$V$};
\end{tikzpicture}}
\qquad
\begin{array}{c}
  ev \maps V^{\ast} \otimes V \to \k \\
  e_i^{\ast} \otimes e_j \mapsto e_i^{\ast}(e_j)
\end{array}
\qquad \qquad
 \hackcenter{\begin{tikzpicture}
    \draw [very thick] (0,.5) to (0,0) .. controls ++(0,-.35) and ++(0,-.35) .. (1,0) to (1,.5);
    \node[draw, thick, fill=blue!20,rounded corners=4pt,inner sep=3pt] () at (.5,-.25) {$coev$};
     \node () at (-.35,.35) {$V$};
        \node () at (1.25,.35) {$V^{\ast}$};
\end{tikzpicture}}
\qquad
\begin{array}{c}
  coev \maps \k \to V\otimes V^{\ast} . \\
  1  \mapsto \sum_i e_i \otimes e_i^{\ast}
\end{array}
\]

\begin{exercise}
Prove that the maps $ev$ and $coev$ satisfy the following identities
\[
 (\Id_V \otimes ev) \circ (coev \otimes \Id_V) = \Id_V, \quad \text{and} \quad
 (ev \otimes \Id_{V^{\ast}})\circ (\Id_{V^{\ast}} \otimes coev) = \Id_{V^{\ast}}
\]
and depict these relations in the graphical calculus described above.
\end{exercise}

It is common practice in this diagrammatic calculus to simplify these diagrams
\[
 \hackcenter{\begin{tikzpicture}
    \draw [very thick] (0,.5) to (0,1) .. controls ++(0,.35) and ++(0,.35) .. (1,1) to (1,.5);
      \node () at (-.35,.65) {$V^{\ast}$};
        \node () at (1.25,.65) {$V$};
\end{tikzpicture}} \;\; :=\;\;
 \hackcenter{\begin{tikzpicture}
    \draw [very thick] (0,.5) to (0,1) .. controls ++(0,.35) and ++(0,.35) .. (1,1) to (1,.5);
    \node[draw, thick, fill=blue!20,rounded corners=4pt,inner sep=3pt] () at (.5,1.25) {$ev$};
      \node () at (-.35,.65) {$V^{\ast}$};
        \node () at (1.25,.65) {$V$};
\end{tikzpicture}}
\qquad \qquad
 \hackcenter{\begin{tikzpicture}
    \draw [very thick] (0,.5) to (0,0) .. controls ++(0,-.35) and ++(0,-.35) .. (1,0) to (1,.5);
     \node () at (-.35,.35) {$V$};
        \node () at (1.25,.35) {$V^{\ast}$};
\end{tikzpicture}} \;\; := \;\;
 \hackcenter{\begin{tikzpicture}
    \draw [very thick] (0,.5) to (0,0) .. controls ++(0,-.35) and ++(0,-.35) .. (1,0) to (1,.5);
    \node[draw, thick, fill=blue!20,rounded corners=4pt,inner sep=3pt] () at (.5,-.25) {$coev$};
     \node () at (-.35,.35) {$V$};
        \node () at (1.25,.35) {$V^{\ast}$};
\end{tikzpicture}}
\]
used to represent $ev$ and $coev$. Using the flip map in $\cat{Vect}_{\k}$
\[
 \hackcenter{\begin{tikzpicture}
    \draw [very thick] (0,0) to (1,1);
    \draw [very thick] (1,0) to (0,1);
     \node () at (-.25,.05) {$V$};
        \node () at (1.25,.05) {$W$};
        \node () at (-.25,.95) {$W$};
        \node () at (1.25,.95) {$V$};
\end{tikzpicture}}
\qquad
\begin{array}{rcc}
  \tau \maps V \otimes W &\to &W\otimes V  \\
   v \otimes w  &\mapsto  &w \otimes v,
   \end{array}
\]
we can also define other caps and caps
\[
 \hackcenter{\begin{tikzpicture}
    \draw [very thick] (0,.5) to (0,1) .. controls ++(0,.35) and ++(0,.35) .. (1,1) to (1,.5);
      \node () at (-.25,.65) {$V$};
        \node () at (1.35,.65) {$V^{\ast}$};
\end{tikzpicture}} \;\; :=\;\;
 \hackcenter{\begin{tikzpicture}
    \draw [very thick] (1,0) [out=90, in=-90] to (0,.7) to (0,1) .. controls ++(0,.35) and ++(0,.35) .. (1,1) to (1,.7) [out=-90, in=90] to (0,0);
    \node[draw, thick, fill=blue!20,rounded corners=4pt,inner sep=3pt] () at (.5,1.25) {$ev$};
      \node () at (-.25,.65) {$V^{\ast}$};
        \node () at (1.35,.65) {$V$};
\end{tikzpicture}}
\qquad \qquad
 \hackcenter{\begin{tikzpicture}
    \draw [very thick] (0,.5) to (0,0) .. controls ++(0,-.35) and ++(0,-.35) .. (1,0) to (1,.5);
     \node () at (-.25,.35) {$V^{\ast}$};
        \node () at (1.35,.35) {$V$};
\end{tikzpicture}} \;\; := \;\;
 \hackcenter{\begin{tikzpicture}
    \draw [very thick] (1,1) [out=-90, in=90] to (0,.3) to (0,0) .. controls ++(0,-.35) and ++(0,-.35) .. (1,0) to (1,.3) [out=90, in=-90] to (0,1);
    \node[draw, thick, fill=blue!20,rounded corners=4pt,inner sep=3pt] () at (.5,-.25) {$coev$};
     \node () at (-.35,.35) {$V$};
        \node () at (1.25,.35) {$V^{\ast}$};
\end{tikzpicture}}
\]

\begin{exercise}
Given a linear endomorphism $f:V\to V$ the following diagram
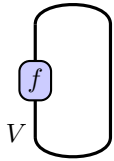
\[
\hackcenter{\begin{tikzpicture}
    \draw[very thick] (0,0) -- (0,1.5);
    \draw [very thick] (0,1.5) .. controls ++(0,.35) and ++(0,.35) .. (1,1.5) to (1,0)
    .. controls ++(0,-.35) and ++(0,-.35) .. (0,0);
     \node () at (-.25,.05) {$V$};
    \node[draw, thick, fill=blue!20,rounded corners=4pt,inner sep=3pt] () at (0,.75) {$f$};
\end{tikzpicture}}
\]
is multiplication by an element of the ground field $\k$.  Show that this element is exactly the trace of $f$.
\end{exercise}

Utilizing the evaluation and coevaluation map it is possible to dualize any map $f \maps V \to W$ to get a map $f^{\ast} \maps W^{\ast} \to V^{\ast}$.  Using these dual morphisms it is possible to slide linear maps through caps and cups
\[
\hackcenter{\begin{tikzpicture}
    \draw[very thick] (0,0) -- (0,1.5);
    \draw [very thick]  (1,1.5) to (1,0) .. controls ++(0,-.35) and ++(0,-.35) .. (0,0);
     \node () at (-.25,.05) {$V$};
       \node () at (-.35,1.25) {$W$};
       \node () at (1.35,1.25) {$V^{\ast}$};
    \node[draw, thick, fill=blue!20,rounded corners=4pt,inner sep=3pt] () at (0,.75) {$f$};
\end{tikzpicture}}
\;\; = \;\;
\hackcenter{\begin{tikzpicture}
    \draw[very thick] (0,0) -- (0,1.5);
    \draw [very thick]  (1,1.5) to (1,0) .. controls ++(0,-.35) and ++(0,-.35) .. (0,0);
     \node () at (1.35,.05) {$W^{\ast}$};
       \node () at (-.35,1.25) {$W$};
        \node () at (1.35,1.25) {$V^{\ast}$};
    \node[draw, thick, fill=blue!20,rounded corners=4pt,inner sep=3pt] () at (1,.75) {$f^{\ast}$};
\end{tikzpicture}}
\qquad \qquad
\hackcenter{\begin{tikzpicture}
    \draw[very thick] (0,0) -- (0,1.5);
    \draw [very thick] (0,1.5) .. controls ++(0,.35) and ++(0,.35) .. (1,1.5) to (1,0);
     \node () at (1.25,.1) {$V$};
     \node () at (-.35,1.25) {$W^{\ast}$};
     \node () at (1.25,1.25) {$W$};
    \node[draw, thick, fill=blue!20,rounded corners=4pt,inner sep=3pt] () at (1,.75) {$f$};
\end{tikzpicture}}
\;\; = \;\;
\hackcenter{\begin{tikzpicture}
    \draw[very thick] (0,0) -- (0,1.5);
    \draw [very thick] (0,1.5) .. controls ++(0,.35) and ++(0,.35) .. (1,1.5) to (1,0);
     \node () at (-.35,.1) {$W^{\ast}$};
     \node () at (-.35,1.25) {$V^{\ast}$};
     \node () at (1.35,1.25) {$V$};
    \node[draw, thick, fill=blue!20,rounded corners=4pt,inner sep=3pt] () at (0,.75) {$f^{\ast}$};
\end{tikzpicture}}
\]
Using these facts and the identification of $(V^{\ast})^{\ast}$ with $V$, it is not hard to verify the trace property:
\[
\hackcenter{\begin{tikzpicture}
    \draw[very thick] (0,0) -- (0,1.5);
    \draw [very thick] (0,1.5) .. controls ++(0,.35) and ++(0,.35) .. (1,1.5) to (1,0)
    .. controls ++(0,-.35) and ++(0,-.35) .. (0,0);
    \node[draw, thick, fill=blue!20,rounded corners=4pt,inner sep=3pt] () at (0,1.15) {$f$};
    \node[draw, thick, fill=blue!20,rounded corners=4pt,inner sep=3pt] () at (0,.35) {$g$};
\end{tikzpicture}}
\;\; = \;\;
\hackcenter{\begin{tikzpicture}
    \draw[very thick] (0,0) -- (0,1.5);
    \draw [very thick] (0,1.5) .. controls ++(0,.35) and ++(0,.35) .. (1,1.5) to (1,0)
    .. controls ++(0,-.35) and ++(0,-.35) .. (0,0);
    \node[draw, thick, fill=blue!20,rounded corners=4pt,inner sep=3pt] () at (0,1.15) {$g$};
    \node[draw, thick, fill=blue!20,rounded corners=4pt,inner sep=3pt] () at (0,.35) {$f$};
\end{tikzpicture}}
\]
In particular, for $f=1_V$, the identity endomorphism of $V$,
 $\tr(1_V)=\dim (V)$.  Here we are using the ``right" trace.  In a pivotal category one could also take the ``left" trace and these two traces need not coincide.

This pictorial presentation was extended to  the category
of colored ribbon graphs, graphically describing arbitrary ribbon categories such as the category of finite dimensional representation of quantum groups \cite{RT, Tur}.
%In particular, quantum link invariants were obtained as a trace of the image by that functor on a corresponding
%endomorphism (1-1 tangle).
This graphical calculus can also be seen as a specific instance of string diagrams in a 2-category~\cite{Street,js2,Lau3}.

% ---------------------------------------------------------------------
%
\subsection{Hochschild homology and the trace of a ring} \label{subsec_HH0ring}
%
% ---------------------------------------------------------------------

In this section we will explain how the Hochschild homology of a ring $R$ can be viewed as ``trace" of the ring $R$. Recall that for a ring $R$, the Hochschild homology $\HH_*(R)$ of $R$
can be defined as the homology of the Hochschild complex (see \cite{Mitchell,Loday}):
$$
C_{\bullet}=C_{\bullet}(R):\quad \quad \dots\longrightarrow C_n \xto{d_n} C_{n-1} \xto{d_{n-1}} \dots \xto{d_2} C_1 \xto{d_1} C_0 \longrightarrow 0,
$$
where $C_n(R)=R^{\otimes n+1}$ and
$$
d_n(a_0\otimes\dots\otimes
a_n):=\sum_{i=0}^{n-1}(-1)^ia_0\otimes\dots\otimes
a_ia_{i+1}\otimes\dots\otimes a_n+
(-1)^na_na_0\otimes a_1\otimes\dots\otimes a_{n-1}
$$
for $a_0,\ldots ,a_{n-1}\in R$.

\begin{exercise}
Show that the zeroth Hochschild homology of the ring $R$ is isomorphic to the quotient of $R$ by the commutator subgroup,
\[
\HH_0(R) \cong R/[R,R].
\]
In particular, $\HH_0(R) \cong R$ for any commutative ring $R$.
\end{exercise}

\begin{exercise}\label{mat}
If $\Mat_n(R)$ denotes the ring of $n \times n$ dimensional matrices with coefficients in $R$, show that
\[
\HH_0(\Mat_n(R) ) \cong \HH_0(R).
\] Hint: consider exercise~\ref{ex-1}.
\end{exercise}

% - - - - - - - - - - - - - - - - - - - - - - - - - - - - - - - - -
%
\subsubsection{Diagrammatics for zeroth Hochschild homology}
%
% - - - - - - - - - - - - - - - - - - - - - - - - - - - - - - - - -

Many interesting rings $R$ can be represented by diagrams in the plane modulo some local relations.  Multiplication in $R$ is represented by stacking diagrams on top of each other:
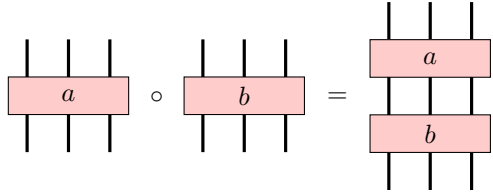
\[
\hackcenter{\begin{tikzpicture}
    \draw[very thick] (-.55,0) -- (-.55,1.5);
    \draw[very thick] (0,0) -- (0,1.5);
    \draw[very thick] (.55,0) -- (.55,1.5);
    \draw[fill=red!20,] (-.8,.5) rectangle (.8,1);
    \node () at (0,.75) {$a$};
\end{tikzpicture}}
\;\;
\circ \;\;
\hackcenter{\begin{tikzpicture}
    \draw[very thick] (-.55,0) -- (-.55,1.5);
    \draw[very thick] (0,0) -- (0,1.5);
    \draw[very thick] (.55,0) -- (.55,1.5);
    \draw[fill=red!20,] (-.8,.5) rectangle (.8,1);
    \node () at (0,.75) {$b$};
\end{tikzpicture}}
\;\; = \;\;
\hackcenter{\begin{tikzpicture}
    \draw[very thick] (-.55,0) -- (-.55,2.5);
    \draw[very thick] (0,0) -- (0,2.5);
    \draw[very thick] (.55,0) -- (.55,2.5);
    \draw[fill=red!20,] (-.8,.5) rectangle (.8,1);
    \draw[fill=red!20,] (-.8,1.5) rectangle (.8,2);
    \node () at (0,.75) {$b$};
    \node () at (0,1.75) {$a$};
\end{tikzpicture}}
\]
for $a,b \in R$.  Then $\HH_0(R)$ can be described by diagrams on the annulus, modulo the same local relations
as those for $R$ and the map
\[
 R \longrightarrow R/[R,R]
\]
is given on diagrams by
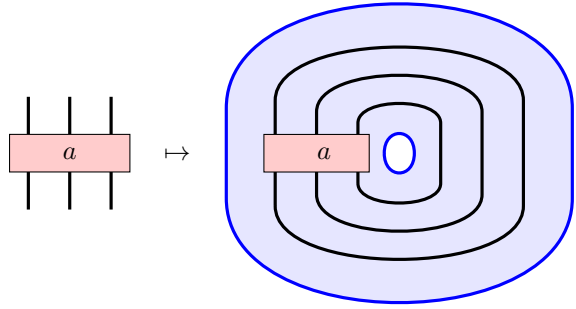
\[
\hackcenter{\begin{tikzpicture}
    \draw[very thick] (-.55,0) -- (-.55,1.5);
    \draw[very thick] (0,0) -- (0,1.5);
    \draw[very thick] (.55,0) -- (.55,1.5);
    \draw[fill=red!20,] (-.8,.5) rectangle (.8,1);
    \node () at (0,.75) {$a$};
\end{tikzpicture}}
\quad \mapsto
\quad
\hackcenter{\begin{tikzpicture}
      \path[draw,blue, very thick, fill=blue!10]
        (-2.3,-.6) to (-2.3,.6) .. controls ++(0,1.85) and ++(0,1.85) .. (2.3,.6)
         to (2.3,-.6)  .. controls ++(0,-1.85) and ++(0,-1.85) .. (-2.3,-.6);
        \path[draw, blue, very thick, fill=white]
            (-0.2,0) .. controls ++(0,.35) and ++(0,.35) .. (0.2,0)
            .. controls ++(0,-.35) and ++(0,-.35) .. (-0.2,0);
    \draw[very thick] (-1.65,-.7) -- (-1.65, .7).. controls ++(0,.95) and ++(0,.95) .. (1.65,.7)
        to (1.65,-.7) .. controls ++(0,-.95) and ++(0,-.95) .. (-1.65,-.7);
    \draw[very thick] (-1.1,-.55) -- (-1.1,.55) .. controls ++(0,.65) and ++(0,.65) .. (1.1,.55)
        to (1.1,-.55) .. controls ++(0,-.65) and ++(0,-.65) .. (-1.1, -.55);
    \draw[very thick] (-.55,-.4) -- (-.55,.4) .. controls ++(0,.35) and ++(0,.35) .. (.55,.4)
        to (.55, -.4) .. controls ++(0,-.35) and ++(0,-.35) .. (-.55,-.4);
    \draw[fill=red!20,] (-1.8,-.25) rectangle (-.4,.25);
    \node () at (-1,0) {$a$};
\end{tikzpicture}}
 \]

Notice that by sliding elements of $R$ around the annulus we see that $ab=ba$ in the quotient $\HH_0(R)$. Thus, the Hochschild homology of a ring $R$ exhibits similar properties to the trace of a linear map.

This method can be applied to Hecke algebras $H_{n,q}$ and their degenerations (0-Hecke algebras, nilCoxeter algebras) for a diagrammatic description of their Hochschild homology $\HH_0$~\cite{Brichard}.  In the next section we will use it to study $\HH_0$ of the nilHecke algebra.

% - - - - - - - - - - - - - - - - - - - - - - - - - - - - - - - - -
%
\subsubsection{Example calculations} \label{sample}
%
% - - - - - - - - - - - - - - - - - - - - - - - - - - - - - - - - -

The nilHecke ring $\NH_n$ is the graded unital associative ring generated by elements $x_1,\ldots,x_n$ of degree 2 and elements $\partial_1,\ldots,\partial_{n-1}$ of degree $-2$, subject to the relations
\begin{equation}\begin{split} \label{eq_rel_nilHecke}
&\partial_i^2=0,\qquad\partial_i\partial_{i+1}\partial_i=\partial_{i+1}\partial_i\partial_{i+1}\\
&x_i\partial_i-\partial_ix_{i+1}=1,\qquad\partial_ix_i-x_{i+1}\partial_i=1\\
&x_ix_j =x_jx_i \quad(i\neq j),\qquad\partial_i\partial_j =\partial_j\partial_i \quad(|i-j|>1),\\
&x_i\partial_j =\partial_jx_i  \quad(|i-j|>1).
\end{split}\end{equation}

The ring $\NH_n$ has graphical presentation via $n$-stranded planar diagrams, generated by dots and crossings, modulo planar isotopies and certain local relations.  In the graphical depiction of $\NH_n$ a dot on the $i$th strand corresponds to a generator $x_i$, while a crossing of the $i$ and $i+1$st strand represents $\partial_i$.  Then we impose local relations
\begin{equation}
\vcenter{
\xy 0;/r.17pc/:
	(-4,-4)*{};(4,4)*{} **\crv{(-4,-1) & (4,1)}?(1)*\dir{};
	(4,-4)*{};(-4,4)*{} **\crv{(4,-1) & (-4,1)}?(1)*\dir{};
	(-4,4)*{};(4,12)*{} **\crv{(-4,7) & (4,9)}?(1)*\dir{};
	(4,4)*{};(-4,12)*{} **\crv{(4,7) & (-4,9)}?(1)*\dir{};
	(-4,12); (-4,13) **\dir{-}?(1)*\dir{};
	(4,12); (4,13) **\dir{-}?(1)*\dir{};
	%(9,8)*{\lambda};
\endxy}
 \;\; =\;\; 0, \qquad \quad
\vcenter{\xy 0;/r.17pc/:
    (-4,-4)*{};(4,4)*{} **\crv{(-4,-1) & (4,1)}?(1)*\dir{};
    (4,-4)*{};(-4,4)*{} **\crv{(4,-1) & (-4,1)}?(1)*\dir{};
    (4,4)*{};(12,12)*{} **\crv{(4,7) & (12,9)}?(1)*\dir{};
    (12,4)*{};(4,12)*{} **\crv{(12,7) & (4,9)}?(1)*\dir{};
    (-4,12)*{};(4,20)*{} **\crv{(-4,15) & (4,17)}?(1)*\dir{};
    (4,12)*{};(-4,20)*{} **\crv{(4,15) & (-4,17)}?(1)*\dir{};
    (-4,4)*{}; (-4,12) **\dir{-};
    (12,-4)*{}; (12,4) **\dir{-};
    (12,12)*{}; (12,20) **\dir{-};
    (4,20); (4,21) **\dir{-}?(1)*\dir{};
    (-4,20); (-4,21) **\dir{-}?(1)*\dir{};
    (12,20); (12,21) **\dir{-}?(1)*\dir{};
   %(18,8)*{\lambda};
\endxy}
 \;\; =\;\;
\vcenter{\xy 0;/r.17pc/:
    (4,-4)*{};(-4,4)*{} **\crv{(4,-1) & (-4,1)}?(1)*\dir{};
    (-4,-4)*{};(4,4)*{} **\crv{(-4,-1) & (4,1)}?(1)*\dir{};
    (-4,4)*{};(-12,12)*{} **\crv{(-4,7) & (-12,9)}?(1)*\dir{};
    (-12,4)*{};(-4,12)*{} **\crv{(-12,7) & (-4,9)}?(1)*\dir{};
    (4,12)*{};(-4,20)*{} **\crv{(4,15) & (-4,17)}?(1)*\dir{};
    (-4,12)*{};(4,20)*{} **\crv{(-4,15) & (4,17)}?(1)*\dir{};
    (4,4)*{}; (4,12) **\dir{-};
    (-12,-4)*{}; (-12,4) **\dir{-};
    (-12,12)*{}; (-12,20) **\dir{-};
    (4,20); (4,21) **\dir{-}?(1)*\dir{};
    (-4,20); (-4,21) **\dir{-}?(1)*\dir{};
    (-12,20); (-12,21) **\dir{-}?(1)*\dir{};
  %(10,8)*{\lambda};
\endxy}
 \label{eq_nil_rels}
  \end{equation}
\begin{equation}
 \xy 0;/r.18pc/:
  (4,6);(4,-4) **\dir{-}?(0)*\dir{}+(2.3,0)*{};
  (-4,6);(-4,-4) **\dir{-}?(0)*\dir{}+(2.3,0)*{};
  %(9,2)*{n};
 \endxy
 \quad =
\xy 0;/r.18pc/:
  (0,0)*{\xybox{
    (-4,-4)*{};(4,6)*{} **\crv{(-4,-1) & (4,1)}?(1)*\dir{}?(.25)*{\bullet};
    (4,-4)*{};(-4,6)*{} **\crv{(4,-1) & (-4,1)}?(1)*\dir{};
    % (8,1)*{ n};
     (-10,0)*{};(10,0)*{};
     }};
  \endxy
 \;\; -
\xy 0;/r.18pc/:
  (0,0)*{\xybox{
    (-4,-4)*{};(4,6)*{} **\crv{(-4,-1) & (4,1)}?(1)*\dir{}?(.75)*{\bullet};
    (4,-4)*{};(-4,6)*{} **\crv{(4,-1) & (-4,1)}?(1)*\dir{};
   %  (8,1)*{ n};
     (-10,0)*{};(10,0)*{};
     }};
  \endxy
 \;\; =
\xy 0;/r.18pc/:
  (0,0)*{\xybox{
    (-4,-4)*{};(4,6)*{} **\crv{(-4,-1) & (4,1)}?(1)*\dir{};
    (4,-4)*{};(-4,6)*{} **\crv{(4,-1) & (-4,1)}?(1)*\dir{}?(.75)*{\bullet};
 %    (8,1)*{ n};
     (-10,0)*{};(10,0)*{};
     }};
  \endxy
 \;\; -
  \xy 0;/r.18pc/:
  (0,0)*{\xybox{
    (-4,-4)*{};(4,6)*{} **\crv{(-4,-1) & (4,1)}?(1)*\dir{} ;
    (4,-4)*{};(-4,6)*{} **\crv{(4,-1) & (-4,1)}?(1)*\dir{}?(.25)*{\bullet};
 %    (8,1)*{ n};
     (-10,0)*{};(10,0)*{};
     }};
  \endxy \label{eq_nil_dotslide}
\end{equation}

By \cite[Section 2.5]{KLMS}, the nilHecke ring
$\NH_n$ is isomorphic to the $n!\times n!$ matrix algebra ${\rm Mat}(n!,\sym_n)$
with coefficients in the ring $\sym_n$ of symmetric polynomials in $n$ variables,
\begin{equation}
  \sym_n \cong \Z[x_1, \dots, x_n]^{S_n}.
\end{equation}
Hence, by exercise~\ref{mat} it follows that $\Tr(\NH_n) \cong \sym_n$. The isomorphism between $\NH_n$ and the matrix algebra holds in the graded case, with $n!$ replaced by the non-symmetric quantum factorial $(n)^!_{q^2}:=(n)_{q^2} (n-1)_{q^2} \dots (1)_{q^2}$ where $(n)_{q^2}=(1-q^{2n})/(1-q^2)$, and with each variable of $\sym_n$ assigned degree $2$.

Let us illustrate the isomorphism
$\Tr(\NH_n) \cong \sym_n$ of graded abelian groups by some computations:
  \begin{enumerate}
  \item The diagram $ \xy 0;/r.18pc/:
   (8,0)*{
 \hackcenter{\begin{tikzpicture}
  \path[draw,blue, very thick, fill=blue!10]
   (-1.4,-.6) to (-1.4,.6) .. controls ++(0,1.15) and ++(0,1.15) ..
   (1.4,.6)
   to (1.4,-.6)  .. controls ++(0,-1.15) and ++(0,-1.15) ..
   (-1.4,-.6);
    \path[draw, blue, very thick, fill=white]
    (-0.15,0) .. controls ++(0,.3) and ++(0,.3) .. (0.15,0)
            .. controls ++(0,-.3) and ++(0,-.3) .. (-0.15,0);
\end{tikzpicture}}
 };
 %(8,0)*{*};
  (8,8)*{\lcap};
 (8,6)*{\scap};
 (8,-6)*{\scup};
 (0,0)*{\naecross};
 (12,0)*{\sline}; (20,0)*{\sline}; (8,-10)*{\lcup};
% (-6,-5)*{\scs i};
 \endxy\;$
 has degree $-2$.  $\sym_n$ is non-negatively
  graded, so this diagram should be 0 in $\Tr$.  Indeed, we have
\begin{align}
 \xy 0;/r.16pc/:
   (8,0)*{
 \hackcenter{\begin{tikzpicture}
  \path[draw,blue, very thick, fill=blue!10]
   (-1.2,-.6) to (-1.2,.6) .. controls ++(0,1) and ++(0,1) ..
   (1.2,.6)
   to (1.2,-.6)  .. controls ++(0,-1) and ++(0,-1) ..
   (-1.2,-.6);
    \path[draw, blue, very thick, fill=white]
    (-0.15,0) .. controls ++(0,.3) and ++(0,.3) .. (0.15,0)
            .. controls ++(0,-.3) and ++(0,-.3) .. (-0.15,0);
  \end{tikzpicture}}
 };
% (8,0)*{*};
  (8,8)*{\lcap};
 (8,6)*{\scap};
 (8,-6)*{\scup};
 (0,0)*{\naecross};
 (12,0)*{\sline};
 (20,0)*{\sline};
 (8,-10)*{\lcup};
 %(15,-3)*{\scs i};
%  (-6,-3)*{\scs i};
 \endxy
  \;\;
 =
  \;\;
 \vcenter{ \xy 0;/r.16pc/:
    (8,4)*{
 \hackcenter{\begin{tikzpicture}
  \path[draw,blue, very thick, fill=blue!10]
   (-1.4,-.6) to (-1.4,.6) .. controls ++(0,1.15) and ++(0,1.15) ..
   (1.4,.6)
   to (1.4,-.6)  .. controls ++(0,-1.15) and ++(0,-1.15) ..
   (-1.4,-.6);
    \path[draw, blue, very thick, fill=white]
    (-0.15,0) .. controls ++(0,.3) and ++(0,.3) .. (0.15,0)
            .. controls ++(0,-.3) and ++(0,-.3) .. (-0.15,0);
  \end{tikzpicture}}
 };
% (8,4)*{*};
  (8,16)*{\lcap};
 (8,14)*{\scap};
 (8,-6)*{\scup};
 (0,0)*{\naecross};(0,8)*{\naecross};
 (12,4)*{\mline};
 (20,4)*{\mline};
 (8,-10)*{\lcup};
 (-2.8,6)*{\bullet};
 %(15,-3)*{\scs i};
 % (-6,-3)*{\scs i};
 \endxy}
  \;\; = \;\;
  \vcenter{ \xy 0;/r.16pc/:
   (8,0)*{
 \hackcenter{\begin{tikzpicture}
  \path[draw,blue, very thick, fill=blue!10]
   (-1.2,-.6) to (-1.2,.6) .. controls ++(0,1) and ++(0,1) ..
   (1.2,.6)
   to (1.2,-.6)  .. controls ++(0,-1) and ++(0,-1) ..
   (-1.2,-.6);
    \path[draw, blue, very thick, fill=white]
    (-0.15,0) .. controls ++(0,.3) and ++(0,.3) .. (0.15,0)
            .. controls ++(0,-.3) and ++(0,-.3) .. (-0.15,0);
  \end{tikzpicture}}
 };
  (8,8)*{\lcap};
 (8,6)*{\scap};
 (8,-6)*{\scup};
 (0,0)*{\naecross};(16,0)*{\naecross};
 (8,-10)*{\lcup};
(-2.5,1.5)*{\bullet};
%(15,-5)*{\scs i};
%  (-7,-5)*{\scs i};
 \endxy}   \;\; = \;
\vcenter{ \xy 0;/r.16pc/:
    (8,4)*{
 \hackcenter{\begin{tikzpicture}
  \path[draw,blue, very thick, fill=blue!10]
   (-1.4,-.6) to (-1.4,.6) .. controls ++(0,1.15) and ++(0,1.15) ..
   (1.4,.6)
   to (1.4,-.6)  .. controls ++(0,-1.15) and ++(0,-1.15) ..
   (-1.4,-.6);
    \path[draw, blue, very thick, fill=white]
    (-0.15,0) .. controls ++(0,.3) and ++(0,.3) .. (0.15,0)
            .. controls ++(0,-.3) and ++(0,-.3) .. (-0.15,0);
  \end{tikzpicture}}
 };
  (8,16)*{\lcap};
 (8,14)*{\scap};
 (8,-6)*{\scup};
 (0,0)*{\naecross};(0,8)*{\naecross};
 (12,4)*{\mline};
 (20,4)*{\mline};
 (8,-10)*{\lcup};
 (-2.5,9.5)*{\bullet};
 %(15,-3)*{\scs i};
%  (-6,-3)*{\scs i};
 \endxy}   \;\;\refequal{\eqref{eq_nil_rels}} \;
  0, \label{eq_bubble_rel_zero}
\end{align}
where the first equality follows from a combination of \eqref{eq_nil_dotslide} and the first equality in \eqref{eq_nil_rels}.

 \item The degree 0 part of $\Tr(\NH_2)$ is spanned by $ \xy 0;/r.16pc/:
   (8,0)*{
 \hackcenter{\begin{tikzpicture}
  \path[draw,blue, very thick, fill=blue!10]
   (-1.2,-.6) to (-1.2,.6) .. controls ++(0,1) and ++(0,1) ..
   (1.2,.6)
   to (1.2,-.6)  .. controls ++(0,-1) and ++(0,-1) ..
   (-1.2,-.6);
    \path[draw, blue, very thick, fill=white]
    (-0.15,0) .. controls ++(0,.3) and ++(0,.3) .. (0.15,0)
            .. controls ++(0,-.3) and ++(0,-.3) .. (-0.15,0);
  \end{tikzpicture}}
 };
  (8,8)*{\lcap};
 (8,6)*{\scap};
 (8,-6)*{\scup};
 (0,0)*{\naecross};
 (12,0)*{\sline};
 (20,0)*{\sline};
 (8,-10)*{\lcup};
 (-3.3,6)*{\bullet};
 %(15,-3)*{\scs i}; (-6,-3)*{\scs i};
 \endxy\;$.
 We have
\begin{center}
\makebox[0pt]{$
 \xy 0;/r.16pc/:
   (8,0)*{
 \hackcenter{\begin{tikzpicture}
  \path[draw,blue, very thick, fill=blue!10]
   (-1.2,-.6) to (-1.2,.6) .. controls ++(0,1) and ++(0,1) ..
   (1.2,.6)
   to (1.2,-.6)  .. controls ++(0,-1) and ++(0,-1) ..
   (-1.2,-.6);
    \path[draw, blue, very thick, fill=white]
    (-0.15,0) .. controls ++(0,.3) and ++(0,.3) .. (0.15,0)
            .. controls ++(0,-.3) and ++(0,-.3) .. (-0.15,0);
  \end{tikzpicture}}
 };
  (8,8)*{\lcap};
 (8,6)*{\scap};
 (8,-6)*{\scup};
 (0,0)*{\naecross};
 (12,0)*{\sline};
 (20,0)*{\sline};
 (8,-10)*{\lcup};
 (-3.3,6)*{\bullet};
 %(15,-3)*{\scs i};
 % (-6,-3)*{\scs i};
 \endxy
  \; \refequal{\eqref{eq_nil_dotslide}}  -
  \vcenter{ \xy 0;/r.16pc/:
    (8,4)*{
 \hackcenter{\begin{tikzpicture}
  \path[draw,blue, very thick, fill=blue!10]
   (-1.4,-.6) to (-1.4,.6) .. controls ++(0,1.15) and ++(0,1.15) ..
   (1.4,.6)
   to (1.4,-.6)  .. controls ++(0,-1.15) and ++(0,-1.15) ..
   (-1.4,-.6);
    \path[draw, blue, very thick, fill=white]
    (-0.15,0) .. controls ++(0,.3) and ++(0,.3) .. (0.15,0)
            .. controls ++(0,-.3) and ++(0,-.3) .. (-0.15,0);
  \end{tikzpicture}}
 };
  (8,16)*{\lcap};
 (8,14)*{\scap};
 (8,-6)*{\scup};
 (0,0)*{\naecross};(0,8)*{\naecross};
 (12,4)*{\mline};
 (20,4)*{\mline};
 (8,-10)*{\lcup};
 (-3.3,14)*{\bullet};(3.2,6)*{\bullet};
 %(15,-3)*{\scs i};
  %(-6,-3)*{\scs i};
 \endxy}
  \;=
\; - \vcenter{ \xy 0;/r.16pc/:
   (8,0)*{
 \hackcenter{\begin{tikzpicture}
  \path[draw,blue, very thick, fill=blue!10]
   (-1.2,-.6) to (-1.2,.6) .. controls ++(0,1) and ++(0,1) ..
   (1.2,.6)
   to (1.2,-.6)  .. controls ++(0,-1) and ++(0,-1) ..
   (-1.2,-.6);
    \path[draw, blue, very thick, fill=white]
    (-0.15,0) .. controls ++(0,.3) and ++(0,.3) .. (0.15,0)
            .. controls ++(0,-.3) and ++(0,-.3) .. (-0.15,0);
  \end{tikzpicture}}
 };
  (8,8)*{\lcap};
 (8,6)*{\scap};
 (8,-6)*{\scup};
 (0,0)*{\naecross};(16,0)*{\naecross};
 (8,-10)*{\lcup};
 (18.3,-1.5)*{\bullet};(2.7,1.5)*{\bullet};
 %(15,-3)*{\scs i};(-6,-3)*{\scs i};
 \endxy}   \; =
\; - \vcenter{ \xy 0;/r.18pc/:
    (8,4)*{
 \hackcenter{\begin{tikzpicture}
  \path[draw,blue, very thick, fill=blue!10]
   (-1.4,-.6) to (-1.4,.6) .. controls ++(0,1.15) and ++(0,1.15) ..
   (1.4,.6)
   to (1.4,-.6)  .. controls ++(0,-1.15) and ++(0,-1.15) ..
   (-1.4,-.6);
    \path[draw, blue, very thick, fill=white]
    (-0.15,0) .. controls ++(0,.3) and ++(0,.3) .. (0.15,0)
            .. controls ++(0,-.3) and ++(0,-.3) .. (-0.15,0);
  \end{tikzpicture}}
 };
  (8,16)*{\lcap};
 (8,14)*{\scap};
 (8,-6)*{\scup};
 (0,0)*{\naecross};(0,8)*{\naecross};
 (12,4)*{\mline};
 (20,4)*{\mline};
 (8,-10)*{\lcup};
 (2.8,9.5)*{\bullet};(-2.8,6)*{\bullet};
 %(15,-3)*{\scs i};(-6,-3)*{\scs i};
 \endxy}   \; =
  -\;\xy 0;/r.16pc/:
   (8,0)*{
 \hackcenter{\begin{tikzpicture}
  \path[draw,blue, very thick, fill=blue!10]
   (-1.2,-.6) to (-1.2,.6) .. controls ++(0,1) and ++(0,1) ..
   (1.2,.6)
   to (1.2,-.6)  .. controls ++(0,-1) and ++(0,-1) ..
   (-1.2,-.6);
    \path[draw, blue, very thick, fill=white]
    (-0.15,0) .. controls ++(0,.3) and ++(0,.3) .. (0.15,0)
            .. controls ++(0,-.3) and ++(0,-.3) .. (-0.15,0);
  \end{tikzpicture}}
 };
  (8,8)*{\lcap};
 (8,6)*{\scap};
 (8,-6)*{\scup};
 (0,0)*{\naecross};
 (12,0)*{\sline};
 (20,0)*{\sline};
 (8,-10)*{\lcup};
 (2.8,1.5)*{\bullet};
 %(15,-3)*{\scs i};(-6,-3)*{\scs i};
 \endxy
$ }
\end{center}
\begin{equation}
 \label{eq_bubble_rel}
\end{equation}
 which implies
\begin{equation}
 \xy 0;/r.16pc/:
   (8,0)*{
 \hackcenter{\begin{tikzpicture}
  \path[draw,blue, very thick, fill=blue!10]
   (-1.2,-.6) to (-1.2,.6) .. controls ++(0,1) and ++(0,1) ..
   (1.2,.6)
   to (1.2,-.6)  .. controls ++(0,-1) and ++(0,-1) ..
   (-1.2,-.6);
    \path[draw, blue, very thick, fill=white]
    (-0.15,0) .. controls ++(0,.3) and ++(0,.3) .. (0.15,0)
            .. controls ++(0,-.3) and ++(0,-.3) .. (-0.15,0);
  \end{tikzpicture}}
 };
  (8,8)*{\lcap};
 (8,6)*{\scap};
 (8,-6)*{\scup};
(-4,0)*{\sline};
 (4,0)*{\sline};
 (12,0)*{\sline};
 (20,0)*{\sline};
 (8,-10)*{\lcup};
 (-3.3,6)*{};
 %(15,-3)*{\scs i};(-6,-3)*{\scs i};
 \endxy
  \;\; = \;\; \vcenter{
    \xy 0;/r.16pc/:
   (0,-4)*{
 \hackcenter{\begin{tikzpicture}
  \path[draw,blue, very thick, fill=blue!10]
   (-1.2,-.6) to (-1.2,.6) .. controls ++(0,1) and ++(0,1) ..
   (1.2,.6)
   to (1.2,-.6)  .. controls ++(0,-1) and ++(0,-1) ..
   (-1.2,-.6);
    \path[draw, blue, very thick, fill=white]
    (-0.15,0) .. controls ++(0,.3) and ++(0,.3) .. (0.15,0)
            .. controls ++(0,-.3) and ++(0,-.3) .. (-0.15,0);
  \end{tikzpicture}}
 };
  (-8,-4)*{\naecross};(4,-4)*{\sline};(12,-4)*{\sline};
  (0,2)*{\scap};(0,4)*{\lcap};
  (0,-10)*{\scup};(0,-14)*{\lcup};
  (-9.8,-3)*{\bullet};
  \endxy}
  \;\; - \;\;
  \vcenter{
    \xy 0;/r.16pc/:
   (0,-4)*{
 \hackcenter{\begin{tikzpicture}
  \path[draw,blue, very thick, fill=blue!10]
   (-1.2,-.6) to (-1.2,.6) .. controls ++(0,1) and ++(0,1) ..
   (1.2,.6)
   to (1.2,-.6)  .. controls ++(0,-1) and ++(0,-1) ..
   (-1.2,-.6);
    \path[draw, blue, very thick, fill=white]
    (-0.15,0) .. controls ++(0,.3) and ++(0,.3) .. (0.15,0)
            .. controls ++(0,-.3) and ++(0,-.3) .. (-0.15,0);
  \end{tikzpicture}}
 };
  (-8,-4)*{\naecross};(4,-4)*{\sline};(12,-4)*{\sline};
  (0,2)*{\scap};(0,4)*{\lcap};
  (0,-10)*{\scup};(0,-14)*{\lcup};
  (-4.8,-6)*{\bullet};
  \endxy}
  \;\;\refequal{\eqref{eq_bubble_rel}}\;\;
  2  \;   \vcenter{\xy 0;/r.16pc/:
    (0,-4)*{
 \hackcenter{\begin{tikzpicture}
  \path[draw,blue, very thick, fill=blue!10]
   (-1.2,-.6) to (-1.2,.6) .. controls ++(0,1) and ++(0,1) ..
   (1.2,.6)
   to (1.2,-.6)  .. controls ++(0,-1) and ++(0,-1) ..
   (-1.2,-.6);
    \path[draw, blue, very thick, fill=white]
    (-0.15,0) .. controls ++(0,.3) and ++(0,.3) .. (0.15,0)
            .. controls ++(0,-.3) and ++(0,-.3) .. (-0.15,0);
  \end{tikzpicture}}
 };
  (-8,-4)*{\naecross};(4,-4)*{\sline};(12,-4)*{\sline};
  (0,2)*{\scap};(0,4)*{\lcap};
  (0,-10)*{\scup};(0,-14)*{\lcup};
  (-9.8,-3)*{\bullet};
  \endxy}
\end{equation}
  \end{enumerate}
%\end{examples}

More generally, the element
\begin{equation}
  e_n = \xy 0;/r.15pc/:
 (-12,-20)*{}; (12,20) **\crv{(-12,-8) & (12,8)}?(1)*\dir{>};
 (-4,-20)*{}; (4,20) **\crv{(-4,-13) & (12,2) & (12,8)&(4,13)}?(1)*\dir{>};?(.88)*\dir{}+(0.1,0)*{\bullet};
 (4,-20)*{}; (-4,20) **\crv{(4,-13) & (12,-8) & (12,-2)&(-4,13)}?(1)*\dir{>}?(.86)*\dir{}+(0.1,0)*{\bullet};
 ?(.92)*\dir{}+(0.1,0)*{\bullet};
 (12,-20)*{}; (-12,20) **\crv{(12,-8) & (-12,8)}?(1)*\dir{>}?(.70)*\dir{}+(0.1,0)*{\bullet};
 ?(.90)*\dir{}+(0.1,0)*{\bullet};?(.80)*\dir{}+(0.1,0)*{\bullet};
 \endxy
\end{equation}
(shown for $n=4$) is a minimal idempotent in $\NH_n$.  The composition
\begin{equation}
 \xymatrix{
 \sym_n  \ar@{^{(}->}[r] & \NH_n  \ar[r]& \Tr(\NH_n) \\ a \ar@{|->}[r] & a e_n &}
\end{equation}
is a grading-preserving isomorphism of abelian groups.
In particular, any basis of $\sym_n$ gives a basis of $\Tr(\NH_n)$.
Taking the basis $\{\epsilon_{\lambda}\}_{\lambda}$ of $\sym_n$, over all partitions with at
most $n$ parts, $\lambda=(\lambda_1,\dots, \lambda_n)$, $\lambda_1 \geq \dots
\geq \lambda_n$, $\epsilon_{\lambda} = \epsilon_{\lambda_1}\dots
\epsilon_{\lambda_n}$, $\epsilon_k$ the $k$-th elementary symmetric function,
gives us a basis of $\Tr(\NH_n)$.

Let $x^{\lambda}=x_1^{\lambda_1}\dots x_n^{\lambda_n} \in \NH_n$.

\begin{prop}[\cite{BHLZ}]
The image of $\{ x^{\lambda} e_n \}_{\lambda}$ with $\lambda$ as above is a basis
of $\Tr(\NH_n)$.
\end{prop}

We call this basis the standard basis of $\Tr(\NH_n)$.

\begin{problem}
 Determine the coefficients of inclusions $\xymatrix{\NH_{n_1} \otimes \NH_{n_2} \ar@{^{(}->}[r]^{i} &
\Tr(\NH_{n_1+n_2})  }$ in the standard basis.
\end{problem}

% - - - - - - - - - - - - - - - - - - - - - - - - - - - - - - - - -
%
\subsubsection{Zeroth Hochschild cohomology} \label{subsec_HHco}
%
% - - - - - - - - - - - - - - - - - - - - - - - - - - - - - - - - -

The zeroth Hochschild cohomology $\HH^0(R)$ of a ring $R$ can be identified with the center of $R$,
\[
 \HH^0(R) \cong Z(R) .
\]
The Hochschild cohomology is naturally a module over the Hochschild homology.  This action has a nice graphical interpretation as well.  The action of $Z(R)$ on $\HH_0(R)$ can be graphically understood by cutting the diagram of an element $b \in \HH_0(R)$ along any ray emanating from the origin and inserting the diagram of $a\in Z(R)$ there:
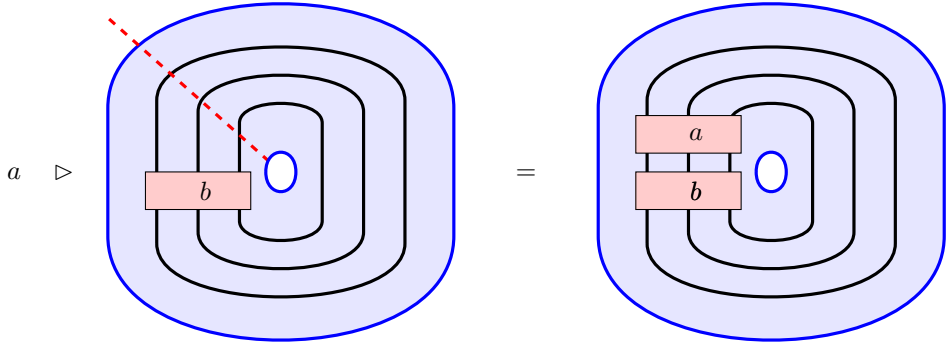
\[ a\quad\rhd\quad
\hackcenter{\begin{tikzpicture}
      \path[draw,blue, very thick, fill=blue!10]
        (-2.3,-.6) to (-2.3,1.1) .. controls ++(0,1.85) and ++(0,1.85) .. (2.3,1.1)
         to (2.3,-.6)  .. controls ++(0,-1.85) and ++(0,-1.85) .. (-2.3,-.6);
    \draw[very thick] (-1.65,-.7) -- (-1.65, 1.2).. controls ++(0,.95) and ++(0,.95) .. (1.65,1.2)
        to (1.65,-.7) .. controls ++(0,-.95) and ++(0,-.95) .. (-1.65,-.7);
    \draw[very thick] (-1.1,-.55) -- (-1.1,1.05) .. controls ++(0,.65) and ++(0,.65) .. (1.1,1.05)
        to (1.1,-.55) .. controls ++(0,-.65) and ++(0,-.65) .. (-1.1, -.55);
    \draw[very thick] (-.55,-.4) -- (-.55,.9) .. controls ++(0,.35) and ++(0,.35) .. (.55,.9)
        to (.55, -.4) .. controls ++(0,-.35) and ++(0,-.35) .. (-.55,-.4);
    \draw[fill=red!20,] (-1.8,-.25) rectangle (-.4,.25);
    \node () at (-1,0) {$b$};
   \draw[very thick, red, dashed] (0,.25) -- (-2.3,2.3);
            \path[draw, blue, very thick, fill=white]
            (-0.2,.25) .. controls ++(0,.35) and ++(0,.35) .. (0.2,.25)
            .. controls ++(0,-.35) and ++(0,-.35) .. (-0.2,.25);
\end{tikzpicture}}
\qquad  =\qquad
\hackcenter{\begin{tikzpicture}
      \path[draw,blue, very thick, fill=blue!10]
        (-2.3,-.6) to (-2.3,1.1) .. controls ++(0,1.85) and ++(0,1.85) .. (2.3,1.1)
         to (2.3,-.6)  .. controls ++(0,-1.85) and ++(0,-1.85) .. (-2.3,-.6);
    \draw[very thick] (-1.65,-.7) -- (-1.65, 1.2).. controls ++(0,.95) and ++(0,.95) .. (1.65,1.2)
        to (1.65,-.7) .. controls ++(0,-.95) and ++(0,-.95) .. (-1.65,-.7);
    \draw[very thick] (-1.1,-.55) -- (-1.1,1.05) .. controls ++(0,.65) and ++(0,.65) .. (1.1,1.05)
        to (1.1,-.55) .. controls ++(0,-.65) and ++(0,-.65) .. (-1.1, -.55);
    \draw[very thick] (-.55,-.4) -- (-.55,.9) .. controls ++(0,.35) and ++(0,.35) .. (.55,.9)
        to (.55, -.4) .. controls ++(0,-.35) and ++(0,-.35) .. (-.55,-.4);
    \draw[fill=red!20,] (-1.8,-.25) rectangle (-.4,.25);
    \node () at (-1,0) {$b$};
      \draw[fill=red!20,] (-1.8,.5) rectangle (-.4,1);
    \node () at (-1,0) {$b$};\node () at (-1,.75) {$a$};
            \path[draw, blue, very thick, fill=white]
            (-0.2,.25) .. controls ++(0,.35) and ++(0,.35) .. (0.2,.25)
            .. controls ++(0,-.35) and ++(0,-.35) .. (-0.2,.25);
\end{tikzpicture}}
\]

In section~\ref{action} we study an extension of this relationship between Hochschild homology and cohomology.

% =====================================================================
%
\section{Trace of a linear  category} \label{sec_trace_lincat}
%
% =====================================================================

In this section we define trace or zeroth  Hochschild homology groups for linear and additive categories, and  discuss their properties.

A category $\C$ is called {\it linear} if its hom-spaces are equipped
with structures of abelian groups and the composition maps are bilinear.
Such a structure is sometimes called an $\rm\bf {Ab}$-category, or a category enriched in the monoidal category $\rm \bf {Ab}$ of abelian groups.
A {\it linear} functor between two linear categories $\C$ and $\cal{D}$ is a functor $F:\C\to \cal{D}$ such that
for $x,y\in \Ob(\C)$ the map $F: \C(x,y)\to \cal{D}(F(x), F(y))$ is an abelian group homomorphism.

\nc\TL{\rm TL}
\begin{example} The Temperley-Lieb category $\TL$ is defined as a quotient of the $\Z[A,A^{-1}]$-span of  tangles by the Kauffman bracket skein relations:
\begin{align*}
 \left\langle \;\hackcenter{
\begin{tikzpicture}
%% draw webs
\draw [very thick] (0,0) -- (1,1);
\draw [very thick] (1,0) -- (.6, .4);
\draw [very thick] (.4,.6) -- (0,1);
\end{tikzpicture} } \;
\right\rangle
\;\;  &= \;\; A \; \left\langle \;
\hackcenter{
\begin{tikzpicture}
%% draw webs
\draw [very thick] (0,0) .. controls (.25,.25) and (.25,.75) .. (0,1);
\draw [very thick] (1,0) .. controls (.75,.25) and (.75,.75) .. (1,1);
\end{tikzpicture} } \; \right\rangle
\;\; + \;\; A^{-1}  \left\langle \;
 \hackcenter{
\begin{tikzpicture}
%% draw webs
\draw [very thick] (0,0) .. controls (.25,.25) and (.75,.25) .. (1,0);
\draw [very thick] (0,1) .. controls (.25,.75) and (.75,.75) .. (1,1);
\end{tikzpicture} } \; \right\rangle,
\\
 \left\langle \;\hackcenter{ T  \coprod
\hackcenter{\begin{tikzpicture}
%% draw webs
\draw [very thick] (0,.5).. controls ++(0,.35) and ++(0,.35) .. (.8,.5);
\draw [very thick] (0,.5).. controls ++(0,-.35) and ++(0,-.35) .. (.8,.5);
\end{tikzpicture} } } \; \right\rangle &= (-A^{-2}-A^{2}) \la T \ra.
\end{align*}
The objects of this category are natural numbers corresponding to the boundary points of the tangle.
A morphism in $\TL$ is a  $\Z[A,A^{-1}]$-linear combination of
crossingless matchings between the boundary points.
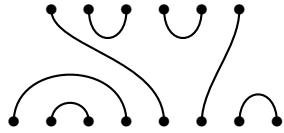
\[
\hackcenter{\begin{tikzpicture}
    \draw[thick] (-0.25,0.5) .. controls ++(0,-0.5) and ++(0,-0.5) .. (0.25,0.5);
    \draw[thick] (-1.25,0.5) .. controls ++(0,-0.5) and ++(0,-0.5) .. (-.75,0.5);
    \draw[thick] (-1.75,0.5) .. controls ++(-0,-0.5) and ++(0,0.85) .. (-0.25,-1);
    \draw[thick] (-2.25,-1) .. controls ++(-0,0.85) and ++(0,0.85) .. (-0.75,-1);
    \draw[thick] (-1.75,-1) .. controls ++(-0,0.35) and ++(0,0.35) .. (-1.25,-1);
    \draw[thick] (0.75,0.5) .. controls ++(-0,-0.5) and ++(0,0.5) .. (0.25,-1);
     \draw[thick] (1.25,-1) .. controls ++(-0,0.5) and ++(0,0.5) .. (0.75,-1);
      \node at (-0.25,0.5) {$\bullet$};
      \node at (-0.75,0.5) {$\bullet$};
      \node at (-1.25,0.5) {$\bullet$};
      \node at (-1.75,0.5) {$\bullet$};
      \node at (0.25,0.5) {$\bullet$};
      \node at (0.75,0.5) {$\bullet$};
      \node at (1.25,-1) {$\bullet$};
      \node at (-0.25,-1) {$\bullet$};
      \node at (-0.75,-1) {$\bullet$};
      \node at (-1.25,-1) {$\bullet$};
      \node at (-1.75,-1) {$\bullet$};
      \node at (-2.25,-1) {$\bullet$};
      \node at (.25,-1) {$\bullet$};
      \node at (.75,-1) {$\bullet$};
%    \draw[thick, color=blue, dashed] (.5,0) -- (.5,.75) [out=90, in=-90] to (0,1.5) -- (0,2);
%    \draw[thick, color=blue, dashed] (0,.8) [out=135, in=-90] to (-.5,1.5) -- (-.5,2);
%    \node[draw, thick, fill=blue!20,rounded corners=4pt,inner sep=3pt] () at (0,.5) {\small$\varphi$};
\end{tikzpicture}}
\]
\end{example}

% ---------------------------------------------------------------------
%
\subsection{Hochschild--Mitchell homology and the trace} \label{sec_HH0trace}
%
% ---------------------------------------------------------------------

Any ring $R$ with unit $1_R$ can be viewed as a linear category with one object $\ast$ and whose morphisms $\Hom(\ast,\ast)$ are given by the abelian group $R$.  Composition is given by multiplication in $R$ and the identity morphisms of $\ast$ corresponds to the unit $1_R \in R$.
Utilizing this observation, it possible to define the Hochschild--Mitchell homology of an arbitrary linear category as follows.

Let $\modC $ be a small linear category. Define the {\em Hochschild--Mitchell complex} of $\modC $
$$
C_{\bullet}=C_{\bullet}(\modC ):\quad \quad \dots\longrightarrow C_n \xto{d_n} C_{n-1} \xto{d_{n-1}} \dots \xto{d_2} C_1 \xto{d_1} C_0 \longrightarrow 0,
$$
where
$$
C_n=C_n(\modC ):=\bigoplus_{x_0,\dots,x_n\in\Ob(\C)}\C(x_n,x_0)\otimes\C(x_{n-1},x_n)\otimes\dots\otimes\C(x_0,x_1),
$$
$$
d_n(\sigma_n\otimes\sigma_{n-1}\otimes\dots\otimes\sigma_0):=\sum_{i=0}^{n-1}(-1)^i\sigma_n\otimes\dots\otimes\sigma_{n-i}\sigma_{n-i-1}\otimes\dots\otimes\sigma_0+
(-1)^n\sigma_0\sigma_n\otimes\sigma_{n-1}\otimes\dots\otimes\sigma_{1}.
$$
%for a fixed  order of morphisms or objects.

The Hochschild--Mitchell homology $\HH_{\bullet}(\modC )$ of $\modC $ is defined to
be the homology of the chain complex $C_{\bullet}(\modC )$.  One can easily check that if $\modC$ is a unital ring $R$ regarded as a linear category, then this definition agrees with the one given in section~\ref{subsec_HH0ring}.

In this article we will be especially concerned with the zeroth Hochschild--Mitchell homology
$\HH_0(\modC )=C_0/d_1(C_1)$ of $\modC$.  We introduce the notation
\[
\Tr(\modC ) := \HH_0(\modC )
\]
which we call the  {\em trace} of the linear category $\modC$.  This terminology is justified by the following observation. For any $x\in \Ob(\C)$ we denote by $\C(x,x)=\Hom_\C(x,x)=\End_\C(x)$ the
space of its endomorphisms.  Then it is an easy calculation to verify that the trace $\Tr(\modC )$ of the linear category $\modC$ is given by
\begin{gather*}
  \Tr(\C )=
\left( \bigoplus_{x\in \Ob(\modC )}\modC (x,x) \right)/\Span\{fg-gf\},
\end{gather*}
where $f$ and $g$ run through all pairs of morphisms $f\col
x\longrightarrow y$, $g\col y\longrightarrow x$ with $x,y\in
\Ob(\modC )$.

We denote by $[f]$ the equivalence class of $f\in \C(x,x)$ in $\Tr(\C)$.
The trace $\Tr$ gives rise to a functor from the category of (small) linear
categories to the category of abelian groups.

% ---------------------------------------------------------------------
%
\subsection{The trace and direct sums}
%
% ---------------------------------------------------------------------

In this section we show that the trace behaves well with respect to direct sums in additive categories.
An {\em additive} category is a linear category equipped with a zero
object and biproducts, also called direct sums.

\begin{lem}
  \label{r14}
  If $\modC $ is an additive category, then for $f\col x\rightarrow x$ and $g\col y\rightarrow y$,
  we have
  \begin{gather*}
    [f\oplus g]=[f]+[g]
  \end{gather*}
in $\Tr(\modC)$.
\end{lem}

\begin{proof}
  Since $f\oplus g=(f\oplus 0)+(0\oplus g)\col x\oplus y\rightarrow x\oplus y$, we have
  \begin{gather*}
    [f\oplus g]=[f\oplus 0]+[0\oplus g].
  \end{gather*}
  Now we have $[f\oplus0]=[ifp]=[pif]=[f]$ where $p\col x\oplus y\rightarrow x$
  and $i\col x\rightarrow x\oplus y$ are the projection and the
  inclusion. Similarly, we have $[0\oplus g]=[g]$. Hence the result.
\end{proof}

A linear (resp. additive) category enriched over $\cat{Vect}_\k$ will be called $\k$-linear, i.e.  its homs are $\k$-linear vector spaces and composition is $\k$-bilinear.

\begin{example}
Let $\C=\cat{Vect}_\k$ be the $\k$-linear category of finite dimensional $\k$-vector spaces.
Since any finite dimensional vector space is isomorphic to a finite direct sum of copies of $\k$,
by Lemma~\ref{r14} we have $\Tr(\cat{Vect}_\k)=\k$.
\end{example}

\nc\GrVect{{\rm gVect}}
\begin{example}
Let $\C=\cat{gVect}_\k$ be the $\k$-linear category of
finite dimensional $\Z$-graded vector spaces, i.e.
any $V\in\Ob(\C)$ decomposes as $V=\bigoplus_{n\in\Z} V_{n}$
with $\deg(x)=n$ for any $x\in V_n$ and
 morphisms in $\cat{gVect}_\k$ are degree preserving. Then
$\Tr(\cat{gVect}_\k)=\bigoplus_{n\in\Z}\k=\k [q,q^{-1}]$. The
multiplication by $q$ is interpreted as a shift of degree by one.
\end{example}

The previous examples can be further generalized. For a linear category $\modC $, there is a
``universal'' additive category generated by $\modC $, called the {\em
additive closure} $\modC ^{\oplus}$, in which the objects are formal
finite direct sums of objects in $\modC $ and the morphisms are
matrices of morphisms in $\modC$.  There is a canonical fully faithful
functor $ \modC \rightarrow \modC ^\oplus$.  Every linear functor
$F\col \modC \rightarrow \modD $  to an additive category
$\modD $ factors through it uniquely up to natural isomorphism.

\begin{exercise}
Show that for a linear category $\C$,
 $\Tr(\C^{\oplus})=\Tr(\C)$. Hint: Hom-spaces of $\C^{\oplus}$
are modules consisting of matrices whose entries are morphisms in $\C$.
%By exercise \ref{ex-1},
%trace of any matrix is proportional to the sum of its diagonal entries.)
\end{exercise}

\begin{exercise}
Show that $\Tr(\TL)=\Z[{\mathit{A}},{\mathit{A}}^{-1}][\mathit{z}]$, where $[z] = [1_1]$.  Hint:  $\Tr(\TL)$ can be naturally identified with the Kauffman skein module of the solid torus
by sending $[1_1]$ to the non-contractible circle.
\end{exercise}

% ---------------------------------------------------------------------
%
\subsection{The trace and idempotent completions}
%
% ---------------------------------------------------------------------

A {\em projection} or {\em idempotent} in $\cat{Vect}_{\k}$ is an endomorphism $p \maps V \to V$ satisfying the relation $p^2=p$.  In this case, $(\Id_V-p )\maps V \to V$ is also an idempotent and together these two idempotents decompose the space $V$ into a direct sum of the image of the projection $p$ and the image of the projection $\Id_V-p$.

More generally, idempotents can be defined in any category $\modC$ as endomorphisms $e \maps x \to x$ satisfying $e^2=e$, but here we will always assume that $\modC$ is a linear category.
\[\hackcenter{\begin{tikzpicture}
    \draw[very thick] (0,0) -- (0,1.5);
   % \draw [very thick] (0,1.5) .. controls ++(0,.35) and ++(0,.35) .. (1,1.5) to (1,0)
   % .. controls ++(0,-.35) and ++(0,-.35) .. (0,0);
    \node[draw, thick, fill=blue!20,rounded corners=4pt,inner sep=3pt] () at (0,1.05) {$e$};
    \node[draw, thick, fill=blue!20,rounded corners=4pt,inner sep=3pt] () at (0,.45) {$e$};
        \node at (0.15,.1) { $\scs x$};
    \node at (.15,1.4) { $\scs x$};
\end{tikzpicture}} \;\; = \;\;
\hackcenter{\begin{tikzpicture}
    \draw[very thick] (0,0) -- (0,1.5);
   % \draw [very thick] (0,1.5) .. controls ++(0,.35) and ++(0,.35) .. (1,1.5) to (1,0)
   % .. controls ++(0,-.35) and ++(0,-.35) .. (0,0);
    \node[draw, thick, fill=blue!20,rounded corners=4pt,inner sep=3pt] () at (0,.75) {$e$};
    \node at (0.15,.1) { $\scs x$};
    \node at (.15,1.4) { $\scs x$};
\end{tikzpicture}}
\]
An idempotent $e\col x\rightarrow x$ in $\modC $ is said to {\em split} if there is an object $y$ and morphisms $g\col x\rightarrow y$, $h\col y\rightarrow x$ such that $hg=e$ and $gh=1_y$.
\[\hackcenter{\begin{tikzpicture}
    \draw[very thick] (0,0) -- (0,1.5);
   % \draw [very thick] (0,1.5) .. controls ++(0,.35) and ++(0,.35) .. (1,1.5) to (1,0)
   % .. controls ++(0,-.35) and ++(0,-.35) .. (0,0);
    \node[draw, thick, fill=blue!20,rounded corners=4pt,inner sep=3pt] () at (0,1.05) {$h$};
    \node[draw, thick, fill=blue!20,rounded corners=4pt,inner sep=3pt] () at (0,.45) {$g$};
        \node at (0.15,.1) { $\scs x$};
    \node at (.15,1.4) { $\scs x$};
\end{tikzpicture}} \;\; = \;\;
\hackcenter{\begin{tikzpicture}
    \draw[very thick] (0,0) -- (0,1.5);
   % \draw [very thick] (0,1.5) .. controls ++(0,.35) and ++(0,.35) .. (1,1.5) to (1,0)
   % .. controls ++(0,-.35) and ++(0,-.35) .. (0,0);
    \node[draw, thick, fill=blue!20,rounded corners=4pt,inner sep=3pt] () at (0,.75) {$e$};
    \node at (0.15,.1) { $\scs x$};
    \node at (.15,1.4) { $\scs x$};
\end{tikzpicture}}
\qquad \qquad
\hackcenter{\begin{tikzpicture}
    \draw[very thick] (0,0) -- (0,1.5);
   % \draw [very thick] (0,1.5) .. controls ++(0,.35) and ++(0,.35) .. (1,1.5) to (1,0)
   % .. controls ++(0,-.35) and ++(0,-.35) .. (0,0);
    \node[draw, thick, fill=blue!20,rounded corners=4pt,inner sep=3pt] () at (0,1.05) {$g$};
    \node[draw, thick, fill=blue!20,rounded corners=4pt,inner sep=3pt] () at (0,.45) {$h$};
        \node at (0.15,.1) { $\scs y$};
    \node at (.15,1.4) { $\scs y$};
\end{tikzpicture}} \;\; = \;\;
\hackcenter{\begin{tikzpicture}
    \draw[very thick] (0,0) -- (0,1.5);
   % \draw [very thick] (0,1.5) .. controls ++(0,.35) and ++(0,.35) .. (1,1.5) to (1,0)
   % .. controls ++(0,-.35) and ++(0,-.35) .. (0,0);
    \node[draw, thick, fill=blue!20,rounded corners=4pt,inner sep=3pt] () at (0,.75) {$1_y$};
    \node at (0.15,.1) { $\scs y$};
    \node at (.15,1.4) { $\scs y$};
\end{tikzpicture}}
\]
%In an additive category $\modC$ it is convenient to write the splitting object $y$ as $\im (e)$.  If both idempotents $e$ and $1_x-e$ split, then they give rise to a direct sum decomposition $x = \im (e) \oplus \im (1_x-e)$ in $\modC$.

Unfortunately, even in an additive category $\modC$ it is not always possible to split idempotents.
%The main difficulty is making sense of the ``image of $e$" categorically.
Such phenomenon occurs quite often when studying additive categories $\modC$ defined diagrammatically by generators and relations~\cite{Kup,Mor,BNM-Kar}.  In this case, making sense out of the ``image" of a diagram is not usually possible.  A common solution is to enlarge the category $\modC$ by passing to the {\em Karoubi envelope} $\Kar(\mathcal C)$ (also called {\em idempotent}   or {\em Cauchy completion}) of $\modC $.  The Karoubi envelope can be thought of as a minimal enlargement of the category $\modC$ in which all idempotents split.

More formally, the Karoubi envelope $\Kar(\modC)$ is the category whose objects are pairs $(x,e)$ of objects $x\in \Ob(\C)$ and
an idempotent endomorphism $e : x\rightarrow x$ in $\mathcal
C$. The morphisms
$$f : (x,e)\rightarrow(y,e')$$
are morphisms $f\col x\rightarrow y$ in $\modC $ such that $f=e'fe$, or alternatively such that
\[
\hackcenter{\begin{tikzpicture}
    \draw[very thick] (0,0) -- (0,2);
   % \draw [very thick] (0,1.5) .. controls ++(0,.35) and ++(0,.35) .. (1,1.5) to (1,0)
   % .. controls ++(0,-.35) and ++(0,-.35) .. (0,0);
    \node[draw, thick, fill=blue!20,rounded corners=4pt,inner sep=3pt] () at (0,1) {$f$};
    \node at (0.15,.1) { $\scs x$};
    \node at (.15,1.9) { $\scs y$};
\end{tikzpicture}} \;\; = \;\;
\hackcenter{\begin{tikzpicture}
    \draw[very thick] (0,0) -- (0,2.5);
   % \draw [very thick] (0,1.5) .. controls ++(0,.35) and ++(0,.35) .. (1,1.5) to (1,0)
   % .. controls ++(0,-.35) and ++(0,-.35) .. (0,0);
    \node[draw, thick, fill=blue!20,rounded corners=4pt,inner sep=3pt] () at (0,1.25) {$f$};
    \node[draw, thick, fill=blue!20,rounded corners=4pt,inner sep=3pt] () at (0,.5) {$e$};
    \node[draw, thick, fill=blue!20,rounded corners=4pt,inner sep=3pt] () at (0,2) {$e'$};
        \node at (0.15,.1) { $\scs x$};
    \node at (.15,2.4) { $\scs y$};
\end{tikzpicture}} \; .
\]
Composition is induced by the composition in $\mathcal C$ and the identity morphism is
$e : (x,e)\rightarrow(x,e)$. $\Kar(\modC )$ is equipped with a linear category structure.

We can %write $\im (e)$ for the pair $(x,e)$ and
identify $(x,1_x)$ with the object $x$ of $\mathcal{C}$. This identification gives rise to a  natural embedding functor $\iota:\mathcal C\rightarrow
\Kar(\mathcal C)$ such that $\iota(x)=(x,1_x)$ for $x\in \Ob(\modC )$ and
$\iota(f\col x\rightarrow y)=f$.  The Karoubi envelope $\Kar(\modC )$ has
the universality property that if $F\col \modC \rightarrow \modD $ is
a linear functor to a linear category $\modD $ with split idempotents,
then $F$ extends to a functor from $\Kar(\modC)$ to $\modD$ uniquely
up to natural isomorphism \cite[Proposition 6.5.9]{Bor}.

\begin{example}
Consider the Temperley-Lieb category $\TL$ defined over $\mathbb{C}(q)$, where $q$ is an indeterminant with $q=-A^{2}$.   The Karoubi envelope $\Kar(\TL)$ is equivalent as a category to the category of finite-dimensional representations of quantum ${\fsl}_2$ (\cite{Kauffman} and \cite{Penrose}).
%after extending the coefficients to $\Z[A,A^{-1}] [ 1/(A^{2n}+A^{2n-4}+... + A^{-2n})   (n\ge1)]$
 The irreducible $N$-dimensional representation arises as the image of a certain idempotent  called {\em Jones-Wenzl} idempotent.
\end{example}

\begin{exercise}
Show that the Karoubi envelope of an additive category
is additive.
\end{exercise}

The following proposition illustrates one of the key advantages of the trace, namely its invariance under passing to the Karoubi envelope.

\begin{prop}
  \label{r1}
  The map $\Tr(\iota)\col \Tr(\modC )\longrightarrow \Tr(\Kar(\modC ))$ induced by $\iota$ is bijective.
\end{prop}

\begin{proof}
Recall that an endomorphism $f\maps (x,e) \to (x,e)$ in $\Kar(\modC)$ is just a morphism $f\maps x \to x$ in $\modC$ satisfying the condition that $f=efe$.  Define a map $u\col \Tr(\Kar(\modC ))\longrightarrow \Tr(\modC )$ sending $[f]\in \Tr(\Kar(\modC ))$ to $[f]\in \Tr(\modC )$.
Then one can check that $u$ is an inverse to $\Tr(\iota)$.
\end{proof}

\begin{example}
Adding Jones-Wenzl idempotents as  objects into $\TL$ does not affect  $\Tr(\TL)$.
\end{example}

% ---------------------------------------------------------------------
%
\subsection{The trace and categories of complexes}
%
% ---------------------------------------------------------------------

The trace of a linear category $\C$ gives rise to an interesting variant of Euler characteristic.
Let $\Kom(\C)$ be the category of bounded complexes over $\C$. Recall that the objects of $\Kom(\C)$ are complexes
$$
A_{\bullet}:\quad \quad \dots\longrightarrow A_i \xto{d_i} A_{i-1} \xto{d_{i-1}} \dots \xto{d_2} A_1 \xto{d_1} A_0 \longrightarrow 0,
$$
where $A_i  \in \Ob(\C)$ and
$d_i\in \C(A_i, A_{i-1})$, and the morphisms are chain maps.

The ``traced'' Euler characteristic of a complex $A_{\bullet}$ is
defined by
\[
\chi_{\tr}(A_{\bullet})=\sum_i (-1)^i [1_{A_i}]  \in \Tr(\C),
\]
where $[1_x]$ is the class of $1_x$ in $\Tr(\C)$ for any $x\in \Ob(\C)$.
More generally, to any chain endomorphism $F:A_{\bullet}\to A_{\bullet}$ we can associate its
``traced''  Lefschetz number
$$ \tau(F)=\sum_i (-1)^i [F_i]\, .$$ Then $\chi_{tr} (A_{\bullet})=\tau (1_{A_{\bullet}})$.

\begin{lem}  For any $A_{\bullet} \in \Kom(\C)$ and homotopic chain endomorphisms $F, G: A_{\bullet}\to A_{\bullet}$
we have $\tau(F)=\tau(G)$. Furthermore,
$\chi_{\tr}$ is  homotopy invariant.
\end{lem}

\begin{exercise}
Prove this lemma using the definition of the chain homotopy, i.e.
$F-G=dh+hd$.
\end{exercise}
\nc\Cob{{\rm Cob}}

An interesting application of the traced Euler characteristic was given
 by Bar-Natan~\cite{Dror}. There he defined a quotient $\Cob_/$ of the $\Z[1/2]$-linear category of 2-dimensional cobordisms,
and associated a complex in $\Kom(\Cob^{\oplus}_/)$ to any tangle, which he called a Khovanov bracket. He showed that isotopies of the tangle do not change the homotopy type of this complex.  The following theorem of Bar-Natan shows that the traced Euler characteristic can be used as decategorification map in a categorification of the Jones polynomial.

\begin{thm}[Theorem 6 \cite{Dror}]    \label{Dror}
The traced Euler characteristic of the Khovanov bracket
of a tangle in Bar-Natan's category  is equal to the Jones
polynomial.
\end{thm}

% ---------------------------------------------------------------------
%
\subsection{The trace and split Grothendieck groups}\label{sec:split-groth-groups}
%
% ---------------------------------------------------------------------

For a small additive category $\modC $, the {\em split Grothendieck group}
$K_0(\modC )$ of $\modC $ is the abelian group generated by the
isomorphism classes of objects of $\modC $ with relations $[x\oplus
y]_{\cong}=[x]_{\cong}+[y]_{\cong}$ for $x,y\in \Ob(\modC)$.  Here
$[x]_{\cong}$ denotes the isomorphism class of $x$.  The split
Grothendieck group $K_0$ defines a functor from additive categories to abelian groups. Here the morphisms in the category of additive categories are additive  functors.

\begin{example}
$K_0(\cat{Vect}_\k)=\Z$ is generated by $[\k]_{\cong}$.
\end{example}

More generally, for any commutative ring $R$, let $\FMod_R$ be the additive category of finitely generated free $R$-modules. Then $K_0(\FMod_R)=\Z$ with the isomorphism given by rank.
The split Grothendieck group $K_0$ of the additive category of finitely generated projective $R$-modules is generated by the indecomposable projective modules. Over a local ring, any finitely generated projective module is free, so up to isomorphism there is a unique indecomposable projective module, see for example~\cite[Chapter 1]{Milnor} or \cite{Ros}.
%
%\begin{exercise}\label{ex-7}
%Show that $K_0(\C^{\oplus})=K_0(\C))$.
%\end{exercise}

In $\Kom(\C)$ the Euler characteristic of a complex $A$  defined by
$$\chi(A_{\bullet})= \sum_i (-1)^i K_0(A_i)=
\sum_i (-1)^i [A_i]_{\cong}\, $$
is homotopy invariant.

%The natural embedding
%$i: \C\to\Kar(\C)$ induces an embedding of Grothendieck groups
%$K_0(\C) \to K_0(\Kar(\C))$, which  in general may not be an isomorphism since  indecomposable objects in $\C$ can decompose in $\Kar(\C)$.

% - - - - - - - - - - - - - - - - - - - - - - - - - - - - - - - - -
%
\subsubsection{Generalized Chern character}
%
% - - - - - - - - - - - - - - - - - - - - - - - - - - - - - - - - -

Classically, the Chern character is a homomorphism of rings from the topological $K$-theory of a manifold into  the de Rham cohomology of the manifold.  Given a vector bundle on a manifold $M$, the Chern character map produces an element in the de Rham cohomology of the manifold.  In the algebraic framework, a commutative associative algebra $A$ plays the role of the manifold $M$, and a vector bundle in topological $K$-theory is replaced by a finitely-generated projective module over $A$ in algebraic $K$-theory.  Passing to the non-commutative setting where $A$ is no longer assumed commutative, the Chern character map can still be defined, giving a homomorphism from $K_0(A)$ into the cyclic homology of $A$ (see chapter 8 of ~\cite{Loday}).  The zeroth cyclic homology $\mathrm{HC}_0(A)$ is isomorphic to the zeroth Hochschild homology $\HH_0(A)$~\cite[Section 2.1.12]{Loday}, so we can view the Chern character as a map from $K_0(A) \to \HH_0(A)$.
In this section we study a further generalization of the Chern character map.

Define a homomorphism
\begin{gather*}
  h_\modC \col K_0(\modC )\longrightarrow \Tr(\modC )
\end{gather*}
by
\begin{gather*}
  h_\modC ([x]_{\cong})=[1_x]
\end{gather*}
for $x\in \Ob(\modC )$.  Indeed, one can easily check that
\begin{gather*}
  h_\modC ([x\oplus y]_{\cong})=[1_x]+[1_y] \, ,
\end{gather*}
since $1_{f \oplus g} = 1_f \oplus 1_g$.
The map $h_\modC $ defines a natural transformation
\begin{gather*}
  h\col K_0 \Rightarrow  \Tr \col \AdCat\rightarrow \Ab,
\end{gather*}
where $\AdCat$ denotes the category of additive small categories.

\begin{example} \label{example_chern-char}
In Example 8.3.6 of \cite{Loday}, it is shown that for an associative unital algebra $A$ the map
$h \maps K_0(A) \to {\HH}_0(A)=A/[A,A]$ is an isomorphism whenever $A$ is a field, a local ring, or $\Z$.
\end{example}

We illustrate below that $h_\C$ is neither injective nor surjective in general.

\begin{example}
Let $F_p = \Z / p\Z$ denote the field with $p$ elements for $p$ prime.
If $\modC$ is the category of finite dimensional $F_p$-vector spaces, then
$ h_{\cal{C}}$  maps $K_0(\modC) = \Z$ surjectively onto $\Tr(\modC)=F_p$, but it is not injective.
\end{example}

Since $\Tr$ is always linear over the ground ring we are using, and $K_0$ is $\Z$-linear, one can wonder whether they are isomorphic after tensoring $K_0$ with the ground ring.
Computations for the nilHecke ring and its quotients illustrate that this  is not true in general.

\begin{example}  We showed earlier that $\Tr(\NH_n) = \Tr(\sym_n) = \sym_n$.  However, $K_0(\NH_n) \cong \Z$ so the map
\[
 h_{\NH_n} \maps  K_0(\NH_n) = \Z \longrightarrow \sym_n = \Tr(\NH_n)
\]
is clearly not surjective.
\end{example}

\begin{example}
For any integer $0 \leq k \leq n$ define cyclotomic quotients of the nilHecke algebra $\NH_n^k$ as the
quotient $\NH_n / (x_1)^k$.    Taken together over all $k$, these cyclotomic quotients give rise to categorifications of irreducible representations of
${\bf U}_q(\fsl_2)$~\cite{BK1,BK2,Lau3,LV,Web}. In ~\cite[Proposition 5.3]{Lau3} it is shown that
\[
\NH_n^k \cong \Mat( (n)^!_{q^2} ; H^{\ast}(\Gras(k,n)) ),
\]
where $H^{\ast}(\Gras(k,n))$ denotes the graded cohomology ring of the Grassmannian of $k$-dimensional planes in complex $n$-dimensional space.  This ring has an explicit description as the quotient of the graded polynomial ring $\Z[c_1, \dots, c_k; \bar{c}_1, \dots, \bar{c}_{N-k}]$ with $\deg c_j =2j$, $\deg \bar{c}_j = 2j$ by the homogeneous ideal generated by the elements $\sum c_j \bar{c}_{\alpha-j}$ for $\alpha>0$.  Note that we set $c_0 = \bar{c}_0=1$, $c_j=0$ for $j<0$ or $j>k$, and $\bar{c}_{\ell} =0$ for $\ell<0$ and $\ell > N-k$.

Since the cohomology ring $H^{\ast}(\Gras(k,n))$ is graded local, $K_0(H^{\ast}(\Gras(k,n))) \cong \Z$, but since $H^{\ast}(\Gras(k,n))$ is commutative,
by exercise \ref{mat} we have $\Tr(\NH_n^k) = \Tr(H^{\ast}(\Gras(k,n))) = H^{\ast}(\Gras(k,n))$, so that the homomorphism
\[
 h_{\NH_n^k} \maps K_0(\NH_n^k) \cong \Z \longrightarrow H^{\ast}(\Gras(k,n)) = \Tr(\NH_n^k)
\]
is also not surjective.
\end{example}

% ---------------------------------------------------------------------
%
\subsection{The trace of a graded linear category}
%
% ---------------------------------------------------------------------

A {\em graded} linear category is a category equipped with an auto-equivalence $\la 1 \ra$. We denote by $\la t \ra$ the auto-equivalence obtained by applying $\la 1 \ra$ $t$ times.  Thus, given any object $x\in \Ob(\modC)$, there exists an object $x \la t \ra \in \Ob(\modC)$ for all $t$.  Likewise, given any morphism $f \maps x \to y$ in $\modC$, there is a morphism $f \la t \ra \maps x \la t \ra \to y \la t \ra$.  This equips Hom-sets with the structure of  $\Z$-graded abelian groups in which composition maps are grading preserving.

Given a graded linear category $\modC$, its Grothendieck group $K_0(\modC)$ is a $\Z[q,q^{-1}]$-module with $[x \la t \ra]_{\cong} := q^t[x]$.  Likewise, the trace $\Tr(\modC)$ is $\Z[q,q^{-1}]$-module with $[f\la t \ra] = q^t [f]$ for an endomorphism $f \maps x \to x$ in $\modC$.    In this case, the generalized Chern character map
\[
 K_0(\modC) \to \Tr(\modC)
\]
is a homomorphism of $\Z[q,q^{-1}]$-modules.

A graded linear category $\modC$ is said to {\em admit a translation} if for every object $x$ in $\modC$ and integer $t\in \Z$, there is an isomorphism  $x \to x\la t \ra$.  Given any graded linear category $\modC$, we can form a graded linear category $\modC^{\ast}$ that admits translations by declaring that $\Ob(\C)=\Ob(\C^{\ast})$ and setting
\[
 \modC^{\ast}(x,y) := \bigoplus_{t \in \Z} \modC(x,y\la t\ra),
\]
for all objects $x,y \in \Ob(\modC)$.  The translations $x \to x \la t \ra$ is given by $1_x$ in $\modC^{\ast}$ since
\[
 1_x \in \modC(x,x\la 0 \ra) = \modC\left(x,\left(x\la t \ra\right) \la -t\ra\right) \subset \modC^*(x,x\la t \ra),
\]
has inverse
\[
 1_x\la t\ra \in \modC(x \la t \ra, x \la t \ra) = \modC(x \la t \ra, (x \la 0 \ra) \la t \ra)   \subset \modC^{\ast}(x\la t\ra, x).
\]

By enlarging the hom spaces by adding isomorphisms $x \to x \la t\ra$, the Grothendieck group  of graded linear category with translation $\modC^{\ast}$ is only a $\Z$-module rather than a $\Z[q,q^{-1}]$-module since $[x\la t \ra]_{\cong} = [x]_{\cong}$.   However, the trace $\Tr(\modC^{\ast})$ changes more drastically because, rather than only considering degree preserving endomorphisms $f\maps x \to x$ in $\modC$, we can now consider arbitrary degree maps $x \to x\la t \ra$ and compose with the translation map from $x\la t \ra \to x$ to get a new endomorphism.  Furthermore, using the translations $x \cong x \la t \ra$, it is not hard to see that given any endomorphism $f\maps x \to x$, we have $[f\la t \ra] = [f]$ in $\Tr(\modC^{\ast})$, so that the graded version of the trace $\Tr(\modC^{\ast})$ is also just a $\Z$-module.

As explained at the beginning of section~\ref{sec_HH0trace}, any ring with unit can be viewed as a linear category
with one object $\ast$.  A graded ring $R = \oplus_{n \in \Z} R_n$ can be viewed as a graded category whose objects are indexed by the integers $n = \ast \la n \ra$ for $n \in \Z$.  The morphisms from $n$ to $m$ are given by the abelian group $R_{m-n}$.   Composition is given by the grading preserving multiplication in $R$.  Example \ref{example_chern-char} can be generalized to show that the generalized Chern character map  $K_0(R) \to \Tr(R_0)$ is an isomorphism whenever $R$ is a graded local ring, so that the trace of the degree zero part of the ring is isomorphic to the Grothendieck ring.

\begin{exercise} \label{ex_equiv}
If $\modC$ denotes the graded linear category associated to a graded ring $R$, then the graded linear category with translation $\modC^{\ast}$ is equivalent as a category to the linear category with one object $\ast$ obtained by forgetting the grading on $R$.
\end{exercise}

\begin{example}
The nilHecke algebra $\NH_n$ is a graded algebra with $\deg(x_i)=2$ and $\deg(\partial_i)=-2$.  Recall that the nilHecke ring $\NH_n$ is isomorphic as a graded algebra to the matrix ring ${\Mat}((n)^!_{q^2} ; \sym_n)$ of $n! \times n!$ matrices with coefficients in the ring $\sym_n=\Z[x_1,\dots,x_n]^{S_n}$ of symmetric functions in $n$ variables.  Here we write $(n)^!_{q^2}:= (1-q^{2n})/(1-q^2)$ for the non-symmetrized quantum integers and each variable of $\sym_n$ has degree $2$.

To compute the trace of $\NH_n$ regarded as a graded category $\modC$, observe that $\Tr(\modC) = \Tr(\sym)$ and that $\sym$ is positively graded and one dimensional in degree zero, in particular it is graded local, so that the graded trace of $\Tr(\modC) \cong K_0(\NH_n) \cong \Z$.  Compare this result with section~\ref{sample} where we showed $\Tr(\NH_n) \cong \sym_n$.   Exercise~\ref{ex_equiv} shows that by enlarging the graded category $\modC$ to $\modC^{\ast}$, we again have $\Tr(\modC^{\ast}) = \sym_n$.
\end{example}

% ---------------------------------------------------------------------
%
\subsection{Hochschild--Mitchell homology of strongly upper triangular categories}
\label{sec:upper-triang-line}
%
% ---------------------------------------------------------------------
Recall that the trace of a linear category is just the zeroth Hochschild--Mitchell homology.
There is a class of linear categories for which we can compute explicitly their Hochschild--Mitchell homology.  The results developed in this section will be useful in our study of the Hochschild--Mitchell homology of categorified quantum groups.

\begin{defn}
  \label{r22}
  A linear category $\modC $ is said to be {\em upper-triangular} if
there is no sequence
\begin{gather*}
  x_0\xto{f_0}x_1\xto{f_1}\dots \xto{f_{n-1}}x_n\xto{f_n}x_0 \quad (n\ge 1)
\end{gather*}
of nonzero morphisms in $\modC $ unless $x_0=x_1=\dots =x_n$.
Alternatively, a linear category $\modC $ is upper-triangular if there is
a partial order $\le$ on $\Ob(\modC )$ such that $\modC (x,y)\neq0$ implies
$x\le y$.
\end{defn}

For a linear category $\modC $, define a ring $\End(\modC )$ by
\begin{gather*}
  \End(\modC )=\bigoplus_{x\in \Ob(\modC )}\modC (x,x).
\end{gather*}
%\KH{
%It would be better to mention that $End(C)$ is not necessarily unital. In Loday's book, the Hochschild homology of a non-unital ring is defined in  a certain way different from the unital case.
% (Perhaps, I do not remember correctly.)
% The definition in Section 2.2 of the Hochschild homology, applied to
% non-unital rings, is called the "naive Hochschild homology".
% If a non-unital ring is called "H-unital" if the Hochschild homology and the
% naive Hochschild homology are isomorphic.
% I believe that (infinite) direct sum of $H$-unital ring is again $H$-unital.
% Since $End(C)$ is a direct sum of unital rings, it follows that $End(C)$ is
% $H$-unital, i.e. $HH_n^{naive}(End(C))=HH_n(End(C))$.
% If we use $HH_n(End(C))$, it would be better to mention something related to
% this comments.  (For the $sl_2$ paper, I removed $End(C)$ just for simplicity.  Perhaps, should
% we recover it?)
% }

\begin{lem}[\cite{BHLZ}]
  \label{r19}  For an upper-triangular linear category $\modC $, we have
  \begin{gather}
    \label{e5}
    \HH_*(\modC )
    \cong\bigoplus_{x\in \Ob(\modC )}\HH_*(\modC (x,x))
    \cong \HH_*(\End(\modC )).
  \end{gather}
\end{lem}

\begin{remark} Note that the ring End(C) is not unital if Ob(C) is infinite.
The definition of the Hochschild homology of a non-unital ring R is
different from the case of unital rings (see \cite{Loday}), but it is well known
that the Hochschild homology of the direct sum of unital rings is
isomorphic to the direct sum of the Hochschild homology of its direct
summands.
\end{remark}

\begin{defn}
  \label{r23}
A {\em strongly upper-triangular} linear category is an
upper-triangular linear category $\modC $ such that for all $x\in \Ob(\modC )$,
we have $\modC (x,x)\cong \modZ $.
\end{defn}

Using that $\HH_0(\Z)=\Z$ and for $n>0$, $\HH_n(\Z)=0$,
from the previous lemma we immediately get the following.

\begin{cor}[\cite{BHLZ}]
  \label{r20}
  For a strongly upper-triangular linear category $\modC $, we have
  \begin{gather*}
    \Tr\modC =\HH_0(\modC ) \cong \End(\modC )\cong\modZ \Ob(\modC ),\\
    \HH_i(\modC )=0\quad \text{for $i>0$}.
  \end{gather*}
  Here $\modZ \Ob(\modC )$ denotes the free $\modZ $-module spanned by $\Ob(\modC )$.
\end{cor}

Let $\modC $ be an additive category, and let $B\subset \Ob(\modC )$ be a subset.
Denote by $\modC |_B$ the full subcategory of $\modC $ with
$\Ob(\modC |_B)=B$.

The set $B$ is called a {\em strongly upper-triangular basis} of
$\modC $ if the following two conditions hold.
\begin{enumerate}
\item The inclusion functor $\modC |_B\rightarrow \modC $ induces equivalence of additive
  categories $(\modC |_B)^\oplus\simeq \modC $.
\item $\modC |_B$ is strongly upper-triangular.
\end{enumerate}

\begin{prop}[\cite{BHLZ}]
  \label{r24}
  Let $\modC $ be an additive category with a strongly upper-triangular
  basis $B$.  Then we have
  \begin{gather*}
    \Tr\modC =\HH_0(\modC )\cong \modZ B,\\
    \HH_i(\modC )=0\quad \text{for $i>0$}.
  \end{gather*}
\end{prop}

\begin{remark} \label{rem-overk}
For $\k$-linear categories it is natural to modify the Definition~\ref{r23} to define a strongly upper-triangular $\k$-linear category by requiring that for all $x\in \Ob(\modC )$,
we have $\modC (x,x)\cong \k$.  In that case, Corollary~\ref{r20} and Proposition ~\ref{r24} remain valid with $\Z$ replaced by the field $\k$.
\end{remark}

An additive category is said to be {\em Krull-Schmidt} if any object has a unique decomposition into a finite direct sum of objects having local endomorphism rings.  That is to say, given an isomorphism $x_1 \oplus x_2 \oplus \dots \oplus x_r \cong y_1 \oplus y_2 \oplus \dots \oplus y_s$ where $x_i$ and $y_j$ are indecomposable, then $r=s$, and there exists some permutation $\sigma \in S_r$ such that $x_{\sigma(i)} \simeq y_i$.

\begin{exercise}
Show that any category $\C$ enriched in $\cat{Vect}_{\k}$, i.e. the hom spaces form vector spaces and composition is a bilinear map, is Krull-Schmidt if and only if all idempotents split.
\end{exercise}

In a Krull-Schmidt category  $\C$ the classes of indecomposables in the split Grothendieck group $K_0(\C)$ give rise to a distinguished basis.
When $\C$ is a graded linear category, then the graded hom between a pair of objects $x$, $y \in \Ob(\C)$ is defined as
\[
 \HOM_{\C}(x,y) := \bigoplus_{k \in \Z} \Hom_{\C}(x, y \la k \ra).
\]
When $\C$ is a graded $\k$-linear category, this graded hom gives rise to a sesquilinear inner product $\la - , - \ra \maps K_0(\C) \times K_0(\C) \to \Z((q))$ defined by
\[
 \la [x]_{\cong}, [y]_{\cong} \ra := \dim_q \HOM_{\C}(x,y) = \sum_{k \in \Z} q^k \dim \Hom_{\C}(x, y \la k\ra).
\]

One consequence of a graded linear category possessing a
strongly upper triangular basis $B$ is that the classes of $B$ in the Grothendieck group satisfy the property that for any pair $x, y \in B$, we have
\[
 \la x, y \ra \in \delta_{x,y} + q\N[q],
\]
in which case the basis resulting from $B$ is said to be {\em almost orthogonal}.

\vspace*{3mm}
\noindent
{\bf Remarks for experts}.
Strongly upper triangular bases often arise in geometric constructions where the categories of interest are categories of sheaves on certain varieties.
Usually one starts with a triangulated category $\C$ and produces an abelian category $\cal{A}$ as the heart of a $t$-structure for $\C$. Distinguished basis  $B$ arise as the simple objects in the heart of the corresponding $t$-structure.
 If $\cal{A}$ is endowed with a grading given by the cohomological ${\rm Ext}$-grading from $\C$, then the resulting basis of simple objects
 is strongly upper triangular and their classes in the Grothendieck group are almost orthogonal.

For example, categories of perverse sheaves are the heart of a certain $t$-structure on the derived category of sheaves~\cite{BBD}.
 In this way, simple perverse sheaves naturally give rise to distinguished bases in geometric representation theory.  The Kazhdan-Lusztig
 basis of a Hecke algebra associated to a Weyl group $W$ arise in this way~\cite{KazLus}.  Likewise, simple integrable representations and
 their tensor products of quantum Kac-Moody algebra  possess canonical bases~\cite{Lus1,Lus2,Lus3,Lus4,Kash1,Kash2}. Outside of geometry,
there is a $t$-structure on the bounded derived category $D^b(\Lambda{\rm -mod})$ of a finite-dimensional $\k$-algebra $\Lambda$ in which
the simple $\Lambda$-modules form a strongly upper triangular basis in the heart of the $t$-structure~\cite{AN}.

% ==============================================================================
%
\section{Categorified  quantum groups}\label{sln}
%
% ==============================================================================
Fix a base field $\Bbbk$. We will always work over this field which is not assumed to be of characteristic 0, nor algebraically closed.

% ---------------------------------------------------------------------
%
\subsection{Quantum $\mf{sl}_n$}
%
% ---------------------------------------------------------------------
% - - - - - - - - - - - - - - - - - - - - - - - - - - - - - - - - -
%
\subsubsection{The Cartan datum }\label{sec:datum}
%
% - - - - - - - - - - - - - - - - - - - - - - - - - - - - - - - - -

Let $I=\{1,2,\dots,n-1\}$ consist of the set of vertices of the Dynkin diagram of type $A_{n-1}$
\begin{equation} \label{eq_dynkin_sln}
    \xy
  (-15,0)*{\circ}="1";
  (-5, 0)*{\circ}="2";
  (5,  0)*{\circ}="3";
  (35,  0)*{\circ}="4";
  "1";"2" **\dir{-}?(.55)*\dir{};
  "2";"3" **\dir{-}?(.55)*\dir{};
  "3";(15,0) **\dir{-}?(.55)*\dir{};
  (25,0);"4" **\dir{-}?(.55)*\dir{};
  (-15,2.2)*{\scs 1};
  (-5,2.2)*{\scs 2};
  (5,2.2)*{\scs 3};
  (35,2.2)*{\scs n-1};
  (20,0)*{\cdots };
  \endxy
  \nn
\end{equation}
enumerated from left to right. Let $X=\Z^{n-1}$ denote the weight lattice for $\mathfrak{sl}_n$ and
$\{\alpha_i\}_{i\ \in I} \subset X$ and $\{\Lambda_i\}_{i \in I} \subset X$ denote the collection of simple roots
and fundamental weights, respectively.
There is a symmetric bilinear form on $X$ defined by $(\alpha_i, \alpha_j) = a_{ij}$ where
\[ a_{ij} =
\left\{
\begin{array}{ll}
  2 & \text{if $i=j$}\\
  -1&  \text{if $|i-j|=1$} \\
  0 & \text{if $|i-j|>1$}
\end{array}
\right.
\]
is the (symmetric) Cartan matrix associated to $\mathfrak{sl}_n$. For $i \in I$ denote the simple coroots by
$h_i \in X^\vee = \Hom_{\Z}(X,\Z)$.  Write $\langle \cdot, \cdot \rangle \maps X^{\vee} \times X
\to \Z$ for the canonical pairing
$\la i,\lambda\ra :=\langle h_i, \lambda \rangle = 2 \frac{(\alpha_i,\lambda)}{(\alpha_i,\alpha_i)}$
for $i \in I$ and $\lambda \in X$ that satisfies $\la h_i, \Lambda_i \ra = \delta_{i,j}$. Any weight $\lambda \in X$
can be written as $\lambda = (\lambda_1,\lambda_2, \dots, \lambda_{n-1})$, where
$\lambda_i=\la h_i, \lambda \ra$.

We let $X^+ \subset X$ denote the dominant weights, which are those of the form
$\sum_i \lambda_i \Lambda_i$
with $\lambda_i \ge 0$.  Finally, let $[k]=\frac{q^k-q^{-k}}{q-q^{-1}}$ and $[k]!=[k][k-1]\dots [1]$.

% - - - - - - - - - - - - - - - - - - - - - - - - - - - - - - - - -
%
\subsubsection{The algebra $\mathbf{U}_q(\mathfrak{sl}_n)$} \label{subsubsec_algebrasln}
%
% - - - - - - - - - - - - - - - - - - - - - - - - - - - - - - - - -

The algebra $\mathbf{U}_q(\mathfrak{sl}_n)$ is the $\Q(q)$-algebra with unit generated by the elements $E_i$,
$F_i$ and $K_i^{\pm 1}$ for $i = 1, 2, \dots , n-1$, with  the defining relations
\begin{equation}
K_iK_i^{-1} = K_i^{-1}K_i = 1, \quad  K_iK_j = K_jK_i,
\end{equation}
\begin{equation}
K_iE_jK_i^{-1} = q^{a_{ij}} E_j, \quad  K_iF_jK_i^{-1} = q^{-a_{ij}} F_j,
\end{equation}
\begin{equation}
E_iF_j - F_jE_i = \delta_{ij} \frac{K_i-K_{i}^{-1}}{q-q^{-1}},
\end{equation}
\begin{equation}
E_i^2E_j-(q+q^{-1})E_iE_jE_i+E_jE_i^2 =0\;\; \text{if $j=i\pm 1$},
\end{equation}\begin{equation}
F_i^2F_j-(q+q^{-1})F_iF_jF_i+F_jF_i^2=0 \;\; \text{if $j=i\pm 1$},
\end{equation}\begin{equation}
E_iE_j=E_jE_i, \quad F_iF_j = F_jF_i \;\; \text{if $|i-j|>1$.}
\end{equation}

Let $\dot{\mathbf{U}}(\mathfrak{sl}_n)$ be the idempotented version of $\mathbf{U}_q(\mathfrak{sl}_n)$
where the unit is replaced by a collection of orthogonal idempotents $1_{\lambda}$ indexed by the weight lattice $X$ of $\mathfrak{sl}_n$,
\begin{equation}
  1_{\lambda}1_{\lambda'} = \delta_{\lambda\lambda'} 1_{\lambda},
\end{equation}
such that if $\lambda = (\lambda_1,\lambda_2, \dots, \lambda_{m-1})$, then
\begin{equation} \label{eq_onesubn}
K_i1_{\lambda} =1_{\lambda}K_i= q^{\lambda_i} 1_{\lambda}, \quad
E_i^{}1_{\lambda} = 1_{\lambda+\alpha_i}E_i, \quad F_i1_{\lambda} = 1_{\lambda-\alpha_i}F_i
,
\end{equation}
where
\begin{eqnarray} \label{eq_weight_action1}
 \lambda +\alpha_i &=& \quad \left\{
\begin{array}{ccl}
   (\lambda_1+2, \lambda_2-1,\lambda_3,\dots,\lambda_{m-2}, \lambda_{m-1}) & \quad & \text{if $i=1$} \\
   (\lambda_1, \lambda_2,\dots,\lambda_{m-2},\lambda_{m-1}-1, \lambda_{m-1}+2) & \quad & \text{if $i=n-1$} \\
  (\lambda_1, \dots, \lambda_{i-1}-1, \lambda_i+2, \lambda_{i+1}-1, \dots,
\lambda_{m-1}) & \quad & \text{otherwise.}
\end{array}
 \right.
\end{eqnarray}
$\dot{\mathbf{U}}(\mathfrak{sl}_n)$ can be viewed as a category with objects $\l\in X$ and morphisms
given  by compositions of $E_i$ and $F_i$ with $1\leq i <n$ modulo the above relations.

Let $\cal{A} := \Z[q,q^{-1}]$; the $\cal{A}$-algebra $\UA(\mathfrak{sl}_n)$ is the integral form of
$\dot{\mathbf{U}}(\mathfrak{sl}_n)$ generated by products of divided powers $E^{(a)}_i1_{\lambda}:=
\frac{E^{a}_i}{[a]!}1_{\lambda}$, $F^{(a)}_i1_{\lambda}:=
\frac{F^{a}_i}{[a]!}1_{\lambda}$ for $\lambda \in X$ and $i = 1, 2, \dots , n-1$.

% ---------------------------------------------------------------------
%
\subsection{The 2-category $\Ucat_Q(\mathfrak{sl}_n)$}
%
% ---------------------------------------------------------------------

Here we describe a categorification of $\mathbf{U}(\mathfrak{sl}_n)$ mainly following \cite{CLau} and \cite{KL3}.
For an elementary introduction to the categorification of $\mf{sl}_2$ see \cite{Lau3}.

% - - - - - - - - - - - - - - - - - - - - - - - - - - - - - - - - -
%
\subsubsection{Choice of scalars $Q$} \label{subsec-choice}
%
% - - - - - - - - - - - - - - - - - - - - - - - - - - - - - - - - -

Associated to the Cartan datum for $\mathfrak{sl}_n$ we also fix a choice of scalars $Q$ consisting of:
\begin{itemize}
  \item $t_{ij}$ for all $i,j \in I$,
\end{itemize}
such that
\begin{itemize}
\item $t_{ii}=1$ for all $i \in I$ and $t_{ij} \in \Bbbk^{\times}$ for $i\neq j$,
 \item $t_{ij}=t_{ji}$ when $a_{ij}=0$.
\end{itemize}

% - - - - - - - - - - - - - - - - - - - - - - - - - - - - - - - - -
%
\subsubsection{The definition }
%
% - - - - - - - - - - - - - - - - - - - - - - - - - - - - - - - - -

By a graded linear 2-category we mean a category enriched in graded linear categories, so that the hom spaces form graded linear categories, and the composition map is grading preserving.

Given a fixed choice of scalars $Q$ we can define the following 2-category.

\begin{defn} \label{defU_cat}
The 2-category $\Ucat_Q(\mathfrak{sl}_n)$ is the graded linear
2-category consisting of:
\begin{itemize}
\item objects $\lambda$ for $\lambda \in X$.
\item 1-morphisms are formal direct sums of (shifts of) compositions of
$$\onel, \quad \onenn{\l+\alpha_i} \sE_i = \onenn{\l+\alpha_i} \sE_i\onel = \sE_i \onel, \quad \text{ and }\quad \onenn{\lambda-\alpha_i} \sF_i = \onenn{\lambda-\alpha_i} \sF_i\onel = \sF_i\onel$$
for $i \in I$ and $\l \in X$.
\item 2-morphisms are $\Bbbk$-vector spaces spanned by compositions of (decorated) tangle-like diagrams illustrated below.
\begin{align}
  \xy 0;/r.17pc/:
 (0,7);(0,-7); **\dir{-} ?(.75)*\dir{>};
 (0,0)*{\bullet};
 (7,3)*{ \scs \lambda};
 (-9,3)*{\scs  \lambda+\alpha_i};
 (-2.5,-6)*{\scs i};
 (-10,0)*{};(10,0)*{};
 \endxy &\maps \cal{E}_i\onel \to \cal{E}_i\onel\la (\alpha_i,\alpha_i) \ra  & \quad
 &
    \xy 0;/r.17pc/:
 (0,7);(0,-7); **\dir{-} ?(.75)*\dir{<};
 (0,0)*{\bullet};
 (7,3)*{ \scs \lambda};
 (-9,3)*{\scs  \lambda-\alpha_i};
 (-2.5,-6)*{\scs i};
 (-10,0)*{};(10,0)*{};
 \endxy\maps \cal{F}_i\onel \to \cal{F}_i\onel\la (\alpha_i,\alpha_i) \ra  \nn \\
   & & & \nn \\
   \xy 0;/r.17pc/:
  (0,0)*{\xybox{
    (-4,-4)*{};(4,4)*{} **\crv{(-4,-1) & (4,1)}?(1)*\dir{>} ;
    (4,-4)*{};(-4,4)*{} **\crv{(4,-1) & (-4,1)}?(1)*\dir{>};
    (-5.5,-3)*{\scs i};
     (5.5,-3)*{\scs j};
     (9,1)*{\scs  \lambda};
     (-10,0)*{};(10,0)*{};
     }};
  \endxy \;\;&\maps \cal{E}_i\cal{E}_j\onel  \to \cal{E}_j\cal{E}_i\onel\la - (\alpha_i,\alpha_j) \ra  &
  &
   \xy 0;/r.17pc/:
  (0,0)*{\xybox{
    (-4,4)*{};(4,-4)*{} **\crv{(-4,1) & (4,-1)}?(1)*\dir{>} ;
    (4,4)*{};(-4,-4)*{} **\crv{(4,1) & (-4,-1)}?(1)*\dir{>};
    (-6.5,-3)*{\scs i};
     (6.5,-3)*{\scs j};
     (9,1)*{\scs  \lambda};
     (-10,0)*{};(10,0)*{};
     }};
  \endxy\;\; \maps \cal{F}_i\cal{F}_j\onel  \to \cal{F}_j\cal{F}_i\onel\la - (\alpha_i,\alpha_j) \ra  \nn \\
  & & & \nn \\
     \xy 0;/r.17pc/:
    (0,-3)*{\bbpef{i}};
    (8,-5)*{\scs  \lambda};
    (-10,0)*{};(10,0)*{};
    \endxy &\maps \onel  \to \cal{F}_i\cal{E}_i\onel\la 1 + (\l, \alpha_i) \ra   &
    &
   \xy 0;/r.17pc/:
    (0,-3)*{\bbpfe{i}};
    (8,-5)*{\scs \lambda};
    (-10,0)*{};(10,0)*{};
    \endxy \maps \onel  \to\cal{E}_i\cal{F}_i\onel\la 1 - (\l, \alpha_i) \ra  \nn \\
      & & & \nn \\
  \xy 0;/r.17pc/:
    (0,0)*{\bbcef{i}};
    (8,4)*{\scs  \lambda};
    (-10,0)*{};(10,0)*{};
    \endxy & \maps \cal{F}_i\cal{E}_i\onel \to\onel\la 1 + (\l, \alpha_i) \ra  &
    &
 \xy 0;/r.17pc/:
    (0,0)*{\bbcfe{i}};
    (8,4)*{\scs  \lambda};
    (-10,0)*{};(10,0)*{};
    \endxy \maps\cal{E}_i\cal{F}_i\onel  \to\onel\la 1 - (\l, \alpha_i) \ra \nn
\end{align}
\end{itemize}
\end{defn}
Here we follow the grading conventions in \cite{CLau} and \cite{LQR}
which are opposite to those from \cite{KL3}.
In this $2$-category (and those throughout the paper) we
read diagrams from right to left and bottom to top.  The identity 2-morphism of the 1-morphism
$\cal{E}_i \onel$ is
represented by an upward oriented line labeled by $i$ and the identity 2-morphism of $\cal{F}_i \onel$ is
represented by a downward such line.

The 2-morphisms satisfy the following relations:
\begin{enumerate}
\item \label{item_cycbiadjoint} The 1-morphisms $\cal{E}_i \onel$ and $\cal{F}_i \onel$ are biadjoint (up to a specified degree shift). These conditions are expressed diagrammatically as
    \begin{equation} \label{eq_biadjoint1}
 \xy   0;/r.17pc/:
    (-8,0)*{}="1";
    (0,0)*{}="2";
    (8,0)*{}="3";
    (-8,-10);"1" **\dir{-};
    "1";"2" **\crv{(-8,8) & (0,8)} ?(0)*\dir{>} ?(1)*\dir{>};
    "2";"3" **\crv{(0,-8) & (8,-8)}?(1)*\dir{>};
    "3"; (8,10) **\dir{-};
    (12,-9)*{\lambda};
    (-6,9)*{\lambda+\alpha_i};
    \endxy
    \; =
    \;
\xy   0;/r.17pc/:
    (-8,0)*{}="1";
    (0,0)*{}="2";
    (8,0)*{}="3";
    (0,-10);(0,10)**\dir{-} ?(.5)*\dir{>};
    (5,8)*{\lambda};
    (-9,8)*{\lambda+\alpha_i};
    \endxy
\qquad \quad \xy  0;/r.17pc/:
    (8,0)*{}="1";
    (0,0)*{}="2";
    (-8,0)*{}="3";
    (8,-10);"1" **\dir{-};
    "1";"2" **\crv{(8,8) & (0,8)} ?(0)*\dir{<} ?(1)*\dir{<};
    "2";"3" **\crv{(0,-8) & (-8,-8)}?(1)*\dir{<};
    "3"; (-8,10) **\dir{-};
    (12,9)*{\lambda+\alpha_i};
    (-6,-9)*{\lambda};
    \endxy
    \; =
    \;
\xy  0;/r.17pc/:
    (8,0)*{}="1";
    (0,0)*{}="2";
    (-8,0)*{}="3";
    (0,-10);(0,10)**\dir{-} ?(.5)*\dir{<};
    (9,-8)*{\lambda+\alpha_i};
    (-6,-8)*{\lambda};
    \endxy
\end{equation}

\begin{equation}\label{eq_biadjoint2}
 \xy   0;/r.17pc/:
    (8,0)*{}="1";
    (0,0)*{}="2";
    (-8,0)*{}="3";
    (8,-10);"1" **\dir{-};
    "1";"2" **\crv{(8,8) & (0,8)} ?(0)*\dir{>} ?(1)*\dir{>};
    "2";"3" **\crv{(0,-8) & (-8,-8)}?(1)*\dir{>};
    "3"; (-8,10) **\dir{-};
    (12,9)*{\lambda};
    (-5,-9)*{\lambda+\alpha_i};
    \endxy
    \; =
    \;
    \xy 0;/r.17pc/:
    (8,0)*{}="1";
    (0,0)*{}="2";
    (-8,0)*{}="3";
    (0,-10);(0,10)**\dir{-} ?(.5)*\dir{>};
    (5,-8)*{\lambda};
    (-9,-8)*{\lambda+\alpha_i};
    \endxy
\qquad \quad \xy   0;/r.17pc/:
    (-8,0)*{}="1";
    (0,0)*{}="2";
    (8,0)*{}="3";
    (-8,-10);"1" **\dir{-};
    "1";"2" **\crv{(-8,8) & (0,8)} ?(0)*\dir{<} ?(1)*\dir{<};
    "2";"3" **\crv{(0,-8) & (8,-8)}?(1)*\dir{<};
    "3"; (8,10) **\dir{-};
    (12,-9)*{\lambda+\alpha_i};
    (-6,9)*{\lambda};
    \endxy
    \; =
    \;
\xy   0;/r.17pc/:
    (-8,0)*{}="1";
    (0,0)*{}="2";
    (8,0)*{}="3";
    (0,-10);(0,10)**\dir{-} ?(.5)*\dir{<};
   (9,8)*{\lambda+\alpha_i};
    (-6,8)*{\lambda};
    \endxy
\end{equation}

  \item The 2-morphisms are $Q$-cyclic with respect to this biadjoint structure.
\begin{equation} \label{eq_cyclic_dot}
 \xy 0;/r.17pc/:
    (-8,5)*{}="1";
    (0,5)*{}="2";
    (0,-5)*{}="2'";
    (8,-5)*{}="3";
    (-8,-10);"1" **\dir{-};
    "2";"2'" **\dir{-} ?(.5)*\dir{<};
    "1";"2" **\crv{(-8,12) & (0,12)} ?(0)*\dir{<};
    "2'";"3" **\crv{(0,-12) & (8,-12)}?(1)*\dir{<};
    "3"; (8,10) **\dir{-};
    (17,-9)*{\lambda+\alpha_i};
    (-12,9)*{\lambda};
    (0,4)*{\bullet};
    (10,8)*{\scs };
    (-10,-8)*{\scs };
    \endxy
    \quad = \quad
      \xy 0;/r.17pc/:
 (0,10);(0,-10); **\dir{-} ?(.75)*\dir{<}+(2.3,0)*{\scriptstyle{}}
 ?(.1)*\dir{ }+(2,0)*{\scs };
 (0,0)*{\bullet};
 (-6,5)*{\lambda};
 (10,5)*{\lambda+\alpha_i};
 (-10,0)*{};(10,0)*{};(-2,-8)*{\scs };
 \endxy
    \quad = \quad
   \xy 0;/r.17pc/:
    (8,5)*{}="1";
    (0,5)*{}="2";
    (0,-5)*{}="2'";
    (-8,-5)*{}="3";
    (8,-10);"1" **\dir{-};
    "2";"2'" **\dir{-} ?(.5)*\dir{<};
    "1";"2" **\crv{(8,12) & (0,12)} ?(0)*\dir{<};
    "2'";"3" **\crv{(0,-12) & (-8,-12)}?(1)*\dir{<};
    "3"; (-8,10) **\dir{-};
    (17,9)*{\lambda+\alpha_i};
    (-12,-9)*{\lambda};
    (0,4)*{\bullet};
    (-10,8)*{\scs };
    (10,-8)*{\scs };
    \endxy
\end{equation}
The $Q$-cyclic relations for crossings are given by
\begin{equation} \label{eq_almost_cyclic}
   \xy 0;/r.17pc/:
  (0,0)*{\xybox{
    (-4,4)*{};(4,-4)*{} **\crv{(-4,1) & (4,-1)}?(1)*\dir{>} ;
    (4,4)*{};(-4,-4)*{} **\crv{(4,1) & (-4,-1)}?(1)*\dir{>};
    (-6.5,-3)*{\scs i};
     (6.5,-3)*{\scs j};
     (9,1)*{\scs  \lambda};
     (-10,0)*{};(10,0)*{};
     }};
  \endxy \quad = \quad
  t_{ij}^{-1}\xy 0;/r.17pc/:
  (0,0)*{\xybox{
    (4,-4)*{};(-4,4)*{} **\crv{(4,-1) & (-4,1)}?(1)*\dir{>};
    (-4,-4)*{};(4,4)*{} **\crv{(-4,-1) & (4,1)};
     (-4,4)*{};(18,4)*{} **\crv{(-4,16) & (18,16)} ?(1)*\dir{>};
     (4,-4)*{};(-18,-4)*{} **\crv{(4,-16) & (-18,-16)} ?(1)*\dir{<}?(0)*\dir{<};
     (-18,-4);(-18,12) **\dir{-};(-12,-4);(-12,12) **\dir{-};
     (18,4);(18,-12) **\dir{-};(12,4);(12,-12) **\dir{-};
     (8,1)*{ \lambda};
     (-10,0)*{};(10,0)*{};
     (-4,-4)*{};(-12,-4)*{} **\crv{(-4,-10) & (-12,-10)}?(1)*\dir{<}?(0)*\dir{<};
      (4,4)*{};(12,4)*{} **\crv{(4,10) & (12,10)}?(1)*\dir{>}?(0)*\dir{>};
      (-20,11)*{\scs j};(-10,11)*{\scs i};
      (20,-11)*{\scs j};(10,-11)*{\scs i};
     }};
  \endxy
\quad =  \quad t_{ji}^{-1}
\xy 0;/r.17pc/:
  (0,0)*{\xybox{
    (-4,-4)*{};(4,4)*{} **\crv{(-4,-1) & (4,1)}?(1)*\dir{>};
    (4,-4)*{};(-4,4)*{} **\crv{(4,-1) & (-4,1)};
     (4,4)*{};(-18,4)*{} **\crv{(4,16) & (-18,16)} ?(1)*\dir{>};
     (-4,-4)*{};(18,-4)*{} **\crv{(-4,-16) & (18,-16)} ?(1)*\dir{<}?(0)*\dir{<};
     (18,-4);(18,12) **\dir{-};(12,-4);(12,12) **\dir{-};
     (-18,4);(-18,-12) **\dir{-};(-12,4);(-12,-12) **\dir{-};
     (8,1)*{ \lambda};
     (-10,0)*{};(10,0)*{};
      (4,-4)*{};(12,-4)*{} **\crv{(4,-10) & (12,-10)}?(1)*\dir{<}?(0)*\dir{<};
      (-4,4)*{};(-12,4)*{} **\crv{(-4,10) & (-12,10)}?(1)*\dir{>}?(0)*\dir{>};
      (20,11)*{\scs i};(10,11)*{\scs j};
      (-20,-11)*{\scs i};(-10,-11)*{\scs j};
     }};
  \endxy
\end{equation}

Sideways crossings can then be defined utilizing the $Q$-cyclic condition by the equalities:
\begin{equation} \label{eq_crossl-gen}
  \xy 0;/r.18pc/:
  (0,0)*{\xybox{
    (-4,-4)*{};(4,4)*{} **\crv{(-4,-1) & (4,1)}?(1)*\dir{>} ;
    (4,-4)*{};(-4,4)*{} **\crv{(4,-1) & (-4,1)}?(0)*\dir{<};
    (-5,-3)*{\scs j};
     (6.5,-3)*{\scs i};
     (9,2)*{ \lambda};
     (-12,0)*{};(12,0)*{};
     }};
  \endxy
\quad := \quad
 \xy 0;/r.17pc/:
  (0,0)*{\xybox{
    (4,-4)*{};(-4,4)*{} **\crv{(4,-1) & (-4,1)}?(1)*\dir{>};
    (-4,-4)*{};(4,4)*{} **\crv{(-4,-1) & (4,1)};
   %  (-4,4)*{};(18,4)*{} **\crv{(-4,16) & (18,16)} ?(1)*\dir{>};
    % (4,-4)*{};(-18,-4)*{} **\crv{(4,-16) & (-18,-16)} ?(1)*\dir{<}?(0)*\dir{<};
     (-4,4);(-4,12) **\dir{-};
     (-12,-4);(-12,12) **\dir{-};
     (4,-4);(4,-12) **\dir{-};(12,4);(12,-12) **\dir{-};
     (16,1)*{\lambda};
     (-10,0)*{};(10,0)*{};
     (-4,-4)*{};(-12,-4)*{} **\crv{(-4,-10) & (-12,-10)}?(1)*\dir{<}?(0)*\dir{<};
      (4,4)*{};(12,4)*{} **\crv{(4,10) & (12,10)}?(1)*\dir{>}?(0)*\dir{>};
      (-14,11)*{\scs i};(-2,11)*{\scs j};
      (14,-11)*{\scs i};(2,-11)*{\scs j};
     }};
  \endxy
  \quad = \quad t_{ij} \;\;
 \xy 0;/r.17pc/:
  (0,0)*{\xybox{
    (-4,-4)*{};(4,4)*{} **\crv{(-4,-1) & (4,1)}?(1)*\dir{<};
    (4,-4)*{};(-4,4)*{} **\crv{(4,-1) & (-4,1)};
     (4,4);(4,12) **\dir{-};
     (12,-4);(12,12) **\dir{-};
     (-4,-4);(-4,-12) **\dir{-};(-12,4);(-12,-12) **\dir{-};
     (16,1)*{\lambda};
     (10,0)*{};(-10,0)*{};
     (4,-4)*{};(12,-4)*{} **\crv{(4,-10) & (12,-10)}?(1)*\dir{>}?(0)*\dir{>};
      (-4,4)*{};(-12,4)*{} **\crv{(-4,10) & (-12,10)}?(1)*\dir{<}?(0)*\dir{<};
     }};
     (12,11)*{\scs j};(0,11)*{\scs i};
      (-17,-11)*{\scs j};(-5,-11)*{\scs i};
  \endxy
\end{equation}
\begin{equation} \label{eq_crossr-gen}
  \xy 0;/r.18pc/:
  (0,0)*{\xybox{
    (-4,-4)*{};(4,4)*{} **\crv{(-4,-1) & (4,1)}?(0)*\dir{<} ;
    (4,-4)*{};(-4,4)*{} **\crv{(4,-1) & (-4,1)}?(1)*\dir{>};
    (5.1,-3)*{\scs i};
     (-6.5,-3)*{\scs j};
     (9,2)*{ \lambda};
     (-12,0)*{};(12,0)*{};
     }};
  \endxy
\quad := \quad
 \xy 0;/r.17pc/:
  (0,0)*{\xybox{
    (-4,-4)*{};(4,4)*{} **\crv{(-4,-1) & (4,1)}?(1)*\dir{>};
    (4,-4)*{};(-4,4)*{} **\crv{(4,-1) & (-4,1)};
     (4,4);(4,12) **\dir{-};
     (12,-4);(12,12) **\dir{-};
     (-4,-4);(-4,-12) **\dir{-};(-12,4);(-12,-12) **\dir{-};
     (16,-6)*{\lambda};
     (10,0)*{};(-10,0)*{};
     (4,-4)*{};(12,-4)*{} **\crv{(4,-10) & (12,-10)}?(1)*\dir{<}?(0)*\dir{<};
      (-4,4)*{};(-12,4)*{} **\crv{(-4,10) & (-12,10)}?(1)*\dir{>}?(0)*\dir{>};
      (14,11)*{\scs j};(2,11)*{\scs i};
      (-14,-11)*{\scs j};(-2,-11)*{\scs i};
     }};
  \endxy
  \quad = \quad t_{ji} \;\;
  \xy 0;/r.17pc/:
  (0,0)*{\xybox{
    (4,-4)*{};(-4,4)*{} **\crv{(4,-1) & (-4,1)}?(1)*\dir{<};
    (-4,-4)*{};(4,4)*{} **\crv{(-4,-1) & (4,1)};
   %  (-4,4)*{};(18,4)*{} **\crv{(-4,16) & (18,16)} ?(1)*\dir{>};
    % (4,-4)*{};(-18,-4)*{} **\crv{(4,-16) & (-18,-16)} ?(1)*\dir{<}?(0)*\dir{<};
     (-4,4);(-4,12) **\dir{-};
     (-12,-4);(-12,12) **\dir{-};
     (4,-4);(4,-12) **\dir{-};(12,4);(12,-12) **\dir{-};
     (16,6)*{\lambda};
     (-10,0)*{};(10,0)*{};
     (-4,-4)*{};(-12,-4)*{} **\crv{(-4,-10) & (-12,-10)}?(1)*\dir{>}?(0)*\dir{>};
      (4,4)*{};(12,4)*{} **\crv{(4,10) & (12,10)}?(1)*\dir{<}?(0)*\dir{<};
      (-14,11)*{\scs i};(-2,11)*{\scs j};(14,-11)*{\scs i};(2,-11)*{\scs j};
     }};
  \endxy
\end{equation}
where the second equality in \eqref{eq_crossl-gen} and \eqref{eq_crossr-gen}
follow from \eqref{eq_almost_cyclic}.

\item The $\cal{E}$'s carry an action of the KLR algebra associated to $Q$. The KLR algebra $R=R_Q$ associated to $Q$ is defined by finite $\Bbbk$-linear combinations of braid--like diagrams in the plane, where each strand is labeled by a vertex $i \in I$.  Strands can intersect and can carry dots but triple intersections are not allowed.  Diagrams are considered up to planar isotopy that do not change the combinatorial type of the diagram. We recall the local relations:
\begin{enumerate}[i)]
\item
If all strands are labeled by the same $i \in I$ then the  nilHecke algebra axioms \eqref{eq_nil_rels} and \eqref{eq_nil_dotslide} hold.

\item For $i \neq j$
\begin{equation}
 \vcenter{\xy 0;/r.17pc/:
    (-4,-4)*{};(4,4)*{} **\crv{(-4,-1) & (4,1)}?(1)*\dir{};
    (4,-4)*{};(-4,4)*{} **\crv{(4,-1) & (-4,1)}?(1)*\dir{};
    (-4,4)*{};(4,12)*{} **\crv{(-4,7) & (4,9)}?(1)*\dir{};
    (4,4)*{};(-4,12)*{} **\crv{(4,7) & (-4,9)}?(1)*\dir{};
    (8,8)*{\lambda};
    (4,12); (4,13) **\dir{-}?(1)*\dir{>};
    (-4,12); (-4,13) **\dir{-}?(1)*\dir{>};
  (-5.5,-3)*{\scs i};
     (5.5,-3)*{\scs j};
 \endxy}
 \qquad = \qquad
 \left\{
 \begin{array}{ccc}
     t_{ij}\;\xy 0;/r.17pc/:
  (3,9);(3,-9) **\dir{-}?(0)*\dir{<}+(2.3,0)*{};
  (-3,9);(-3,-9) **\dir{-}?(0)*\dir{<}+(2.3,0)*{};
  %(8,2)*{\lambda};
  (-5,-6)*{\scs i};     (5.1,-6)*{\scs j};
 \endxy &  &  \text{if $(\alpha_i, \alpha_j)=0$,}\\ \\
 t_{ij} \vcenter{\xy 0;/r.17pc/:
  (3,9);(3,-9) **\dir{-}?(0)*\dir{<}+(2.3,0)*{};
  (-3,9);(-3,-9) **\dir{-}?(0)*\dir{<}+(2.3,0)*{};
  %(8,2)*{\lambda};
  (-3,4)*{\bullet};(-6.5,5)*{};
  (-5,-6)*{\scs i};     (5.1,-6)*{\scs j};
 \endxy} \;\; + \;\; t_{ji}
  \vcenter{\xy 0;/r.17pc/:
  (3,9);(3,-9) **\dir{-}?(0)*\dir{<}+(2.3,0)*{};
  (-3,9);(-3,-9) **\dir{-}?(0)*\dir{<}+(2.3,0)*{};
  %(12,2)*{\lambda};
  (3,4)*{\bullet};(7,5)*{};
  (-5,-6)*{\scs i};     (5.1,-6)*{\scs j};
 \endxy}
   &  & \text{if $(\alpha_i, \alpha_j) \neq 0$,}
 \end{array}
 \right. \label{eq_r2_ij-gen}
\end{equation}

\item For $i \neq j$ the dot sliding relations
\begin{equation} \label{eq_dot_slide_ij-gen}
\xy 0;/r.18pc/:
  (0,0)*{\xybox{
    (-4,-4)*{};(4,6)*{} **\crv{(-4,-1) & (4,1)}?(1)*\dir{>}?(.75)*{\bullet};
    (4,-4)*{};(-4,6)*{} **\crv{(4,-1) & (-4,1)}?(1)*\dir{>};
    (-5,-3)*{\scs i};
     (5.1,-3)*{\scs j};
    % (8,1)*{ \lambda};
     (-10,0)*{};(10,0)*{};
     }};
  \endxy
 \;\; =
\xy 0;/r.18pc/:
  (0,0)*{\xybox{
    (-4,-4)*{};(4,6)*{} **\crv{(-4,-1) & (4,1)}?(1)*\dir{>}?(.25)*{\bullet};
    (4,-4)*{};(-4,6)*{} **\crv{(4,-1) & (-4,1)}?(1)*\dir{>};
    (-5,-3)*{\scs i};
     (5.1,-3)*{\scs j};
     %(8,1)*{ \lambda};
     (-10,0)*{};(10,0)*{};
     }};
  \endxy
\qquad  \xy 0;/r.18pc/:
  (0,0)*{\xybox{
    (-4,-4)*{};(4,6)*{} **\crv{(-4,-1) & (4,1)}?(1)*\dir{>};
    (4,-4)*{};(-4,6)*{} **\crv{(4,-1) & (-4,1)}?(1)*\dir{>}?(.75)*{\bullet};
    (-5,-3)*{\scs i};
     (5.1,-3)*{\scs j};
   %  (8,1)*{ \lambda};
     (-10,0)*{};(10,0)*{};
     }};
  \endxy
\;\;  =
  \xy 0;/r.18pc/:
  (0,0)*{\xybox{
    (-4,-4)*{};(4,6)*{} **\crv{(-4,-1) & (4,1)}?(1)*\dir{>} ;
    (4,-4)*{};(-4,6)*{} **\crv{(4,-1) & (-4,1)}?(1)*\dir{>}?(.25)*{\bullet};
    (-5,-3)*{\scs i};
     (5.1,-3)*{\scs j};
   %  (8,1)*{ \lambda};
     (-10,0)*{};(12,0)*{};
     }};
  \endxy
\end{equation}
hold.

\item Unless $i = k$ and $(\alpha_i, \alpha_j) < 0$ the relation
\begin{equation}
\vcenter{\xy 0;/r.17pc/:
    (-4,-4)*{};(4,4)*{} **\crv{(-4,-1) & (4,1)}?(1)*\dir{};
    (4,-4)*{};(-4,4)*{} **\crv{(4,-1) & (-4,1)}?(1)*\dir{};
    (4,4)*{};(12,12)*{} **\crv{(4,7) & (12,9)}?(1)*\dir{};
    (12,4)*{};(4,12)*{} **\crv{(12,7) & (4,9)}?(1)*\dir{};
    (-4,12)*{};(4,20)*{} **\crv{(-4,15) & (4,17)}?(1)*\dir{};
    (4,12)*{};(-4,20)*{} **\crv{(4,15) & (-4,17)}?(1)*\dir{};
    (-4,4)*{}; (-4,12) **\dir{-};
    (12,-4)*{}; (12,4) **\dir{-};
    (12,12)*{}; (12,20) **\dir{-};
    (4,20); (4,21) **\dir{-}?(1)*\dir{>};
    (-4,20); (-4,21) **\dir{-}?(1)*\dir{>};
    (12,20); (12,21) **\dir{-}?(1)*\dir{>};
   (18,8)*{\lambda};  (-6,-3)*{\scs i};
  (6,-3)*{\scs j};
  (15,-3)*{\scs k};
\endxy}
 \;\; =\;\;
\vcenter{\xy 0;/r.17pc/:
    (4,-4)*{};(-4,4)*{} **\crv{(4,-1) & (-4,1)}?(1)*\dir{};
    (-4,-4)*{};(4,4)*{} **\crv{(-4,-1) & (4,1)}?(1)*\dir{};
    (-4,4)*{};(-12,12)*{} **\crv{(-4,7) & (-12,9)}?(1)*\dir{};
    (-12,4)*{};(-4,12)*{} **\crv{(-12,7) & (-4,9)}?(1)*\dir{};
    (4,12)*{};(-4,20)*{} **\crv{(4,15) & (-4,17)}?(1)*\dir{};
    (-4,12)*{};(4,20)*{} **\crv{(-4,15) & (4,17)}?(1)*\dir{};
    (4,4)*{}; (4,12) **\dir{-};
    (-12,-4)*{}; (-12,4) **\dir{-};
    (-12,12)*{}; (-12,20) **\dir{-};
    (4,20); (4,21) **\dir{-}?(1)*\dir{>};
    (-4,20); (-4,21) **\dir{-}?(1)*\dir{>};
    (-12,20); (-12,21) **\dir{-}?(1)*\dir{>};
  (10,8)*{\lambda};
  (-14,-3)*{\scs i};
  (-6,-3)*{\scs j};
  (6,-3)*{\scs k};
\endxy}
 \label{eq_r3_easy-gen}
\end{equation}
holds. Otherwise, $(\alpha_i, \alpha_j) =-1$ and
\begin{equation}
\vcenter{\xy 0;/r.17pc/:
    (-4,-4)*{};(4,4)*{} **\crv{(-4,-1) & (4,1)}?(1)*\dir{};
    (4,-4)*{};(-4,4)*{} **\crv{(4,-1) & (-4,1)}?(1)*\dir{};
    (4,4)*{};(12,12)*{} **\crv{(4,7) & (12,9)}?(1)*\dir{};
    (12,4)*{};(4,12)*{} **\crv{(12,7) & (4,9)}?(1)*\dir{};
    (-4,12)*{};(4,20)*{} **\crv{(-4,15) & (4,17)}?(1)*\dir{};
    (4,12)*{};(-4,20)*{} **\crv{(4,15) & (-4,17)}?(1)*\dir{};
    (-4,4)*{}; (-4,12) **\dir{-};
    (12,-4)*{}; (12,4) **\dir{-};
    (12,12)*{}; (12,20) **\dir{-};
    (4,20); (4,21) **\dir{-}?(1)*\dir{>};
    (-4,20); (-4,21) **\dir{-}?(1)*\dir{>};
    (12,20); (12,21) **\dir{-}?(1)*\dir{>};
   (18,8)*{\lambda};  (-6,-3)*{\scs i};
  (6,-3)*{\scs j};
  (15,-3)*{\scs k};
\endxy}
\quad - \quad
\vcenter{\xy 0;/r.17pc/:
    (4,-4)*{};(-4,4)*{} **\crv{(4,-1) & (-4,1)}?(1)*\dir{};
    (-4,-4)*{};(4,4)*{} **\crv{(-4,-1) & (4,1)}?(1)*\dir{};
    (-4,4)*{};(-12,12)*{} **\crv{(-4,7) & (-12,9)}?(1)*\dir{};
    (-12,4)*{};(-4,12)*{} **\crv{(-12,7) & (-4,9)}?(1)*\dir{};
    (4,12)*{};(-4,20)*{} **\crv{(4,15) & (-4,17)}?(1)*\dir{};
    (-4,12)*{};(4,20)*{} **\crv{(-4,15) & (4,17)}?(1)*\dir{};
    (4,4)*{}; (4,12) **\dir{-};
    (-12,-4)*{}; (-12,4) **\dir{-};
    (-12,12)*{}; (-12,20) **\dir{-};
    (4,20); (4,21) **\dir{-}?(1)*\dir{>};
    (-4,20); (-4,21) **\dir{-}?(1)*\dir{>};
    (-12,20); (-12,21) **\dir{-}?(1)*\dir{>};
  (10,8)*{\lambda};
  (-14,-3)*{\scs i};
  (-6,-3)*{\scs j};
  (6,-3)*{\scs i};
\endxy}
 \;\; =\;\;
 t_{ij} \;\;
\xy 0;/r.17pc/:
  (4,12);(4,-12) **\dir{-}?(0)*\dir{<};
  (-4,12);(-4,-12) **\dir{-}?(0)*\dir{<}?(.25)*\dir{};
  (12,12);(12,-12) **\dir{-}?(0)*\dir{<}?(.25)*\dir{};
  %(20,-2)*{\lambda};
  (-6,-9)*{\scs i};     (6.1,-9)*{\scs j};
  (14,-9)*{\scs i};
 \endxy
 \label{eq_r3_hard-gen}
\end{equation}
\end{enumerate}

\item When $i \ne j$ one has the mixed relations  relating $\cal{E}_i \cal{F}_j$ and $\cal{F}_j \cal{E}_i$:
\begin{equation} \label{mixed_rel}
 \vcenter{   \xy 0;/r.18pc/:
    (-4,-4)*{};(4,4)*{} **\crv{(-4,-1) & (4,1)}?(1)*\dir{>};
    (4,-4)*{};(-4,4)*{} **\crv{(4,-1) & (-4,1)}?(1)*\dir{<};?(0)*\dir{<};
    (-4,4)*{};(4,12)*{} **\crv{(-4,7) & (4,9)};
    (4,4)*{};(-4,12)*{} **\crv{(4,7) & (-4,9)}?(1)*\dir{>};
  (8,8)*{\lambda};(-6,-3)*{\scs i};
     (6,-3)*{\scs j};
 \endxy}
 \;\; = \;\; t_{ji}\;\;
\xy 0;/r.18pc/:
  (3,9);(3,-9) **\dir{-}?(.55)*\dir{>}+(2.3,0)*{};
  (-3,9);(-3,-9) **\dir{-}?(.5)*\dir{<}+(2.3,0)*{};
  (8,2)*{\lambda};(-5,-6)*{\scs i};     (5.1,-6)*{\scs j};
 \endxy
\qquad \quad
    \vcenter{\xy 0;/r.18pc/:
    (-4,-4)*{};(4,4)*{} **\crv{(-4,-1) & (4,1)}?(1)*\dir{<};?(0)*\dir{<};
    (4,-4)*{};(-4,4)*{} **\crv{(4,-1) & (-4,1)}?(1)*\dir{>};
    (-4,4)*{};(4,12)*{} **\crv{(-4,7) & (4,9)}?(1)*\dir{>};
    (4,4)*{};(-4,12)*{} **\crv{(4,7) & (-4,9)};
  (8,8)*{\lambda};(-6,-3)*{\scs i};
     (6,-3)*{\scs j};
 \endxy}
 \;\;=\;\; t_{ij}\;\;
\xy 0;/r.18pc/:
  (3,9);(3,-9) **\dir{-}?(.5)*\dir{<}+(2.3,0)*{};
  (-3,9);(-3,-9) **\dir{-}?(.55)*\dir{>}+(2.3,0)*{};
  (8,2)*{\lambda};(-5,-6)*{\scs i};     (5.1,-6)*{\scs j};
 \endxy
\end{equation}

\item \label{item_positivity} Negative degree bubbles are zero. That is, for all $m \in \Z_+$ one has
\begin{equation} \label{eq_positivity_bubbles}
\xy 0;/r.18pc/:
 (-12,0)*{\icbub{m}{i}};
 (-8,8)*{\lambda};
 \endxy
  = 0
 \qquad  \text{if $m<\lambda_i-1$,} \qquad \xy 0;/r.18pc/: (-12,0)*{\iccbub{m}{i}};
 (-8,8)*{\lambda};
 \endxy = 0\quad
  \text{if $m< -\lambda_i-1$.}
\end{equation}
On the other hand, a dotted bubble of degree zero is just  the identity 2-morphism:
%\footnote{One can define the 2-category so that degree zero bubbles are multiplication by
%arbitrary scalars at the cost of modifying some of the other relations, see for example~\cite{Lau4,MSV2}.  However, it is shown in \cite{CLau} that the resulting 2-categories are all isomorphic.}:
\[
\xy 0;/r.18pc/:
 (0,0)*{\icbub{\lambda_i-1}{i}};
  (4,8)*{\lambda};
 \endxy
  =  \Id_{\onenn{\lambda}} \quad \text{for $\lambda_i \geq 1$,}
  \qquad \quad
  \xy 0;/r.18pc/:
 (0,0)*{\iccbub{-\lambda_i-1}{i}};
  (4,8)*{\lambda};
 \endxy  =  \Id_{\onenn{\lambda}} \quad \text{for $\lambda_i \leq -1$.}\]

\item \label{item_highersl2} For any $i \in I$ one has the extended ${\mathfrak{sl}}_2$-relations. In order to describe certain extended ${\mathfrak{sl}}_2$ relations it is convenient to use a shorthand notation from \cite{Lau1} called fake bubbles. These are diagrams for dotted bubbles where the labels of the number of dots is negative, but the total degree of the dotted bubble taken with these negative dots is still positive. They allow us to write these extended ${\mathfrak{sl}}_2$ relations more uniformly (i.e. independent on whether the weight $\lambda_i$ is positive or negative).
\begin{itemize}
 \item Degree zero fake bubbles are equal to the identity 2-morphisms
\[
 \xy 0;/r.18pc/:
    (2,0)*{\icbub{\l_i-1}{i}};
  (12,8)*{\lambda};
 \endxy
  =  \Id_{\onenn{\lambda}} \quad \text{if $\lambda_i \leq 0$,}
  \qquad \quad
\xy 0;/r.18pc/:
    (2,0)*{\iccbub{-\lambda_i-1}{i}};
  (12,8)*{\lambda};
 \endxy =  \Id_{\onenn{\lambda}} \quad  \text{if $\lambda_i \geq 0$}.\]

  \item Higher degree fake bubbles for $\lambda_i<0$ are defined inductively as
  \begin{equation} \label{eq_fake_nleqz}
  \vcenter{\xy 0;/r.18pc/:
    (2,-11)*{\icbub{\l_i-1+j}{i}};
  (12,-2)*{\l};
 \endxy} \;\; =
 \left\{
 \begin{array}{cl}
  \;\; -\;\;
\xsum{\xy (0,6)*{};  (0,1)*{\scs a+b=j}; (0,-2)*{\scs b\geq 1}; \endxy}
\;\; \vcenter{\xy 0;/r.18pc/:
    (2,0)*{\cbub{\l_i-1+a}{}};
    (20,0)*{\ccbub{-\l-1+b}{}};
  (12,8)*{\lambda};
 \endxy}  & \text{if $0 \leq j < -\l_i+1$} \\ & \\
   0 & \text{if $j < 0$. }
 \end{array}
\right.
 \end{equation}

  \item Higher degree fake bubbles for $\lambda_i>0$ are defined inductively as
   \begin{equation} \label{eq_fake_ngeqz}
  \vcenter{\xy 0;/r.18pc/:
    (2,-11)*{\iccbub{-\l_i-1+j}{i}};
  (12,-2)*{\l};
 \endxy} \;\; =
 \left\{
 \begin{array}{cl}
  \;\; -\;\;
\xsum{\xy (0,6)*{}; (0,1)*{\scs a+b=j}; (0,-2)*{\scs a\geq 1}; \endxy}
\;\; \vcenter{\xy 0;/r.18pc/:
    (2,0)*{\cbub{\l_i-1+a}{}};
    (20,0)*{\ccbub{-\l-1+b}{}};
  (12,8)*{\lambda};
 \endxy}  & \text{if $0 \leq j < \l_i+1$} \\ & \\
   0 & \text{if $j < 0$. }
 \end{array}
\right.
\end{equation}
\end{itemize}
These equations arise from the homogeneous terms in $t$ of the `infinite Grassmannian' equation
\begin{center}
% i'll concede this one, equation array makes the numbering better...
\begin{eqnarray}
 \makebox[0pt]{ $
\left( \xy 0;/r.15pc/:
 (0,0)*{\iccbub{-\l_i-1}{i}};
  (4,8)*{\l};
 \endxy
 +
 \xy 0;/r.15pc/:
 (0,0)*{\iccbub{-\l_i-1+1}{i}};
  (4,8)*{\l};
 \endxy t
 + \cdots +
\xy 0;/r.15pc/:
 (0,0)*{\iccbub{-\l_i-1+\alpha}{i}};
  (4,8)*{\l};
 \endxy t^{\alpha}
 + \cdots
\right)
%%%%
\left(\xy 0;/r.15pc/:
 (0,0)*{\icbub{\l_i-1}{i}};
  (4,8)*{\l};
 \endxy
 + \xy 0;/r.15pc/:
 (0,0)*{\icbub{\l_i-1+1}{i}};
  (4,8)*{\l};
 \endxy t
 +\cdots +
\xy 0;/r.15pc/:
 (0,0)*{\icbub{\l_i-1+\alpha}{i}};
 (4,8)*{\l};
 \endxy t^{\alpha}
 + \cdots
\right) = \Id_{\onel}.$ } \nn \\ \label{eq_infinite_Grass}
\end{eqnarray}
\end{center}
Now we can define the extended ${\mathfrak{sl}}_2$ relations.  Note that in \cite{CLau} additional curl relations were provided that can be derived from those above.  Here we provide a minimal set of relations.

If $\l_i > 0$ then we have:
\begin{equation} \label{eq_reduction-ngeqz}
  \xy 0;/r.17pc/:
  (14,8)*{\l};
  (-3,-10)*{};(3,5)*{} **\crv{(-3,-2) & (2,1)}?(1)*\dir{>};?(.15)*\dir{>};
    (3,-5)*{};(-3,10)*{} **\crv{(2,-1) & (-3,2)}?(.85)*\dir{>} ?(.1)*\dir{>};
  (3,5)*{}="t1";  (9,5)*{}="t2";
  (3,-5)*{}="t1'";  (9,-5)*{}="t2'";
   "t1";"t2" **\crv{(4,8) & (9, 8)};
   "t1'";"t2'" **\crv{(4,-8) & (9, -8)};
   "t2'";"t2" **\crv{(10,0)} ;
   (-6,-8)*{\scs i};
 \endxy\;\; = \;\; 0
   \qquad \quad
 \vcenter{\xy 0;/r.17pc/:
  (-8,0)*{};(-6,-8)*{\scs i};(6,-8)*{\scs i};
  (8,0)*{};
  (-4,10)*{}="t1";
  (4,10)*{}="t2";
  (-4,-10)*{}="b1";
  (4,-10)*{}="b2";
  "t1";"b1" **\dir{-} ?(.5)*\dir{>};
  "t2";"b2" **\dir{-} ?(.5)*\dir{<};
  (10,2)*{\l};
  \endxy}
\;\; = \;\; -\;\;
   \vcenter{\xy 0;/r.17pc/:
    (-4,-4)*{};(4,4)*{} **\crv{(-4,-1) & (4,1)}?(1)*\dir{<};?(0)*\dir{<};
    (4,-4)*{};(-4,4)*{} **\crv{(4,-1) & (-4,1)}?(1)*\dir{>};
    (-4,4)*{};(4,12)*{} **\crv{(-4,7) & (4,9)}?(1)*\dir{>};
    (4,4)*{};(-4,12)*{} **\crv{(4,7) & (-4,9)};
  (8,8)*{\l};(-6.5,-3)*{\scs i};  (6,-3)*{\scs i};
 \endxy}
\end{equation}
\begin{equation}
 \vcenter{\xy 0;/r.17pc/:
  (-8,0)*{};
  (8,0)*{};
  (-4,10)*{}="t1";
  (4,10)*{}="t2";
  (-4,-10)*{}="b1";
  (4,-10)*{}="b2";(-6,-8)*{\scs i};(6,-8)*{\scs i};
  "t1";"b1" **\dir{-} ?(.5)*\dir{<};
  "t2";"b2" **\dir{-} ?(.5)*\dir{>};
  (10,2)*{\l};
  \endxy}
\;\; = \;\; -\;\;
 \vcenter{   \xy 0;/r.17pc/:
    (-4,-4)*{};(4,4)*{} **\crv{(-4,-1) & (4,1)}?(1)*\dir{>};
    (4,-4)*{};(-4,4)*{} **\crv{(4,-1) & (-4,1)}?(1)*\dir{<};?(0)*\dir{<};
    (-4,4)*{};(4,12)*{} **\crv{(-4,7) & (4,9)};
    (4,4)*{};(-4,12)*{} **\crv{(4,7) & (-4,9)}?(1)*\dir{>};
  (8,8)*{\l};
     (-6,-3)*{\scs i};
     (6.5,-3)*{\scs i};
 \endxy}
  \;\; + \;\;
   \sum_{ \xy  (0,3)*{\scs f_1+f_2+f_3}; (0,0)*{\scs =\lambda_i-1};\endxy}
    \vcenter{\xy 0;/r.17pc/:
    (-12,10)*{\l};
    (-8,0)*{};
  (8,0)*{};
  (-4,-15)*{}="b1";
  (4,-15)*{}="b2";
  "b2";"b1" **\crv{(5,-8) & (-5,-8)}; ?(.05)*\dir{<} ?(.93)*\dir{<}
  ?(.8)*\dir{}+(0,-.1)*{\bullet}+(-3,2)*{\scs f_3};
  (-4,15)*{}="t1";
  (4,15)*{}="t2";
  "t2";"t1" **\crv{(5,8) & (-5,8)}; ?(.15)*\dir{>} ?(.95)*\dir{>}
  ?(.4)*\dir{}+(0,-.2)*{\bullet}+(3,-2)*{\scs \; f_1};
  (0,0)*{\iccbub{\scs \quad -\l_i-1+f_2}{i}};
  (7,-13)*{\scs i};
  (-7,13)*{\scs i};
  \endxy} \label{eq_ident_decomp-ngeqz}
\end{equation}

If $\lambda_i < 0$ then we have:
\begin{equation} \label{eq_reduction-nleqz}
  \xy 0;/r.17pc/:
  (-14,8)*{\l};
  (3,-10)*{};(-3,5)*{} **\crv{(3,-2) & (-2,1)}?(1)*\dir{>};?(.15)*\dir{>};
    (-3,-5)*{};(3,10)*{} **\crv{(-2,-1) & (3,2)}?(.85)*\dir{>} ?(.1)*\dir{>};
  (-3,5)*{}="t1";  (-9,5)*{}="t2";
  (-3,-5)*{}="t1'";  (-9,-5)*{}="t2'";
   "t1";"t2" **\crv{(-4,8) & (-9, 8)};
   "t1'";"t2'" **\crv{(-4,-8) & (-9, -8)};
   "t2'";"t2" **\crv{(-10,0)} ;
   (6,-8)*{\scs i};
 \endxy\;\; = \;\;
0
\qquad \qquad
\vcenter{\xy 0;/r.17pc/:
  (-8,0)*{};
  (8,0)*{};
  (-4,10)*{}="t1";
  (4,10)*{}="t2";
  (-4,-10)*{}="b1";
  (4,-10)*{}="b2";(-6,-8)*{\scs i};(6,-8)*{\scs i};
  "t1";"b1" **\dir{-} ?(.5)*\dir{<};
  "t2";"b2" **\dir{-} ?(.5)*\dir{>};
  (10,2)*{\l};
  \endxy}
\;\; = \;\; -\;\;
\vcenter{   \xy 0;/r.17pc/:
    (-4,-4)*{};(4,4)*{} **\crv{(-4,-1) & (4,1)}?(1)*\dir{>};
    (4,-4)*{};(-4,4)*{} **\crv{(4,-1) & (-4,1)}?(1)*\dir{<};?(0)*\dir{<};
    (-4,4)*{};(4,12)*{} **\crv{(-4,7) & (4,9)};
    (4,4)*{};(-4,12)*{} **\crv{(4,7) & (-4,9)}?(1)*\dir{>};
  (8,8)*{\l};(-6,-3)*{\scs i};
     (6.5,-3)*{\scs i};
 \endxy}
\end{equation}
\begin{equation}
 \vcenter{\xy 0;/r.17pc/:
  (-8,0)*{};(-6,-8)*{\scs i};(6,-8)*{\scs i};
  (8,0)*{};
  (-4,10)*{}="t1";
  (4,10)*{}="t2";
  (-4,-10)*{}="b1";
  (4,-10)*{}="b2";
  "t1";"b1" **\dir{-} ?(.5)*\dir{>};
  "t2";"b2" **\dir{-} ?(.5)*\dir{<};
  (10,2)*{\l};
  (-10,2)*{\l};
  \endxy}
\;\; = \;\;
  -\;\;\vcenter{\xy 0;/r.17pc/:
    (-4,-4)*{};(4,4)*{} **\crv{(-4,-1) & (4,1)}?(1)*\dir{<};?(0)*\dir{<};
    (4,-4)*{};(-4,4)*{} **\crv{(4,-1) & (-4,1)}?(1)*\dir{>};
    (-4,4)*{};(4,12)*{} **\crv{(-4,7) & (4,9)}?(1)*\dir{>};
    (4,4)*{};(-4,12)*{} **\crv{(4,7) & (-4,9)};
  (8,8)*{\l};(-6.5,-3)*{\scs i};  (6,-3)*{\scs i};
 \endxy}
  \;\; + \;\;
    \sum_{ \xy  (0,3)*{\scs g_1+g_2+g_3}; (0,0)*{\scs =-\l_i-1};\endxy}
    \vcenter{\xy 0;/r.17pc/:
    (-8,0)*{};
  (8,0)*{};
  (-4,-15)*{}="b1";
  (4,-15)*{}="b2";
  "b2";"b1" **\crv{(5,-8) & (-5,-8)}; ?(.1)*\dir{>} ?(.95)*\dir{>}
  ?(.8)*\dir{}+(0,-.1)*{\bullet}+(-3,2)*{\scs g_3};
  (-4,15)*{}="t1";
  (4,15)*{}="t2";
  "t2";"t1" **\crv{(5,8) & (-5,8)}; ?(.15)*\dir{<} ?(.9)*\dir{<}
  ?(.4)*\dir{}+(0,-.2)*{\bullet}+(3,-2)*{\scs g_1};
  (0,0)*{\icbub{\scs \quad\; \l_i-1 + g_2}{i}};
    (7,-13)*{\scs i};
  (-7,13)*{\scs i};
  (-10,10)*{\l};
  \endxy} \label{eq_ident_decomp-nleqz}
\end{equation}

If $\lambda_i =0$ then we have:
\begin{equation}\label{eq_reduction-neqz}
  \xy 0;/r.17pc/:
  (14,8)*{\l};
  (-3,-10)*{};(3,5)*{} **\crv{(-3,-2) & (2,1)}?(1)*\dir{>};?(.15)*\dir{>};
    (3,-5)*{};(-3,10)*{} **\crv{(2,-1) & (-3,2)}?(.85)*\dir{>} ?(.1)*\dir{>};
  (3,5)*{}="t1";  (9,5)*{}="t2";
  (3,-5)*{}="t1'";  (9,-5)*{}="t2'";
   "t1";"t2" **\crv{(4,8) & (9, 8)};
   "t1'";"t2'" **\crv{(4,-8) & (9, -8)};
   "t2'";"t2" **\crv{(10,0)} ;
   (-6,-8)*{\scs i};
 \endxy\;\; = \;\;-
   \xy 0;/r.17pc/:
  (-8,8)*{\l};
  (0,0)*{\bbe{}};
  (-3,-8)*{\scs i};
 \endxy
\qquad \qquad
  \xy 0;/r.17pc/:
  (-14,8)*{\l};
  (3,-10)*{};(-3,5)*{} **\crv{(3,-2) & (-2,1)}?(1)*\dir{>};?(.15)*\dir{>};
    (-3,-5)*{};(3,10)*{} **\crv{(-2,-1) & (3,2)}?(.85)*\dir{>} ?(.1)*\dir{>};
  (-3,5)*{}="t1";  (-9,5)*{}="t2";
  (-3,-5)*{}="t1'";  (-9,-5)*{}="t2'";
   "t1";"t2" **\crv{(-4,8) & (-9, 8)};
   "t1'";"t2'" **\crv{(-4,-8) & (-9, -8)};
   "t2'";"t2" **\crv{(-10,0)} ;
   (6,-8)*{\scs i};
 \endxy \;\; = \;\;
  \xy 0;/r.17pc/:
  (-8,8)*{\l};
  (0,0)*{\bbe{}};
  (-3,-8)*{\scs i};
 \endxy
\end{equation}
\begin{equation}\label{eq_reduction-neqz_2}
 \vcenter{\xy 0;/r.17pc/:
  (-8,0)*{};
  (8,0)*{};
  (-4,10)*{}="t1";
  (4,10)*{}="t2";
  (-4,-10)*{}="b1";
  (4,-10)*{}="b2";(-6,-8)*{\scs i};(6,-8)*{\scs i};
  "t1";"b1" **\dir{-} ?(.5)*\dir{<};
  "t2";"b2" **\dir{-} ?(.5)*\dir{>};
  (10,2)*{\l};
  \endxy}
\;\; = \;\;
 \;\; - \;\;
 \vcenter{   \xy 0;/r.17pc/:
    (-4,-4)*{};(4,4)*{} **\crv{(-4,-1) & (4,1)}?(1)*\dir{>};
    (4,-4)*{};(-4,4)*{} **\crv{(4,-1) & (-4,1)}?(1)*\dir{<};?(0)*\dir{<};
    (-4,4)*{};(4,12)*{} **\crv{(-4,7) & (4,9)};
    (4,4)*{};(-4,12)*{} **\crv{(4,7) & (-4,9)}?(1)*\dir{>};
  (8,8)*{l};(-6,-3)*{\scs i};
     (6.5,-3)*{\scs i};
 \endxy}
   \qquad \quad
 \vcenter{\xy 0;/r.17pc/:
  (-8,0)*{};(-6,-8)*{\scs i};(6,-8)*{\scs i};
  (8,0)*{};
  (-4,10)*{}="t1";
  (4,10)*{}="t2";
  (-4,-10)*{}="b1";
  (4,-10)*{}="b2";
  "t1";"b1" **\dir{-} ?(.5)*\dir{>};
  "t2";"b2" **\dir{-} ?(.5)*\dir{<};
  (10,2)*{\l};
  \endxy}
\;\; = \;\;
 \;\; - \;\;
   \vcenter{\xy 0;/r.17pc/:
    (-4,-4)*{};(4,4)*{} **\crv{(-4,-1) & (4,1)}?(1)*\dir{<};?(0)*\dir{<};
    (4,-4)*{};(-4,4)*{} **\crv{(4,-1) & (-4,1)}?(1)*\dir{>};
    (-4,4)*{};(4,12)*{} **\crv{(-4,7) & (4,9)}?(1)*\dir{>};
    (4,4)*{};(-4,12)*{} **\crv{(4,7) & (-4,9)};
  (8,8)*{\l};(-6,-3)*{\scs i};  (6,-3)*{\scs i};
 \endxy}
\end{equation}
\end{enumerate}

While the definition above may seem quite involved, the relations above encode a great deal of interesting combinatorics.  In particular, the relations above imply that all of the equations between generators in $\dot{\mathbf{U}}(\mathfrak{sl}_n)$ lift to explicit isomorphisms in $\Ucat_Q(\mf{sl}_n)$.  For more details see ~\cite{KL3}.

In what follows it is often convenient to introduce a shorthand notation
\[
    \xy 0;/r.18pc/:
  (4,8)*{\lambda};
  (2,-2)*{\icbub{\spadesuit+\alpha}{i}};
 \endxy \;\; := \;\;
   \xy 0;/r.18pc/:
  (4,8)*{\lambda};
  (2,-2)*{\icbub{\lambda_i-1+\alpha}{i}};
 \endxy
 \qquad
 \qquad
    \xy 0;/r.18pc/:
  (4,8)*{\lambda};
  (2,-2)*{\iccbub{\spadesuit+\alpha}{i}};
 \endxy \;\; := \;\;
   \xy 0;/r.18pc/:
  (4,8)*{\lambda};
  (2,-2)*{\iccbub{-\lambda_i-1+\alpha}{i}};
 \endxy
\]
for all $\lambda_i$.  Note that as long as $\alpha \geq 0$ this notation makes sense even when $\spadesuit+\alpha <0$.  These negative values are the fake bubbles defined above.

\begin{exercise} \label{exercise_bubble-slides}
Use the relations above to show that the following bubble slide equations
\begin{eqnarray}
    \xy 0;/r.18pc/:
  (14,8)*{\lambda};
  (0,0)*{\bbe{}};
  (0,-12)*{\scs j};
  (12,-2)*{\iccbub{\spadesuit+\alpha}{i}};
  (0,6)*{ }+(7,-1)*{\scs  };
 \endxy
 &\quad = \quad&
 \left\{
 \begin{array}{ccl}
  \xsum{f=0}{\alpha}(\alpha+1-f)
   \xy 0;/r.18pc/:
  (0,8)*{\lambda+\alpha_j};
  (12,0)*{\bbe{}};
  (12,-12)*{\scs j};
  (0,-2)*{\iccbub{\spadesuit+f}{i}};
  (12,6)*{\bullet}+(5,-1)*{\scs \alpha-f};
 \endxy
    &  & \text{if $i=j$} \\ \\
        \xy 0;/r.18pc/:
  (0,8)*{\lambda+\alpha_j};
  (12,0)*{\bbe{}};
  (11,-12)*{\scs j};
  (0,-2)*{\iccbub{\spadesuit+\alpha}{i}};
 \endxy
 \quad + \quad  t_{ij}^{-1}t_{ji} \;\;
  \xy 0;/r.18pc/:
  (0,8)*{\lambda+\alpha_j};
  (12,0)*{\bbe{}};
  (12,-12)*{\scs j};
  (0,-2)*{\iccbub{\spadesuit+\alpha-1}{i}};
  (12,6)*{\bullet}+(5,-1)*{\scs };
 \endxy
   &  & \text{if $i \cdot j =-1$} \\
 \qquad \qquad  \xy 0;/r.18pc/:
  (0,8)*{\lambda+\alpha_j};
  (12,0)*{\bbe{}};
  (12,-12)*{\scs j};
  (0,-2)*{\iccbub{\spadesuit+\alpha}{i}};
 \endxy  &  & \text{if $i \cdot j=0$}
 \end{array}
 \right.
\end{eqnarray}
\begin{eqnarray} \label{c_slide_right}
    \xy 0;/r.18pc/:
  (15,8)*{\lambda};
  (11,0)*{\bbe{}};
  (11,-12)*{\scs j};
  (0,-2)*{\icbub{\spadesuit+\alpha\quad }{i}};
 \endxy
   &\quad = \quad&
  \left\{\begin{array}{ccl}
     \xsum{f=0}{\alpha}(\alpha+1-f)
     \xy 0;/r.18pc/:
  (18,8)*{\lambda};
  (0,0)*{\bbe{}};
  (0,-12)*{\scs j};
  (14,-4)*{\icbub{\spadesuit+f}{i}};
  (0,6)*{\bullet }+(5,-1)*{\scs \alpha-f};
 \endxy
         &  & \text{if $i=j$}  \\
  t_{ij}^{-1}t_{ji}\;\;  \xy 0;/r.18pc/:
  (18,8)*{\lambda};
  (0,0)*{\bbe{}};
  (0,-12)*{\scs j};
  (12,-2)*{\icbub{\spadesuit+ \alpha-1}{i}};
    (0,6)*{\bullet }+(5,-1)*{\scs};
 \endxy
 \quad + \quad
\xy 0;/r.18pc/:
  (18,8)*{\lambda};
  (0,0)*{\bbe{}};
  (0,-12)*{\scs j};
  (12,-2)*{\icbub{\spadesuit+ \alpha}{i}};
 \endxy
         & &  \text{if $i\cdot j =-1$}\\
    \xy 0;/r.18pc/:
  (15,8)*{\lambda};
  (0,0)*{\bbe{}};
  (0,-12)*{\scs j};
  (12,-2)*{\icbub{\spadesuit+\alpha}{i}};
 \endxy          &  & \text{if $i \cdot j = 0$}
         \end{array}
 \right.
\end{eqnarray}
 \begin{equation}
      \xy 0;/r.18pc/:
  (15,8)*{\lambda};
  (0,0)*{\bbe{}};
  (0,-12)*{\scs j};
  (12,-2)*{\icbub{\spadesuit+\alpha}{i}};
 \endxy
   =
  \left\{
  \begin{array}{cl}
     \xy 0;/r.18pc/:
  (0,8)*{\lambda+\alpha_i};
  (12,0)*{\bbe{}};
  (12,-12)*{\scs j};
  (0,-2)*{\icbub{\spadesuit+(\alpha-2)}{i}};
  (12,6)*{\bullet}+(3,-1)*{\scs 2};
 \endxy
   -2 \;
         \xy 0;/r.18pc/:
  (0,8)*{\lambda+\alpha_i};
  (12,0)*{\bbe{}};
  (12,-12)*{\scs j};
  (0,-2)*{\icbub{\spadesuit+(\alpha-1)}{i}};
  (12,6)*{\bullet}+(8,-1)*{\scs };
 \endxy
 + \;\;
     \xy 0;/r.18pc/:
  (0,8)*{\lambda+\alpha_i};
  (12,0)*{\bbe{}};
  (12,-12)*{\scs j};
  (0,-2)*{\icbub{\spadesuit+\alpha}{i}};
  (12,6)*{}+(8,-1)*{\scs };
 \endxy
  &   \text{if $i = j$} \\
  \xsum{f=0}{\alpha} (-t_{ij}^{-1}t_{ji})^{f}
  \xy 0;/r.18pc/:
  (0,8)*{\lambda+\alpha_j};
  (14,0)*{\bbe{}};
  (12,-12)*{\scs j};
  (0,-2)*{\icbub{\spadesuit+\alpha-f}{i}};
  (14,6)*{\bullet}+(3,-1)*{\scs f};
 \endxy &   \text{if $i\cdot j =-1$}
  \end{array}
 \right.
\end{equation}
\begin{equation} \label{cc_slide_right}
    \xy 0;/r.18pc/:
  (0,8)*{\lambda+\alpha_j};
  (12,0)*{\bbe{}};
  (12,-12)*{\scs j};
  (0,-2)*{\iccbub{\spadesuit+\alpha}{i}};
  (12,6)*{}+(8,-1)*{\scs };
 \endxy
  =
\left\{
\begin{array}{cc}
    \xy 0;/r.18pc/:
  (15,8)*{\lambda};
  (0,0)*{\bbe{}};
  (0,-12)*{\scs j};
  (12,-2)*{\iccbub{\spadesuit+(\alpha-2)}{i}};
  (0,6)*{\bullet }+(3,1)*{\scs 2};
 \endxy
  -2 \;
      \xy 0;/r.18pc/:
  (15,8)*{\lambda};
  (0,0)*{\bbe{}};
  (0,-12)*{\scs j};
  (12,-2)*{\iccbub{\spadesuit+(\alpha-1)}{i}};
  (0,6)*{\bullet }+(5,-1)*{\scs };
 \endxy
 + \;\;
      \xy 0;/r.18pc/:
  (15,8)*{\lambda};
  (0,0)*{\bbe{}};
  (0,-12)*{\scs j};
  (12,-2)*{\iccbub{(\spadesuit+\alpha}{i}};
  %(0,6)*{\bullet }+(5,-1)*{\scs \alpha-f};
 \endxy
  &   \text{if $i=j$} \\
   \xsum{f=0}{\alpha}(-t_{ij}^{-1}t_{ji})^f
    \xy 0;/r.18pc/:
  (15,8)*{\lambda};
  (0,0)*{\bbe{}};
  (0,-12)*{\scs j};
  (14,-2)*{\iccbub{\spadesuit+(\alpha-f)}{i}};
  (0,6)*{\bullet }+(3,1)*{\scs f};
 \endxy
    &   \text{if $i \cdot j =-1$}
\end{array}
\right.
\end{equation}
hold in $\Ucat_Q(\mf{sl}_n)$.   Hint: see equations (6.8) and (6.9) of \cite{CLau}.
\end{exercise}

% ---------------------------------------------------------------------
%
\subsection{Symmetric functions and bubbles}\label{sec_sym-bub}
%
% ---------------------------------------------------------------------

The calculus of closed diagrams in the 2-category $\Ucat_Q(\mf{sl}_n)$ is remarkably rich.   A prominent role is played by the non-nested dotted bubbles of a fixed orientation since any closed diagram in the graphical calculus for $\Ucat_Q(\mf{sl}_n)$ can be reduced to composites of such diagrams.

There is a beautiful analogy between dotted bubbles in the graphical calculus and various bases for the ring of symmetric functions.  In \cite{Lau1} it is shown that there is an isomorphism

\begin{eqnarray}
  \psi_{\lambda} \maps \;\;{\rm Sym} & \;\; \longrightarrow\;\;& Z(\l)= \Ucat_Q(\mf{sl}_2)(\onel,\onel) \\
  h_r & \;\;\mapsto \;\; &  \xy
  (0,0)*{\icbub{\spadesuit+r}{i}};
  (8,8)*{\lambda};
  \endxy \nn \\
   (-1)^s e_s& \;\;\mapsto \;\; & \xy
  (0,0)*{\iccbub{\spadesuit+s}{i}};
  (8,8)*{\lambda};
  \endxy \nn
\end{eqnarray}
where ${\rm Sym}$ denotes the ring of symmetric functions, $h_r$ denotes the complete symmetric function of degree $r$, and $e_{s}$ denotes the elementary symmetric function of degree $s$.  In fact, under this isomorphisms the well known relationship between complete and elementary symmetric functions becomes the infinite Grassmannian equation~\eqref{eq_infinite_Grass}.

It is well known that for a partition $\lambda = (\lambda_1, \dots, \lambda_n)$ with $\lambda_1 \geq \lambda_2 \geq \dots \geq \lambda_n$, products of elementary symmetric functions $e_{\lambda} = e_{\lambda_1} \dots e_{\lambda_n}$ form a basis for ${\rm Sym}$, see for example \cite{McD}.  Likewise, products of complete symmetric functions also provide a basis for $\sym$.  This mirrors the fact that any closed diagram in the graphical calculus for $\Ucat_Q(\mf{sl}_2)$ can be reduced to a product on non-nested bubbles of a given orientation.

In the $\mf{sl}_n$ calculus of the 2-category $\Ucat(\mf{sl}_n)$,  we have the isomorphism
$$ \psi_{\lambda} \maps \;\; \prod_{i\in I}{\rm Sym}  \;\; \longrightarrow\;\; Z(\l)= \Ucat_Q(\mf{sl}_n)(\onel,\onel) $$
 since any closed diagram can still be reduced to products of non-nested closed bubbles labelled by $i \in I$.

In what follows, it will be interesting to consider which products of closed diagrams correspond to the basis of $\sym$ given by the power sum $p_r$ symmetric functions.
Using a formula that expresses power sum symmetric functions in terms of products of complete and elementary symmetric functions, we can denote by $p_{i,r}(\lambda)$ for $r>0$, the image of the power symmetric polynomial on $i$-labelled strands in $Z(\l)$:
\begin{equation} \label{eq_defpil}
 p_{i,r}(\lambda) := \sum_{a+b=r} (a+1)
 \xy
   (0,-2)*{\icbub{\spadesuit+a}{i}};
   (12,-2)*{\iccbub{\spadesuit+b}{i}};
  (8,8)*{\lambda};
 \endxy
 =  - \sum_{a+b=r} (b+1)
 \xy
   (0,-2)*{\icbub{\spadesuit+a}{i}};
   (12,-2)*{\iccbub{\spadesuit+b}{i}};
  (8,8)*{\lambda};
 \endxy =  -\sum_{a+b=r} a
 \xy
   (0,-2)*{\icbub{\spadesuit+b}{i}};
   (12,-2)*{\iccbub{\spadesuit+a}{i}};
  (8,8)*{\lambda};
 \endxy
\end{equation}
For later convenience we set $p_{i,0}(\l) = \lambda_i$.

\begin{exercise}
Using the bubble slide formulas from exercise~\ref{exercise_bubble-slides}, show that
\[
 p_{i,r}(\lambda) =
\begin{cases} \quad -\;\;
 \text{$\xy
 (-6,0)*{};
  (6,0)*{};
  (15,12)*{\lambda};
  (-4,0)*{}="t1";  (4,0)*{}="t2";
 "t2";"t1" **\crv{(4,6) & (-4,6)}; ?(.02)*\dir{>}
   ?(1)*\dir{>} ?(.3)*\dir{}+(2,2)*{\scs i};
  "t2";"t1" **\crv{(4,-6) & (-4,-6)};
  ?(.7)*\dir{}+(0,0)*{\bullet}+(2,-3)*{\scs \spadesuit +r};
  (-12,0)*{}="t1";  (12,0)*{}="t2";
  "t2";"t1" **\crv{(12,15) & (-12,15)}; ?(0)*\dir{<} ?(1)*\dir{<} ?(.3)*\dir{}+(2,2)*{\scs i};
  "t2";"t1" **\crv{(12,-15) & (-12,-15)};
  ?(.7)*\dir{}+(0,0)*{\bullet}+(2,-4)*{\scs \spadesuit+0};
  \endxy $}
  & \mbox{if } \lambda_i  \geq 0  \\
\quad  \;\;\;\text{$\xy
 (-6,0)*{};
  (6,0)*{};
  (15,12)*{\lambda};
  (-4,0)*{}="t1";  (4,0)*{}="t2";
 "t2";"t1" **\crv{(4,6) & (-4,6)}; ?(0)*\dir{<}
   ?(.96)*\dir{<} ?(.3)*\dir{}+(2,2)*{\scs i};
  "t2";"t1" **\crv{(4,-6) & (-4,-6)};
  ?(.7)*\dir{}+(0,0)*{\bullet}+(2,-3)*{\scs \spadesuit +r};
  (-12,0)*{}="t1";  (12,0)*{}="t2";
  "t2";"t1" **\crv{(12,15) & (-12,15)}; ?(0)*\dir{>} ?(1)*\dir{>} ?(.3)*\dir{}+(2,2)*{\scs i};
  "t2";"t1" **\crv{(12,-15) & (-12,-15)};
  ?(.7)*\dir{}+(0,0)*{\bullet}+(2,-4)*{\scs \spadesuit+0};
  \endxy $}
  & \mbox{if }  \lambda_i  \leq 0,
\end{cases}
\]
so that the formulas defining $p_r(\lambda)$ can be obtained by simplifying a degree $r$ bubble with one orientation nested inside of a degree zero bubble with the opposite orientation.
\end{exercise}

Note that other interesting basis for symmetric functions also have natural graphical analogs.  In \cite{KLMS} diagrams corresponding to the basis of Schur functions are given.

\begin{exercise} \label{exercise_power-slides}
Using the bubble sliding equations from exercise~\ref{exercise_bubble-slides}, prove the following power sum slide rules
\begin{eqnarray}
\label{eq_powerslide1}
\xy 0;/r.18pc/:
  (12,6)*{\lambda-\alpha_j};
  (0,0)*{\bbe{}};
  (0,-12)*{\scs j};
  (14,-2)*{p_{i,r}(\lambda-\alpha_j)};
  (0,6)*{ }+(7,-1)*{\scs  };
 \endxy
 & = &
 \begin{cases}
   \xy 0;/r.18pc/:
  (-10,6)*{\lambda};
  (0,0)*{\bbe{}};
  (2,-7)*{\scs j};
  (-10,-2)*{p_{i,r}(\lambda)};
  (0,6)*{ }+(7,-1)*{\scs  };
 \endxy
 \;\; - \;\; 2\;\; \;\;\xy 0;/r.18pc/:
  (-8,6)*{\lambda};
  (0,0)*{\bbe{}};
  (2,-7)*{\scs j};
  (0,2)*{\bullet}+(3,1)*{r};
  (0,6)*{ }+(7,-1)*{\scs  };
 \endxy & \mbox{if } i = j, \\ & \\
    \xy 0;/r.18pc/:
  (-10,6)*{\lambda};
  (0,0)*{\bbe{}};
  (2,-7)*{\scs j};
  (-10,-2)*{p_{i,r}(\lambda)};
  (0,6)*{ }+(7,-1)*{\scs  };
 \endxy
 & \mbox{if } i \cdot j = 0, \\ & \\
    \xy 0;/r.18pc/:
  (-10,6)*{\lambda};
  (0,0)*{\bbe{}};
  (2,-7)*{\scs j};
  (-10,-2)*{p_{i,r}(\lambda)};
  (0,6)*{ }+(7,-1)*{\scs  };
 \endxy
 \;\; - \;\; (-t_{ij}^{-1}t_{ji})^r\;\;\;\;\xy 0;/r.18pc/:
  (-8,6)*{\lambda};
  (0,0)*{\bbe{}};
  (2,-7)*{\scs j};
  (0,2)*{\bullet}+(3,1)*{r};
  (0,6)*{ }+(7,-1)*{\scs  };
 \endxy
 & \mbox{if } i \cdot j = -1
 \end{cases}
 \\  \label{eq_powerslide2}
   \xy 0;/r.18pc/:
  (10,6)*{\lambda};
  (0,0)*{\bbe{}};
  (2,-7)*{\scs j};
  (-14,-2)*{p_{i,r}(\lambda+\alpha_j)};
  (0,6)*{ }+(7,-1)*{\scs  };
 \endxy
 & = &
 \begin{cases}
\xy 0;/r.18pc/:
  (14,6)*{\lambda};
  (0,0)*{\bbe{}};
  (0,-12)*{\scs j};
  (10,-2)*{p_{i,r}(\lambda)};
  (0,6)*{ }+(7,-1)*{\scs  };
 \endxy
 \;\; + \;\; 2\;\;\xy 0;/r.18pc/:
  (10,6)*{\lambda};
  (0,0)*{\bbe{}};
  (2,-7)*{\scs j};
  (0,2)*{\bullet}+(3,1)*{r};
  (0,6)*{ }+(7,-1)*{\scs  };
 \endxy & \mbox{if } i = j, \\ & \\
\xy 0;/r.18pc/:
  (14,6)*{\lambda};
  (0,0)*{\bbe{}};
  (0,-12)*{\scs j};
  (10,-2)*{p_{i,r}(\lambda)};
  (0,6)*{ }+(7,-1)*{\scs  };
 \endxy
 \;\; + \;\; (-t_{ij}^{-1}t_{ji})^r\;\;\xy 0;/r.18pc/:
  (10,6)*{\lambda};
  (0,0)*{\bbe{}};
  (2,-7)*{\scs j};
  (0,2)*{\bullet}+(3,1)*{r};
  (0,6)*{ }+(7,-1)*{\scs  };
 \endxy
 & \mbox{if } i \cdot j = -1
 \end{cases}
\end{eqnarray}
\end{exercise}

% ---------------------------------------------------------------------
%
\subsection{$\Kar$, $K_0$ and $\Tr$ for 2-categories} \label{sec_traceK0-2cats}
%
% ---------------------------------------------------------------------

We can extend many of the constructions from section \ref{sec_trace_lincat} defined for (additive) linear categories to the 2-categorical setting.  A 2-category $\CC$ is linear if the categories $\CC(x,y)$ are linear for all $x,y\in \Ob(\CC)$ and the composition functor preserves the linear structure (see \cite{Bor} for more details).  Similarly, an additive linear 2-category is a linear 2-category in which the categories $\CC(x,y)$ are also additive and composition is given by an additive functor.

The following definitions extend the Karoubi envelope, split Grothendieck group, and trace to the 2-categorical setting.
\begin{itemize}
\item Given an additive linear 2-category $\CC$, define the split Grothendieck group $K_0(\CC)$ of $\CC$ to be the linear category with $\Ob(K_0(\CC)) = \Ob(\CC)$ and with $K_0(\CC)(x,y) := K_0(\CC(x,y))$
for any two objects $x,y \in \Ob(\CC)$.  For $[f]_{\cong} \in
\Ob(K_0(\CC)(x,y))$ and $[g]_{\cong} \in \Ob(K_0(\CC)(y,z)) $ the
composition in $K_0(\CC)$ is defined by $[g]_{\cong} \circ [f]_{\cong}
:= [g \circ f]_{\cong}$.

\item The trace $\Tr(\CC)$ of a linear 2-category is the linear category with $\Ob(\Tr(\CC)) = \Ob(\CC)$ and with $\Tr(\CC)(x,y) := \Tr(\CC(x,y))$ for any two objects $x,y \in \Ob(\CC)$. For $\sigma \in \End_{\CC(x,y)}$ and
$\tau \in \End_{\CC(y,z)}$, we have
$[\tau ]\circ[\sigma ]=[\tau \circ\sigma ]$.
The identity morphism for $x\in \Ob(\TrCC)=\Ob(\CC)$ is given by
$[1_{I_x}]$.

\item The Karoubi envelope $\Kar(\CC)$ of a linear 2-category $\CC$ is the linear 2-category with $\Ob(\Kar(\CC)) = \Ob(\CC)$ and with hom categories $\Kar(\CC)(x,y) := \Kar(\CC(x,y))$.  The composition functor
$\Kar(\CC)(y,z) \times \Kar(\CC)(x,y) \to \Kar(\CC)(x,z)$ is induced
by the universal property of the Karoubi envelope from the composition
functor in $\CC$.  The fully-faithful additive functors $\CC(x,y) \to
\Kar(\CC(x,y))$ combine to form an additive $2$-functor $\CC \to
\Kar(\CC)$ that is universal with respect to splitting idempotents in
the  Hom-categories $\CC(x,y)$.
\end{itemize}

%Let $X \in \{\Kar, K_0, \Tr \}$. Given a linear 2-category $\CC$, we define $X(\CC)$ as the linear category with  $\Ob(X(\CC)) = \Ob(\CC)$ and $X(\CC)(x,y) := X(\CC(x,y))$
%for any two objects $x,y \in \Ob(\CC)$.  For $[f]_{X} \in \Ob(X(\CC)(x,y))$ and $[g]_{X} \in \Ob(X(\CC)(x,y)) $ the composition in $X(\CC)$ is defined
%by $[g]_{X} \circ [f]_{X} := [g \circ f]_{X}$.
%
%\begin{exercise}
%For $X \in \{ K_0, \Tr \}$ show that $X(\CC)$ is a linear 1-category.
%\end{exercise}

For categorification using $K_0$ it is necessary to consider the 2-category $\Kar(\Ucat(\mf{sl}_n))$.
We will denote by $\Udot$ the Karoubi envelope of $\U$.

The homomorphisms $h_{\CC(x,y)}$ taken over all objects $x,y \in \Ob(\CC)$ give rise to a linear functor
\begin{equation} \label{eq_chernchar}
h_{\CC} \col \K_0(\CC) \to \Tr(\CC)
\end{equation}
which is the identity map on objects and sends $K_0(\CC)(x,y) \to \Tr(\CC)(x,y)$ via the homomorphism $h_{\CC(x,y)}$.  It is easy to see that this assignment preserves composition since
\begin{align}
  h_{\CC}([g]_{\cong} \circ [f]_{\cong}) = h_{\CC}([g \circ f]_{\cong}) = [1_{g \circ f}]
   = [1_g \circ 1_f] = [1_g] \circ [1_f] = h_{\CC}([g]_{\cong}) \circ h_{\CC}([f]_{\cong}).
\end{align}

In the following sections we study the trace and $K_0$ for the 2-categories $\UcatD(\mf{sl}_n)$.

% ---------------------------------------------------------------------
%
\subsection{Two categorifications of Lusztig's integral form $\UA(\mathfrak{sl}_2)$}
%
% ---------------------------------------------------------------------
Here we present the results of \cite{Lau1,KLMS} computing $K_0$ and
of \cite{BHLZ} computing $\Tr$ of $\Ucat(\mf{sl}_2)$.

The algebra $\UA(\mf{sl}_2)$ has a well-known basis, called
Lusztig's canonical basis $\Bbb B$ consisting of
\begin{enumerate}[(i)]
     \item $E^{(a)}F^{(b)}1_{n} \quad $ for $a,b\geq 0$,
     $n\in\Z$, $n\leq b-a$,
     \item $F^{(b)}E^{(a)}1_{n} \quad$ for $a$,$b\geq 0$, $n\in\Z$,
     $n\geq
     b-a$,
\end{enumerate}
where $E^{(a)}F^{(b)}1_{b-a}=F^{(b)}E^{(a)}1_{b-a}$.
Let $_m\Bbb B_n$ be set of elements in $\Bbb B$ belonging to $1_m(\UA)1_n$.
%The defining relations for the algebra $\UA$ are
%\begin{equation}
% \label{lubp}
% E^{(a)}E^{(b)}1_n=\left[ \begin{array}{c}a+b\\b\end{array}\right]E^{(a+b)}1_n,
%\end{equation}
%\begin{equation}
% F^{(a)}F^{(b)}1_n=\left[ \begin{array}{c}a+b\\b\end{array}\right]F^{(a+b)}1_n,
%\end{equation}
%\begin{equation}
% E^{(a)}F^{(b)}1_n=\sum_{j=0}^{\min(a,b)}\left[ \begin{array}{c}a-b+n\\j\end{array}\right]F^{(b-j)}E^{(a-j)}1_n,\quad\text{for }n\geq b-a,
%\end{equation}
%\begin{equation}
% \label{lubz}
% F^{(b)}E^{(a)}1_n=\sum_{j=0}^{\min(a,b)}\left[ \begin{array}{c}b-a-n\\j\end{array}\right]E^{(a-j)}F^{(b-j)}1_n,\quad\text{for }n\leq b-a \, .
%\end{equation}

%all lift to explicit isomorphisms in $\UcatD$ (see \cite[Theorem 5.1, Theorem 5.9]{KLMS})
We associate to
 each $x \in \Bbb B$  a $1$-morphism in $\UcatD$ as follows:
\begin{equation} \label{eq_basis}
  x \mapsto \cal{E}(x) := \left\{
\begin{array}{cl}
  \cal{E}^{(a)}\cal{F}^{(b)}\onen & \text{if $x=E^{(a)}F^{(b)}1_n$,} \nn \\
  \cal{F}^{(b)}\cal{E}^{(a)}\onen & \text{if $x=F^{(b)}E^{(a)}1_n$.} \nn
\end{array}
  \right.
\end{equation}
Let $\B=\{\E(x)\mid x\in\Bbb B\}$
and $_m\B_n =\{\E(x)\mid x\in\, _m{\Bbb B_n}\}$.

The first categorification of $\UA$ was given in \cite{Lau1} where the decategorification functor is given by $K_0$.  The original proof worked with the 2-category $\UcatD(\mf{sl}_2)$ defined over a field $\k$, so that spaces of 2-morphisms formed vector spaces and the Krull-Schmidt property could be utilized.  In rank one the choice of scalars $Q$ does not enter the definition.  Later this result was enhanced in \cite{KLMS} to prove the result when working over the integers.   Unique decomposition for spaces of 2-morphisms follows as a corollary of these results, implying that the category $\U_{\Z}(\mf{sl}_2)$ is Krull-Schmidt. Very recently, it was shown in \cite{BHLZ} that $\Udot_\Z$ also categorifies $\UA$ via the alternative trace decategorification functor $\Tr$.

\begin{thm}\label{K0}
(\cite{KLMS}, \cite{BHLZ}) For ${\fsl}_2$,
$\B$ is the strongly upper triangular basis of $\Udot_\Z$.
The maps
\[ \begin{array}{rrr}
\gamma\maps  \UA & \longrightarrow& K_0(\UcatD_\Z)
\\
   x & \mapsto & [\cal{E}(x)]_{\cong} \end{array}\quad{\text{and}}\quad
\begin{array}{rrr}
  h_{\Udot_\Z}\maps K_0(\Udot_\Z)&\longrightarrow& \Tr(\U_\Z)\\
 x &\mapsto & [1_x] \\
\end{array}\]
are isomorphisms of graded linear categories, providing two categorifications of $\UA$.
\end{thm}

% - - - - - - - - - - - - - - - - - - - - - - - - - - - - - - - - -
%
\subsubsection{Strategy of the Proof}
%
% - - - - - - - - - - - - - - - - - - - - - - - - - - - - - - - - -

In section 6 of \cite{BHLZ}, based on the results of  \cite{KLMS} we give an
explicit presentation of the subcategory $\Udot(n,m)_\Z|_{_m\B_n}$
by generators and relations, which shows that this category is strongly upper triangular.
Hence, the indecomposable 1-morphisms from $\B$ generates $K_0$ and, by Proposition \ref{r24}, also $\Tr$ as $\Z[q,q^{-1}]$-modules. In \cite{KLMS} explicit isomorphisms in $\Udot_\Z$ are given lifting the
the defining relations of $\UA$, showing that $\gamma$ is an isomorphism of graded categories. Since the functor $h_{\Udot}: K_0(\Udot_\Z) \to \Tr(\Udot_\Z)$ preserves composition.
Finally, we use that $\Tr(\Udot)=\Tr(\U)$, the result follows.

% ---------------------------------------------------------------------
%
\subsection{Two categorifications of $\UA ({\mathfrak sl}_n)$}
%
% ---------------------------------------------------------------------
Here we present the results of  \cite{KL3} computing $K_0$ and
of \cite{BGHL} computing $\Tr$ for general $n$.

There is a special choice of scalars, where all $t_{i\, j}=1$  for $1\leq i< n-1$.  In this case, the graphical calculus is cyclic with respect to the biadjoint structure, so that isotopic diagrams represent the same 2-morphism.   This will be our default choice of $Q$ in sections unless explicitly stated otherwise. We will omit explicitly mentioning $Q$ in the notation when no confusion is likely to arise.

Theorem \ref{K0} was possible because the explicit isomorphisms constructed in \cite{KLMS} provide decompositions of arbitrary 1-morphisms into indecomposables.   However, the lack of an explicit description of the canonical basis for $\UA$ beyond $\mf{sl}_2$ presents an obstacle to generalizing this result beyond $\mf{sl}_2$.   However, if one works with the 2-category $\Ucat(\mf{sl}_n)$ defined over a field $\k$ of characteristic 0, then one can utilize various connections between this category and known geometric constructions to prove the following result.

\begin{thm}\label{K0n}
(\cite{KL3}, \cite{BGHL}) Let $\k$ be a field of characteristic $0$.
For ${\mathfrak {sl}}_n$,
the 2-category $\Udot$
is strongly upper triangular (see Remark~\ref{rem-overk}).
Moreover,
$$K_0(\Udot)=\UA$$
as graded linear categories and
$$\Tr(\U)=K_0(\Udot)\otimes_\Z \k, \quad
\HH_i(\U)=0 \quad{\text {for}}\quad i>0.$$
\end{thm}

%The  $\k$-linear category $K_0(\Udot)\otimes \k$ is
%$K_0(\Udot)$ enriched over $\cat{Vect}_\k$.
Following a recommendation
of the first author, the $\fsl_3$ case of this theorem was studied in
\cite{Marko}.

% - - - - - - - - - - - - - - - - - - - - - - - - - - - - - - - - -
%
\subsubsection{Strategy of the Proof}
%
% - - - - - - - - - - - - - - - - - - - - - - - - - - - - - - - - -
Let us make few remarks about the proof.
The result about $K_0$ was proved in \cite{KL3}. It relies on the fact that
$\Udot$ is Krull-Schmidt.  To control the size of $K_0$ an action of $\Udot$ was defined on the cohomology rings of partial flag varieties.
Even when for $n>2$, the explicit form of the Lusztig canonical basis is not known,
but the results of \cite{VV} and \cite{Web4} allows us to deduce that
$\Udot$ is strongly upper triangular.
The full argumentation will be given in \cite{BGHL}.

% - - - - - - - - - - - - - - - - - - - - - - - - - - - - - - - - -
%
\subsubsection{Applications}
%
% - - - - - - - - - - - - - - - - - - - - - - - - - - - - - - - - -

In \cite{LQR,QR} it was shown that $\mf{sl}_n$ link homology theories can be obtained
from $\Udot(\fsl_m)$ by means of the so-called foamation functors $F$, first studied in \cite{MackFoam}.
More precisely, there is an extended version $2{\rm BFoam}$ of the Bar-Natan foam category introduced by Blanchet~\cite{Blan}.   (The advantage of Blanchet's version is that it
is functorial with respect to link cobordisms, when Bar-Natan version is functorial up to sign.)
In \cite{LQR,QR} an analog $n{\rm BFoam}$ of the Blanchet foam category for $\mf{sl}_2$ is defined for $\mf{sl}_n$ and families of 2-representations $F$ for were defined from $\Udot(\fsl_m)$ into the foam categories $n{\rm BFoam}$.
%These functors were also extended to other versions of $\fsl_2$ and $\fsl_3$ foams.

A general procedure was laid out by Cautis ~\cite{Cautis-clasp} using categorical skew Howe duality to pull back complexes associated to a tangle to a complex in $\Kom(\Udot)$.  Combining this with the foamation functors allows one to pull back the complexes of foams associated to a tangle to a complex in $\Kom(\Udot)$.   It was shown that all of the foam relations used for $\mf{sl}_n$-link homology arise from relations in $\Udot(\fsl_m)$ via foamation functors.  These foamation functors categorify a skew Howe functor defined at the decategorified level by Cautis, Kamnitzer, and Morrison that realize web relations from relations in the quantum group $\dot{\mathbf{U}}_q(\mf{sl}_m)$~\cite{CKM}.

Recall that we have two flavors of Euler characteristic associated with $K_0$ and $\Tr$,
\begin{eqnarray*}
\chi_{\tr}: &\Kom(F(\U)) \to \Tr(F(\U))\\
\chi: &\Kom(F(\Udot)) \to K_0(F(\Udot))  .
\end{eqnarray*}
Below is one of the applications of Theorem \ref{K0n}.

\begin{cor}\label{chi}
In $\mf{sl}_n$ link homologies theories, we have
$$\chi_{tr}=\chi\, .$$
\end{cor}

\begin{remark}
The $\mf{sl}_2$ case of  Corollary \ref{chi} was first proven  by Cooper-Krushkal~\cite[Section 3]{CK2}.
\end{remark}

\begin{proof}
%Recall that link homology complexes belong to image of the foamation functor $\F$ (reference?).
Let $\C:=F(\Udot(\fsl_m))$. We have the Chern character map
$h_\C: K_0(\C) \to \Tr(\C)$ sending $[x]_{\cong} \to [1_x]$. The induced  map $\chi \to \chi_{\tr}$ is surjective,
since $\chi_{\tr}$ uses trace classes of the identity maps only.
It remains to show that this map is injective.
Note that $K_0(\Udot)$ is freely generated by indecomposables and $h_{\Udot}$ is an isomorphism after tensoring with any field of zero characteristic by Theorem \ref{K0n}, so
 $h_{\Udot}$ is injective. Hence, $F(K_0(\Udot))=F(h_{\Udot}(K_0(\Udot)))$. The fact that any linear functor commutes with $K_0$ implies the result.
\end{proof}

% =====================================================================
%
\section{Graded traces of categorified quantum groups}
%
% =====================================================================

By allowing homogeneous, but not necessarily grading preserving 2-morphisms in $\Ucat(\mf{g})$ we can define a version of the 2-category with larger 2-hom spaces.  Let $\Ucat^{\ast}(\mf{g})$ denote the 2-category with the same objects and the same 1-morphisms as $\Ucat(\mf{g})$, but with 2-hom spaces between one morphisms $X \onel$ and $Y\onel$ given by
\[
 \Ucat^{\ast}(\mf{g})(X\onel, Y\onel) := \bigoplus_{t\in \Z} \Ucat(\mf{g})(X\onel, Y\onel \la t \ra).
\]
One can alternatively think of $\Ucat^{\ast}$ as the result of adding isomorphisms $X\onel \to X\onel \la t\ra $ for all $t\in \Z$.  This essentially kills the grading making Grothendieck ring $K_0(\Ucat^{\ast})$ only a $\Z$-module, rather than a $\Z[q,q^{-1}]$-module.  While this version of the 2-category is less interesting from a $K_0$ perspective, it has interesting consequences for the trace.

% ---------------------------------------------------------------------
%
\subsection{The trace of 2-category $\U^\ast(\fsl_2)$ and the current algebra of $\mf{sl}_2$}
%
% ---------------------------------------------------------------------

In this section we outline the results of \cite{BHLZ}. We work with $\g=\fsl_2$ and our 2-morphisms are defined over $\Z$.

Let us consider the  algebra
$\fsl_2[t]=\fsl_2\otimes \modQ [t]$.
% of the loop Lie algebra
%$L\fsl_2=\fsl_2\otimes \modQ [t,t^{-1}]$.
As a $\modQ $-algebra, its
universal enveloping algebra $\mathbf U (\fsl_2[t])$ is generated by $E_i,F_i$ and $H_i$
for $i\geq 0$, where $X_i=X\otimes t^i$, subject to the following
relations:
\begin{gather*}
  [H_i,H_j]=[E_i,E_j]=[F_i,F_j]=0,\\
  [H_i,E_j]=2E_{i+j},\quad [H_i,F_j]=-2F_{i+j},\quad [E_i,F_j]=H_{i+j}.\label{rel2}
\end{gather*}
 An integral basis of $\mathbf U (\fsl_2[t])$
was constructed in \cite{Gar} and it is the lifting of the
integral Poincare-Birkhoff-Witt (PBW) type basis.
The following is one of the main results of \cite{BHLZ}.

\begin{thm} \label{qqq}
There is an isomorphism of additive categories
\[
 h \maps \dot{\bfU}(\fsl_2[t]) \longrightarrow \Tr(\U^\ast_\Z(\fsl_2)),
\]
where $\dot{\bfU}(\fsl_2[t])$ is the idempotented integral form  of the current algebra for $\mf{sl}_2$.
\end{thm}

% - - - - - - - - - - - - - - - - - - - - - - - - - - - - - - - - -
%
\subsubsection{Sketch of the Proof}
%
% - - - - - - - - - - - - - - - - - - - - - - - - - - - - - - - - -

To define an isomorphism $h \maps \dot {\bfU}(\fsl_2[t]) \longrightarrow \Tr(\U^\ast_\Z(\fsl_2))$ we must specify where the generating elements of $ \dot{\bfU}(\fsl_2[t])$ are sent.
 To define this map over the integers as was done in \cite{BHLZ} we must work with the integral basis of the current algebra introduced by Garland.  However, to illustrate the central ideas of this theorem we will explain which elements
of the trace correspond to the current algebra generators $E_r=E\otimes t^r$, $F_s=F \otimes t^s$, and $H_k = H \otimes t^k$.  We illustrate these assignments below:
\[
E_r 1_{\lambda} \mapsto
\left[
\xy 0;/r.18pc/:
  (10,6)*{\lambda};
  (0,0)*{\bbe{}};
  (0,2)*{\bullet}+(3,1)*{r};
  (-8,6)*{ }; (12,-7)*{};
 \endxy
\right]
 \qquad \quad
F_s 1_{\lambda}\mapsto
\left[
\xy 0;/r.18pc/:
  (10,6)*{\lambda};
  (0,0)*{\bbf{}};
  %(2,-7)*{\scs j};
  (0,2)*{\bullet}+(3,1)*{s};
   (-8,6)*{ }; (12,-7)*{};
 \endxy
\right]
 \qquad \quad
H_r 1_{\lambda} \mapsto
\left[
 p_r(\lambda)\Id_{\onel}
\right]
\]
where $p_r(\lambda)$ is the degree $r$ endomorphism of $\onel$ given by the power sum symmetric function defined in section~\ref{sec_sym-bub} and  $p_0(\l)$ is  $\l$.

In order to verify that these assignments define a homomorphism we must verify the current algebra axioms.  Since our target is the trace of the 2-category $\U^\ast_\Z(\fsl_2)$, we perform these calculations in the graphical calculus as though each diagram was on an annulus.  For example, the most interesting relation $[E_r,F_s]=H_{r+s}$ follows from the computation achieved by twice utilizing the most sophisticated $sl(2)$-relations \eqref{eq_ident_decomp-ngeqz} and  \eqref{eq_ident_decomp-nleqz}:
\begin{equation}
 E_rF_s 1_{\lambda}
 \;\; = \;\;
  \xy 0;/r.18pc/:
    (8,-1)*{
 \hackcenter{\begin{tikzpicture}
  \path[draw,blue, very thick, fill=blue!10]
   (-1.4,-.6) to (-1.4,.6) .. controls ++(0,1.15) and ++(0,1.15) ..
   (1.4,.6)
   to (1.4,-.6)  .. controls ++(0,-1.15) and ++(0,-1.15) ..
   (-1.4,-.6);
    \path[draw, blue, very thick, fill=white]
    (-0.15,0) .. controls ++(0,.25) and ++(0,.25) .. (0.15,0)
            .. controls ++(0,-.25) and ++(0,-.25) .. (-0.15,0);
\end{tikzpicture}}
 };(8,4)*{\lambda};
  (8,8)*{\lcap};
 (8,6)*{\scap};
 (8,-6)*{\scupfe};
 (-4,0)*{\sline};
 (4,0)*{\sline};
 (12,0)*{\sline};
 (20,0)*{\sline};
 (8,-10)*{\lcupef};
 (-4,3)*{\bullet}+(-2.5,1)*{\scs r};
 (4.2,3)*{\bullet}+(-2.5,1)*{\scs s};
 %(15,-5)*{\scs i};
%  (22,-8)*{\scs i};
 \endxy
  \;\; \refequal{\eqref{eq_ident_decomp-ngeqz}} \;\; -\;
    \xy 0;/r.18pc/:
     (0,-1)*{
 \hackcenter{\begin{tikzpicture}
  \path[draw,blue, very thick, fill=blue!10]
   (-1.4,-.6) to (-1.4,.6) .. controls ++(0,1.55) and ++(0,1.55) ..
   (1.4,.6)
   to (1.4,-.6)  .. controls ++(0,-1.55) and ++(0,-1.55) ..
   (-1.4,-.6);
    \path[draw, blue, very thick, fill=white]
    (-0.15,0) .. controls ++(0,.25) and ++(0,.25) .. (0.15,0)
            .. controls ++(0,-.25) and ++(0,-.25) .. (-0.15,0);
\end{tikzpicture}}
 };(1,6)*{\lambda};(-8,-4)*{\fecross};(-8,4)*{\naecross};
  (4,0)*{\mline};(12,0)*{\mline};(0,10)*{\scap};(0,12)*{\lcap};
  (0,-10)*{\scupfe};(0,-14)*{\lcupef};
  %(6,-3)*{\scs i};(15,-3)*{\scs i};
   (-12,8)*{\bullet}+(-2.5,1)*{\scs r};
 (-4,8)*{\bullet}+(-2.5,1)*{\scs s};
  \endxy
 \;\; + \;\;
  \sum_{ \xy  (0,3)*{\scs f_1+f_2+f_3}; (0,0)*{\scs =\lambda-1};\endxy} \;
  \xy 0;/r.18pc/:
       (-2,-1)*{
 \hackcenter{\begin{tikzpicture}
  \path[draw,blue, very thick, fill=blue!10]
   (-1.5,-.8) to (-1.5,.8) .. controls ++(0,2) and ++(0,2) ..
   (1.5,.8)
   to (1.5,-.8)  .. controls ++(0,-2) and ++(0,-2) ..
   (-1.5,-.8);
    \path[draw, blue, very thick, fill=white]
    (-0.15,0) .. controls ++(0,.25) and ++(0,.25) .. (0.15,0)
            .. controls ++(0,-.25) and ++(0,-.25) .. (-0.15,0);
\end{tikzpicture}}
 };
  (-1,7)*{\lambda};
  (-10,0)*{\ccbub{\spadesuit+f_2}{}};
  (-10,14)*{\scup};(-10,-14)*{\scap};
  (2,-8)*{\meline};(2,8)*{\mline};
  (10,-8)*{\mfline};(10,8)*{\mline};
  (-2,18)*{\scap};(-2,20)*{\lcap};
 (-2,-18)*{\scup};(-2,-22)*{\lcup};
 (-13,13.7)*{\bullet}+(-0,-3)*{\scs f_1};
 (-13,-13.7)*{\bullet}+(-2.5,3.3)*{\scs f_3};
 (-14,18)*{\bullet}+(-2.5,1)*{\scs r};
 (-6,18)*{\bullet}+(-2.5,1)*{\scs s};
  \endxy \nn
\end{equation}
\begin{equation}
  \;\; = \;\; -\; \vcenter{
    \xy 0;/r.18pc/:
      (0,-6)*{
 \hackcenter{\begin{tikzpicture}
  \path[draw,blue, very thick, fill=blue!10]
   (-1.4,-.6) to (-1.4,.6) .. controls ++(0,1.15) and ++(0,1.15) ..
   (1.4,.6)
   to (1.4,-.6)  .. controls ++(0,-1.15) and ++(0,-1.15) ..
   (-1.4,-.6);
    \path[draw, blue, very thick, fill=white]
    (-0.15,0) .. controls ++(0,.25) and ++(0,.25) .. (0.15,0)
            .. controls ++(0,-.25) and ++(0,-.25) .. (-0.15,0);
\end{tikzpicture}}
 };(0,0)*{\lambda};
  (-8,-4)*{\naecross};(8,-4)*{\fecross};
  (0,2)*{\scap};(0,4)*{\lcap};
  (0,-10)*{\scup};(0,-14)*{\lcup};
   (-12,-8)*{\bullet}+(-2.5,1)*{\scs r};
 (-4,-8)*{\bullet}+(-2.5,1)*{\scs s};
  \endxy}
 \;\; + \;\; \sum_{f_2=0}^{\lambda-1}\sum_{f_1=0}^{\lambda-1-f_2}
  \;
 \vcenter{ \xy 0;/r.18pc/:
        (4,4)*{
 \hackcenter{\begin{tikzpicture}
  \path[draw,blue, very thick, fill=blue!10]
   (-1.5,-.6) to (-1.5,.6) .. controls ++(0,1.15) and ++(0,1.15) ..
   (1.5,.6)
   to (1.5,-.6)  .. controls ++(0,-1.15) and ++(0,-1.15) ..
   (-1.5,-.6);
    \path[draw, blue, very thick, fill=white]
    (-0.15,0) .. controls ++(0,.25) and ++(0,.25) .. (0.15,0)
            .. controls ++(0,-.25) and ++(0,-.25) .. (-0.15,0);
\end{tikzpicture}}
 };(6,15)*{\lambda};
  (-8,6)*{\ccbub{\spadesuit+f_2}{}};
  (6,-6)*{\cbub{\spadesuit+r+s-f_2}{}};
  \endxy} \nn
 \end{equation}
\begin{equation}
  \;\; = \;\;
 \;\; -\;
    \xy 0;/r.18pc/:
      (0,-1)*{
 \hackcenter{\begin{tikzpicture}
  \path[draw,blue, very thick, fill=blue!10]
   (-1.4,-.6) to (-1.4,.6) .. controls ++(0,1.55) and ++(0,1.55) ..
   (1.4,.6)
   to (1.4,-.6)  .. controls ++(0,-1.55) and ++(0,-1.55) ..
   (-1.4,-.6);
    \path[draw, blue, very thick, fill=white]
    (-0.15,0) .. controls ++(0,.25) and ++(0,.25) .. (0.15,0)
            .. controls ++(0,-.25) and ++(0,-.25) .. (-0.15,0);
\end{tikzpicture}}
 };(1,6)*{\lambda};(-8,-4)*{\naecross};(-8,4)*{\fecross};
  (4,0)*{\mline};(12,0)*{\mline};(0,10)*{\scap};(0,12)*{\lcap};
  (0,-10)*{\scupef};(0,-15)*{\lcupfe};
  (6,-3)*{\scs i};(15,-3)*{\scs i};
     (-12,0)*{\bullet}+(-2.5,1)*{\scs r};
 (-4,0)*{\bullet}+(-2.5,1)*{\scs s};
  \endxy
  \;\;+ \;\; \sum_{f_2=0}^{\lambda-1}(\lambda-f_2)
  \;
   \vcenter{ \xy 0;/r.18pc/:
        (4,4)*{
 \hackcenter{\begin{tikzpicture}
  \path[draw,blue, very thick, fill=blue!10]
   (-1.5,-.6) to (-1.5,.6) .. controls ++(0,1.15) and ++(0,1.15) ..
   (1.5,.6)
   to (1.5,-.6)  .. controls ++(0,-1.15) and ++(0,-1.15) ..
   (-1.5,-.6);
    \path[draw, blue, very thick, fill=white]
    (-0.15,0) .. controls ++(0,.25) and ++(0,.25) .. (0.15,0)
            .. controls ++(0,-.25) and ++(0,-.25) .. (-0.15,0);
\end{tikzpicture}}
 };(6,15)*{\lambda};
  (-8,6)*{\ccbub{\spadesuit+f_2}{}};
  (6,-6)*{\cbub{\spadesuit+r+s-f_2}{}};
  \endxy}\nn
\end{equation}
Using the nilHecke relations and the identity decomposition equation it is not hard to check that
\begin{equation}
 \;\; -\;
    \xy 0;/r.18pc/:
       (0,-1)*{
 \hackcenter{\begin{tikzpicture}
  \path[draw,blue, very thick, fill=blue!10]
   (-1.4,-.6) to (-1.4,.6) .. controls ++(0,1.55) and ++(0,1.55) ..
   (1.4,.6)
   to (1.4,-.6)  .. controls ++(0,-1.55) and ++(0,-1.55) ..
   (-1.4,-.6);
    \path[draw, blue, very thick, fill=white]
    (-0.15,0) .. controls ++(0,.25) and ++(0,.25) .. (0.15,0)
            .. controls ++(0,-.25) and ++(0,-.25) .. (-0.15,0);
\end{tikzpicture}}
 };(1,6)*{\lambda};(-8,-4)*{\naecross};(-8,4)*{\fecross};
  (4,0)*{\mline};(12,0)*{\mline};(0,10)*{\scap};(0,12)*{\lcap};
  (0,-10)*{\scupef};(0,-15)*{\lcupfe};
 % (6,-3)*{\scs i};(15,-3)*{\scs i};
     (-12,0)*{\bullet}+(-2.5,1)*{\scs r};
 (-4,0)*{\bullet}+(-2.5,1)*{\scs s};
  \endxy \;\; = \;\; F_s E_r 1_{\lambda}
 \;\;- \sum_{j=0}^{-\lambda-1+a+b}(-\lambda+r+s-j)
    \xy 0;/r.18pc/:
          (0,4)*{
 \hackcenter{\begin{tikzpicture}
  \path[draw,blue, very thick, fill=blue!10]
   (-1.5,-.6) to (-1.5,.6) .. controls ++(0,1.15) and ++(0,1.15) ..
   (1.5,.6)
   to (1.5,-.6)  .. controls ++(0,-1.15) and ++(0,-1.15) ..
   (-1.5,-.6);
    \path[draw, blue, very thick, fill=white]
    (-0.15,0) .. controls ++(0,.25) and ++(0,.25) .. (0.15,0)
            .. controls ++(0,-.25) and ++(0,-.25) .. (-0.15,0);
\end{tikzpicture}}
 };(4,11)*{\lambda};
  (-8,0)*{\ccbub{\spadesuit+r+s-j}{}};
  (12,0)*{\cbub{\spadesuit+j}{}};
  \endxy
\end{equation}
Therefore it suffices to show that the action of $H_{r+s}$ is equal to
\begin{equation} \label{eq_shoulbBh}
\sum_{f=0}^{\lambda-1}(\lambda-f)
  \;
  \xy 0;/r.18pc/:
            (2,4)*{
 \hackcenter{\begin{tikzpicture}
  \path[draw,blue, very thick, fill=blue!10]
   (-1.5,-.6) to (-1.5,.6) .. controls ++(0,1.15) and ++(0,1.15) ..
   (1.5,.6)
   to (1.5,-.6)  .. controls ++(0,-1.15) and ++(0,-1.15) ..
   (-1.5,-.6);
    \path[draw, blue, very thick, fill=white]
    (-0.15,0) .. controls ++(0,.25) and ++(0,.25) .. (0.15,0)
            .. controls ++(0,-.25) and ++(0,-.25) .. (-0.15,0);
\end{tikzpicture}}
 };(4,11)*{\lambda};
  (-8,0)*{\ccbub{\spadesuit+f}{}};
  (12,0)*{\cbub{\spadesuit+r+s-f}{}};
  \endxy
  \;\;- \sum_{j=0}^{-\lambda-1+a+b}(-\lambda+a+b-j)
    \xy 0;/r.18pc/:
            (0,4)*{
 \hackcenter{\begin{tikzpicture}
  \path[draw,blue, very thick, fill=blue!10]
   (-1.5,-.6) to (-1.5,.6) .. controls ++(0,1.15) and ++(0,1.15) ..
   (1.5,.6)
   to (1.5,-.6)  .. controls ++(0,-1.15) and ++(0,-1.15) ..
   (-1.5,-.6);
    \path[draw, blue, very thick, fill=white]
    (-0.15,0) .. controls ++(0,.25) and ++(0,.25) .. (0.15,0)
            .. controls ++(0,-.25) and ++(0,-.25) .. (-0.15,0);
\end{tikzpicture}}
 };(4,11)*{\lambda};
  (-8,0)*{\ccbub{\spadesuit+r+s-j}{}};
  (12,0)*{\cbub{\spadesuit+j}{}};
  \endxy
\end{equation}
The first summation is zero when $\lambda-1<0$ and the second summation is zero when $-\lambda-1+a+b<0$.

We prove the claim by considering four cases:
\begin{itemize}
  \item Case 1: $r+s = 0$.
  \item Case 2: $r+s>0$, $\lambda-1\geq 0$ and $-\lambda-1+r+s<0$.
  \item Case 3: $r+s>0$, $\lambda-1<0$ and $-\lambda-1+r+s\geq0$.
  \item Case 4: $r+s>0$, $\lambda-1\geq 0$ and $-\lambda-1+r+s \geq0$.
\end{itemize}
Case 1 is easy to verify.  For cases 2-4 we show that $\eqref{eq_shoulbBh}$ is equal to
\begin{equation} \label{eq_expandP}
  - \sum_{j =0}^{r+s}(r+s+1-j)     \xy 0;/r.18pc/:
            (0,4)*{
 \hackcenter{\begin{tikzpicture}
  \path[draw,blue, very thick, fill=blue!10]
   (-1.5,-.6) to (-1.5,.6) .. controls ++(0,1.15) and ++(0,1.15) ..
   (1.5,.6)
   to (1.5,-.6)  .. controls ++(0,-1.15) and ++(0,-1.15) ..
   (-1.5,-.6);
    \path[draw, blue, very thick, fill=white]
    (-0.15,0) .. controls ++(0,.25) and ++(0,.25) .. (0.15,0)
            .. controls ++(0,-.25) and ++(0,-.25) .. (-0.15,0);
\end{tikzpicture}}
 };(4,11)*{\lambda};
  (-8,0)*{\ccbub{\spadesuit+r+s-j}{}};
  (12,0)*{\cbub{\spadesuit+j}{}};
  \endxy
\end{equation}
by adding 0 in the form of the infinite Grassmannian relation \eqref{eq_infinite_Grass}.
For case 2 ,since $\lambda \geq 1$ we can add zero in the form
\begin{equation}
  -(a+b+1-\lambda)\sum_{f=0}^{\lambda} \xy 0;/r.18pc/:
 (4,11)*{\lambda};
  (-8,0)*{\ccbub{\spadesuit+f}{}};
  (12,0)*{\cbub{\spadesuit+r+s-f}{}};
  \endxy
\end{equation}
to \eqref{eq_shoulbBh} and get \eqref{eq_expandP}.  For case 3, we add
\begin{equation}
  -(\lambda+1)\sum_{j=0}^{-\lambda+r+s} \xy 0;/r.18pc/:
 (4,11)*{\lambda};
  (-8,0)*{\ccbub{\spadesuit+r+s-j}{}};
  (12,0)*{\cbub{\spadesuit+j}{}};
  \endxy
  \end{equation}
and for case 4 we re-index the first summation setting $j=r+s-f$, combine the terms, and add
\begin{equation}
  -(\lambda+1)\sum_{j=0}^{r+s} \xy 0;/r.18pc/:
 (4,11)*{\lambda};
  (-8,0)*{\ccbub{\spadesuit+r+s-j}{}};
  (12,0)*{\cbub{\spadesuit+j}{}};
  \endxy
  \end{equation}

Once we have defined a homomorphism $h \maps \dot {\bfU}(\fsl_2[t]) \longrightarrow \Tr(\U^\ast_\Z(\fsl_2))$, establishing that the map is an isomorphism utilizes the following Lemma.

\begin{lemma}
\label{general}
Let $\C$ be a linear category. Let $H:=\bigoplus_{x\in\Ob(\C)}\C(x,x)$,
and  let $K\subset H$ be a subgroup.
Assume that there is a linear map
 $\pi:H\rightarrow K$  with the following properties:
\begin{enumerate}
\item $\pi$ is a projection,
\item $[\pi(f)]=[f]$ for every $f\in H$, and
\item for every $g\in\C(x,y)$ and $h\in\C(y,x)$ ($x,y\in\Ob(\C)$), $\pi(gh)=\pi(hg)$;
\end{enumerate}
then $\Tr(\C)$ is isomorphic to $K$.
\end{lemma}

The two main ingredients in establishing that $h$ is a homomorphism  are
\begin{itemize}
\item
 identification of a certain subset $K$ of 2-endomorphisms in $\U$ with the Garland integral form of the positive part of the loop algebra (or the current algebra),
\item
construction of the projection $\pi$ satisfying conditions of the lemma.
\end{itemize}

To solve the first problem, we explicitly construct generators of
 $\Tr(\U^*)$   and compute  relations between them. Then we  identify this  PBW type basis
 with Garland's integral form.

% ---------------------------------------------------------------------
%
\subsection{The current algebra for $\mf{sl}_n$} \label{current-sln}
%
% ---------------------------------------------------------------------

We expect to be able to generalize Theorem~\ref{qqq} to $\fsl_n$.

Let us define the current algebra
${\bf U}(\fsl_n[t])={\bf U}(\fsl_n\otimes\Q[t])$. This algebra is generated over $\k$ by
$x^+_{i,r}$, $x^-_{i,s}$ and $\xi_{i,k}$ for $r,s,k\in \N\cup\{0\}$ and $1\leq i<n$
modulo the following relations:
\begin{description}
  \item[C1] For $i,j$ in $I$ and $r,s \in \N\cup\{0\}$
\[
[ \xi_{i,r}, \xi_{j,s} ] = 0
\]

  \item[C2] For any $i,j \in I$ and $k \in \N\cup\{0\}$,
\[
 [ \xi_{i,0}, x^{\pm}_{j,k} ] = \pm a_{ij}x^{\pm}_{j,k}
\]

  \item[C3] For any $i,j \in I$ and $r,k \in \N\cup\{0\}$,
\[
 [\xi_{i,r}, x^{\pm}_{j,k} ] = \pm a_{ij} x^{\pm}_{j, r+k}
\]

  \item[C4] For $i,j \in I$ and $k, \ell \in \N\cup\{0\}$
\[
[x^\pm_{i, k+1}, x^\pm_{j,\ell}]=[x^\pm_{i,k}, x^\pm_{j,\ell+1}]
\]
  \item[C5] For $i,j \in I$ and $k,\ell \in \N\cup\{0\}$
\[
 [x^+_{i,k}, x^-_{j,\ell}] = \delta_{i,j} \xi_{i, k+\ell}
\]
  \item[C6] Let $i \neq j$ and set $m=1-a_{ij}$.  For every $k_1, \dots, k_m \in \N\cup\{0\}$ and $\ell \in \N\cup\{0\}$
\[
 \sum_{\pi \in S_m} \sum_{s=0}^{m} (-1)^s
 \binom{m}{s} x^{\pm}_{i, k_{\pi(1)}} \dots x^{\pm}_{i, k_{\pi(s)}}
 x^{\pm}_{j, \ell} x^{\pm}_{i, k_{\pi(s+1)}} \dots x^{\pm}_{i, k_{\pi(m)}} = 0.
\]
\end{description}
For more on the relationship between the current algebra and specializations of the Yangian see \cite{Bag}.
Recall that we denote by $\dot{\bfU}(\fsl_n[t])$ the idempotented form of the current algebra.

\begin{exercise}
Show for $t=1$ that the current algebra relations restrict to the defining relations
for ${\bf U}(\fsl_n)$.
\end{exercise}

The next proposition is
 a step towards generalizing Theorem~\ref{qqq}.

\begin{proposition}\label{prop-cur}
Let \(\mathsf{E}_{i,r}, \mathsf{F}_{j,s}, \mathsf{H}_{i,r}\) denote  the generators of \(\Tr(\U^\ast(\fsl_n))\):
\[\mathsf{E}_{i,r}1_{\lambda}:=  \left[
\xy 0;/r.18pc/:
  (10,6)*{\lambda};
  (0,0)*{\bbe{}};
  (0,2)*{\bullet}+(3,1)*{r};
  (-8,6)*{ }; (12,-7)*{}; (2,-7)*{\scs i};
 \endxy
\right],  \quad \quad
\mathsf{F}_{j,s}1_{\lambda}:=  \left[
\xy 0;/r.18pc/:
  (10,6)*{\lambda};
  (0,0)*{\bbf{}};
  (2,-7)*{\scs j};
  (0,4)*{\bullet}+(3,1)*{s};
   (-8,6)*{ }; (12,-7)*{};
 \endxy
\right],\quad \quad
  \mathsf{H}_{i,r}1_{\lambda}:= \left[
 p_{i,r}(\lambda)\; \Id_{\onel}
\right],\]where $p_{i,r}(\l)$ was defined in equation~\ref{eq_defpil}.
For any choice of scalars $Q$ with $t_{ij}^2=t_{ji}^2=t_ij^{-1}t_{ji}=1$ for $i \neq j$, there is a well defined homomorphism
\begin{equation} \label{eq_sln-homo}
 h \maps \dot{\bfU}(\fsl_n[t]) \longrightarrow \Tr(\U^\ast(\fsl_n)),
\end{equation}
given by
\begin{equation} \label{homomorp}
x^{+}_{i, r} 1_{\lambda} \mapsto
(-1)^{(i+1)r}\,\mathsf{E}_{i,r}1_{\lambda}
%:=  \left[
%\xy 0;/r.18pc/:
 % (10,6)*{\lambda};
  %(0,0)*{\bbe{}};
 % (0,2)*{\bullet}+(3,1)*{r};
 % (-8,6)*{ }; (12,-7)*{}; (2,-7)*{\scs i};
 %\endxy
%\right],
 ,\quad \quad
x^{-}_{j, s} 1_{\lambda}\mapsto
(-1)^{(j+1)s}\,\mathsf{F}_{j,s}1_{\lambda}
%:= } \left[
%\xy 0;/r.18pc/:
%  (10,6)*{\lambda};
  %(0,0)*{\bbf{}};
  %(2,-7)*{\scs j};
  %(0,4)*{\bullet}+(3,1)*{s};
   %(-8,6)*{ }; (12,-7)*{};
 %\endxy
%\right]
,\quad \quad
\xi_{i, r} 1_{\lambda}\mapsto
  (-1)^{(i+1)r}\,\mathsf{H}_{i,r}1_{\lambda}.
 %:= \left[
 %p_{i,r}(\lambda)\; \Id_{\onel}
%\right]
\end{equation}
%where $p_{i,r}(\l)$ was defined in equation~\ref{eq_defpil}.
\end{proposition}

\begin{proof}
To prove this proposition we verify the current algebra relations  using the relations in the 2-category $\U^\ast(\fsl_n)$. We only need
to consider the case \(i\neq j\), since the relations in $\U^\ast(\fsl_2)$ have been proven in \cite{BHLZ}.
\textbf{C1} is clear, since multiplication by a bubble is a commutative operation. We prove  the equality \textbf{C3} for the case \(j=i\pm1\).
For convenience we will only depict the interior of the annulus for diagrams on the annulus in our calculation below where $v_{ij}:= t_{ij}^{-1}t_{ji}$.
\[ \mathsf{H}_{i,a}\mathsf{E}_{j,b} 1_\lambda\;\;= \;\;
\sum^a_{k=0} k \; \;
 \xy
   (-12,-2)*{\icbub{\spadesuit+k}{i}};
   (0,-2)*{\iccbub{\spadesuit+a-k}{i}};
  (12,-2)*{\icbub{b}{j}};
    (12,1)*{\lambda}; (12,-3)*{ \hackcenter{\begin{tikzpicture}
    \path[draw, blue, very thick, fill=white]
    (-0.15,0) .. controls ++(0,.2) and ++(0,.2) .. (0.15,0)
            .. controls ++(0,-.2) and ++(0,-.2) .. (-0.15,0);
\end{tikzpicture}}};
 \endxy
 \]
\[ \refequal{\eqref{c_slide_right}} \;\;
\sum^a_{k=0} k \; \;
 \xy
   (-18,-2)*{\iccbub{\spadesuit+a-k}{i}};
   (-1,-2)*{\icbub{\spadesuit+k}{i}};
  (0,6.5)*{\lambda};(5,3)*{ \hackcenter{\begin{tikzpicture}
    \path[draw, blue, very thick, fill=white]
    (-0.15,0) .. controls ++(0,.2) and ++(0,.2) .. (0.15,0)
            .. controls ++(0,-.2) and ++(0,-.2) .. (-0.15,0);
\end{tikzpicture}}};
    (-10,0)*{}="t1";  (10,0)*{}="t2";
  "t2";"t1" **\crv{(10,14) & (-10,14)}; ?(0)*\dir{<} ?(1)*\dir{<} ?(.3)*\dir{}+(2,2)*{\scs j};
  "t2";"t1" **\crv{(10,-14) & (-10,-14)};
  ?(.7)*\dir{}+(0,0)*{\bullet}+(1,-3)*{\scs b};
   \endxy
   \;\;+\;\; v_{ij} \; \sum^a_{k=0} k \;
 \xy
   (-18,-2)*{\iccbub{\spadesuit+a-k}{i}};
   (-1,-1)*{\icbub{\spadesuit+k-1}{i}};
  (1,6.5)*{\lambda};(5,3)*{ \hackcenter{\begin{tikzpicture}
    \path[draw, blue, very thick, fill=white]
    (-0.15,0) .. controls ++(0,.2) and ++(0,.2) .. (0.15,0)
            .. controls ++(0,-.2) and ++(0,-.2) .. (-0.15,0);
\end{tikzpicture}}};
    (-10,0)*{}="t1";  (10,0)*{}="t2";
  "t2";"t1" **\crv{(10,14) & (-10,14)}; ?(0)*\dir{<} ?(1)*\dir{<} ?(.3)*\dir{}+(2,2)*{\scs j};
  "t2";"t1" **\crv{(10,-14) & (-10,-14)};
  ?(.7)*\dir{}+(0,0)*{\bullet}+(1,-3)*{\scs b+1};
   \endxy
 \]
\[\refequal{\eqref{cc_slide_right}}
\sum^a_{k=0}\;\sum^{a-k}_{\ell=0} k \; (-v_{ij})^{\ell}\;
 \xy
   (-6,-2)*{\icbub{\spadesuit+k}{i}};
   (8,-2)*{\iccbub{\spadesuit+a-k-\ell}{i}};
  (8,8)*{\lambda};(4,8)*{ \hackcenter{\begin{tikzpicture}
    \path[draw, blue, very thick, fill=white]
    (-0.15,0) .. controls ++(0,.2) and ++(0,.2) .. (0.15,0)
            .. controls ++(0,-.2) and ++(0,-.2) .. (-0.15,0);
\end{tikzpicture}}};
    (-20,0)*{}="t1";  (20,0)*{}="t2";
  "t2";"t1" **\crv{(20,24) & (-20,24)}; ?(0)*\dir{<} ?(1)*\dir{<} ?(.3)*\dir{}+(2,2)*{\scs j};
  "t2";"t1" **\crv{(20,-24) & (-20,-24)};
  ?(.7)*\dir{}+(0,0)*{\bullet}+(2,-4)*{\scs \ell+b};
   \endxy \;\; +\;\;v_{ij}\;\sum^a_{k=0}\;\sum^{a-k}_{\ell=0} k \; (-v_{ij})^{\ell}\;
 \xy
   (-6,-2)*{\icbub{\spadesuit+k-1}{i}};
   (8,-2)*{\iccbub{\spadesuit+a-k-\ell}{i}};
  (8,8)*{\lambda}; (4,8)*{ \hackcenter{\begin{tikzpicture}
    \path[draw, blue, very thick, fill=white]
    (-0.15,0) .. controls ++(0,.2) and ++(0,.2) .. (0.15,0)
            .. controls ++(0,-.2) and ++(0,-.2) .. (-0.15,0);
\end{tikzpicture}}};
    (-20,0)*{}="t1";  (20,0)*{}="t2";
  "t2";"t1" **\crv{(20,24) & (-20,24)}; ?(0)*\dir{<} ?(1)*\dir{<} ?(.3)*\dir{}+(2,2)*{\scs j};
  "t2";"t1" **\crv{(20,-24) & (-20,-24)};
  ?(.7)*\dir{}+(0,0)*{\bullet}+(2,-4)*{\scs \ell+b+1};
   \endxy
\]
\[  =\;\; \sum^a_{\ell=0}\;\sum^{a-\ell}_{k=0} k \; (-v_{ij})^{\ell}\;
 \xy
   (-9,-2)*{\icbub{\spadesuit+k}{i}};
   (5,-2)*{\iccbub{\spadesuit+a-k-\ell}{i}};
  (8,8)*{\lambda};(4,8)*{ \hackcenter{\begin{tikzpicture}
    \path[draw, blue, very thick, fill=white]
    (-0.15,0) .. controls ++(0,.2) and ++(0,.2) .. (0.15,0)
            .. controls ++(0,-.2) and ++(0,-.2) .. (-0.15,0);
\end{tikzpicture}}};
    (-20,0)*{}="t1";  (20,0)*{}="t2";
  "t2";"t1" **\crv{(20,24) & (-20,24)}; ?(0)*\dir{<} ?(1)*\dir{<} ?(.3)*\dir{}+(2,2)*{\scs j};
  "t2";"t1" **\crv{(20,-24) & (-20,-24)};
  ?(.7)*\dir{}+(0,0)*{\bullet}+(1,-3)*{\scs \ell+b};
   \endxy \;\; +\;v_{ij}\;\sum^a_{\ell=0}\;\sum^{a-\ell-1}_{k=0} (k+1) \; (-v_{ij})^{\ell}\;
 \xy
   (-9,-2)*{\icbub{\spadesuit+k}{i}};
   (6,-2)*{\iccbub{\spadesuit+a-k-\ell-1}{i}};
  (8,8)*{\lambda};(4,8)*{ \hackcenter{\begin{tikzpicture}
    \path[draw, blue, very thick, fill=white]
    (-0.15,0) .. controls ++(0,.2) and ++(0,.2) .. (0.15,0)
            .. controls ++(0,-.2) and ++(0,-.2) .. (-0.15,0);
\end{tikzpicture}}};
    (-20,0)*{}="t1";  (20,0)*{}="t2";
  "t2";"t1" **\crv{(20,24) & (-20,24)}; ?(0)*\dir{<} ?(1)*\dir{<} ?(.3)*\dir{}+(2,2)*{\scs j};
  "t2";"t1" **\crv{(20,-24) & (-20,-24)};
  ?(.7)*\dir{}+(0,0)*{\bullet}+(2,-4)*{\scs \ell+b+1};
   \endxy
\]
\[ \refequal{\eqref{eq_defpil}}
\;\; \sum^a_{\ell=0} \; (-v_{ij})^{\ell}\; \;
 \xy
  (5,5)*{\lambda}; (1,5)*{ \hackcenter{\begin{tikzpicture}
    \path[draw, blue, very thick, fill=white]
    (-0.15,0) .. controls ++(0,.2) and ++(0,.2) .. (0.15,0)
            .. controls ++(0,-.2) and ++(0,-.2) .. (-0.15,0);
\end{tikzpicture}}};
  (0,0)*{\mathsf{H}_{i,a-\ell}};
(-10,0)*{}="t1";  (10,0)*{}="t2";
  "t2";"t1" **\crv{(10,14) & (-10,14)}; ?(0)*\dir{<} ?(1)*\dir{<} ?(.3)*\dir{}+(2,2)*{\scs j};
  "t2";"t1" **\crv{(10,-14) & (-10,-14)};
  ?(.7)*\dir{}+(0,0)*{\bullet}+(1,-3)*{\scs \ell+b};
   \endxy \;\; -\;\;\sum^{a+1}_{\ell=1}\;\sum^{a-\ell}_{k=0} (k+1) \; (-v_{ij})^{\ell}\;
 \xy
   (-6,-2)*{\icbub{\spadesuit+k}{i}};
   (10,-2)*{\iccbub{\spadesuit+a-k-\ell}{i}};
  (8,8)*{\lambda}; (4,8)*{ \hackcenter{\begin{tikzpicture}
    \path[draw, blue, very thick, fill=white]
    (-0.15,0) .. controls ++(0,.2) and ++(0,.2) .. (0.15,0)
            .. controls ++(0,-.2) and ++(0,-.2) .. (-0.15,0);
\end{tikzpicture}}};
    (-20,0)*{}="t1";  (20,0)*{}="t2";
  "t2";"t1" **\crv{(20,24) & (-20,24)}; ?(0)*\dir{<} ?(1)*\dir{<} ?(.3)*\dir{}+(2,2)*{\scs j};
  "t2";"t1" **\crv{(20,-24) & (-20,-24)};
  ?(.7)*\dir{}+(0,0)*{\bullet}+(1,-3)*{\scs \ell+b};
   \endxy
\]
Pulling off $k$ of the summands in the second term and observing that the $\ell=a+1$ term vanishes we have
\[
=\;\; \sum^a_{\ell=0} \; (-v_{ij})^{\ell}\; \;
 \xy
  (5,5)*{\lambda}; (1,5)*{ \hackcenter{\begin{tikzpicture}
    \path[draw, blue, very thick, fill=white]
    (-0.15,0) .. controls ++(0,.2) and ++(0,.2) .. (0.15,0)
            .. controls ++(0,-.2) and ++(0,-.2) .. (-0.15,0);
\end{tikzpicture}}};
  (0,0)*{\mathsf{H}_{i,a-\ell}};
(-10,0)*{}="t1";  (10,0)*{}="t2";
  "t2";"t1" **\crv{(10,14) & (-10,14)}; ?(0)*\dir{<} ?(1)*\dir{<} ?(.3)*\dir{}+(2,2)*{\scs j};
  "t2";"t1" **\crv{(10,-14) & (-10,-14)};
  ?(.7)*\dir{}+(0,0)*{\bullet}+(1,-3)*{\scs \ell+b};
   \endxy \;\;  -\;\; \sum^a_{\ell=1} \; (-v_{ij})^{\ell}\; \;
 \xy
  (5,5)*{\lambda}; (1,5)*{ \hackcenter{\begin{tikzpicture}
    \path[draw, blue, very thick, fill=white]
    (-0.15,0) .. controls ++(0,.2) and ++(0,.2) .. (0.15,0)
            .. controls ++(0,-.2) and ++(0,-.2) .. (-0.15,0);
\end{tikzpicture}}};
  (0,0)*{\mathsf{H}_{i,a-\ell}};
(-10,0)*{}="t1";  (10,0)*{}="t2";
  "t2";"t1" **\crv{(10,14) & (-10,14)}; ?(0)*\dir{<} ?(1)*\dir{<} ?(.3)*\dir{}+(2,2)*{\scs j};
  "t2";"t1" **\crv{(10,-14) & (-10,-14)};
  ?(.7)*\dir{}+(0,0)*{\bullet}+(1,-3)*{\scs \ell+b};
   \endxy \;-\;\sum^{a}_{\ell=1}\;\sum^{a-\ell}_{k=0}  \; (-v_{ij})^{\ell}\;
 \xy
   (-6,-2)*{\icbub{\spadesuit+k}{i}};
   (10,-2)*{\iccbub{\spadesuit+a-k-\ell}{i}};
  (8,8)*{\lambda};(4,8)*{ \hackcenter{\begin{tikzpicture}
    \path[draw, blue, very thick, fill=white]
    (-0.15,0) .. controls ++(0,.2) and ++(0,.2) .. (0.15,0)
            .. controls ++(0,-.2) and ++(0,-.2) .. (-0.15,0);
\end{tikzpicture}}};
    (-20,0)*{}="t1";  (20,0)*{}="t2";
  "t2";"t1" **\crv{(20,24) & (-20,24)}; ?(0)*\dir{<} ?(1)*\dir{<} ?(.3)*\dir{}+(2,2)*{\scs j};
  "t2";"t1" **\crv{(20,-24) & (-20,-24)};
  ?(.7)*\dir{}+(0,0)*{\bullet}+(1,-3)*{\scs \ell+b};
   \endxy
  \]
\[=\; \;
 \xy
  (5,5)*{\lambda}; (1,5)*{ \hackcenter{\begin{tikzpicture}
    \path[draw, blue, very thick, fill=white]
    (-0.15,0) .. controls ++(0,.2) and ++(0,.2) .. (0.15,0)
            .. controls ++(0,-.2) and ++(0,-.2) .. (-0.15,0);
\end{tikzpicture}}};
  (0,0)*{\mathsf{H}_{i,a}};
(-10,0)*{}="t1";  (10,0)*{}="t2";
  "t2";"t1" **\crv{(10,14) & (-10,14)}; ?(0)*\dir{<} ?(1)*\dir{<} ?(.3)*\dir{}+(2,2)*{\scs j};
  "t2";"t1" **\crv{(10,-14) & (-10,-14)};
  ?(.7)*\dir{}+(0,0)*{\bullet}+(1,-3)*{\scs b};
   \endxy \;\; -\;\; \sum^a_{\ell=1} \; (-v_{ij})^{\ell}\; \;
 \xy
  (5,5)*{\lambda};(1,5)*{ \hackcenter{\begin{tikzpicture}
    \path[draw, blue, very thick, fill=white]
    (-0.15,0) .. controls ++(0,.2) and ++(0,.2) .. (0.15,0)
            .. controls ++(0,-.2) and ++(0,-.2) .. (-0.15,0);
\end{tikzpicture}}};
   (0,0)*{\delta_{a-\ell,0}};
(-10,0)*{}="t1";  (10,0)*{}="t2";
  "t2";"t1" **\crv{(10,14) & (-10,14)}; ?(0)*\dir{<} ?(1)*\dir{<} ?(.3)*\dir{}+(2,2)*{\scs j};
  "t2";"t1" **\crv{(10,-14) & (-10,-14)};
  ?(.7)*\dir{}+(0,0)*{\bullet}+(1,-3)*{\scs \ell+b};
   \endxy \;\; =\;\; \xy
  (5,5)*{\lambda};
  (1,5)*{ \hackcenter{\begin{tikzpicture}
    \path[draw, blue, very thick, fill=white]
    (-0.15,0) .. controls ++(0,.2) and ++(0,.2) .. (0.15,0)
            .. controls ++(0,-.2) and ++(0,-.2) .. (-0.15,0);
\end{tikzpicture}}};(0,0)*{\mathsf{H}_{i,a}};
(-10,0)*{}="t1";  (10,0)*{}="t2";
  "t2";"t1" **\crv{(10,14) & (-10,14)}; ?(0)*\dir{<} ?(1)*\dir{<} ?(.3)*\dir{}+(2,2)*{\scs j};
  "t2";"t1" **\crv{(10,-14) & (-10,-14)};
  ?(.7)*\dir{}+(0,0)*{\bullet}+(1,-3)*{\scs b};
   \endxy \;\; -\;\; (-v_{ij})^a\; \;
 \xy
  (3,3)*{\lambda}; (0,0)*{ \hackcenter{\begin{tikzpicture}
    \path[draw, blue, very thick, fill=white]
    (-0.15,0) .. controls ++(0,.2) and ++(0,.2) .. (0.15,0)
            .. controls ++(0,-.2) and ++(0,-.2) .. (-0.15,0);
\end{tikzpicture}}};
(-10,0)*{}="t1";  (10,0)*{}="t2";
  "t2";"t1" **\crv{(10,14) & (-10,14)}; ?(0)*\dir{<} ?(1)*\dir{<} ?(.3)*\dir{}+(2,2)*{\scs j};
  "t2";"t1" **\crv{(10,-14) & (-10,-14)};
  ?(.7)*\dir{}+(0,0)*{\bullet}+(1,-3)*{\scs a+b};
   \endxy
\]
\[=\;\;\mathsf{E}_{j,b} \mathsf{H}_{i,a} 1_\lambda - (-v_{ij})^a\;\mathsf{E}_{j,a+b}  1_\lambda. \]
By choosing \(v_{ij}=1\), we get
\[[ \mathsf{H}_{i,a}, \mathsf{E}_{j,b}] 1_\lambda = -(-1)^a \mathsf{E}_{j,a+b}1_\lambda .\]
The relation
\[[ \mathsf{H}_{i,a}, \mathsf{F}_{j,b}] 1_\lambda = (-1)^a \mathsf{F}_{j,a+b}1_\lambda .\]
can be proven in a similar way.
\textbf{C4} follows from the relations
\eqref{eq_r2_ij-gen} and \eqref{eq_dot_slide_ij-gen}:
\[
  t_{ij}\;\; \xy
 (-6,0)*{};
  (6,0)*{};
  (0,2)*{\lambda};(0,-2)*{ \hackcenter{\begin{tikzpicture}
    \path[draw, blue, very thick, fill=white]
    (-0.15,0) .. controls ++(0,.2) and ++(0,.2) .. (0.15,0)
            .. controls ++(0,-.2) and ++(0,-.2) .. (-0.15,0);
\end{tikzpicture}}};
  (-4,0)*{}="t1";  (4,0)*{}="t2";
 "t2";"t1" **\crv{(4,6) & (-4,6)}; ?(.02)*\dir{<}
   ?(1)*\dir{<} ?(.3)*\dir{}+(2,2)*{\scs j};
  "t2";"t1" **\crv{(4,-6) & (-4,-6)};
  ?(.7)*\dir{}+(0,0)*{\bullet}+(2,-3)*{\scs b};
  (-12,0)*{}="t1";  (12,0)*{}="t2";
  "t2";"t1" **\crv{(12,15) & (-12,15)}; ?(0)*\dir{<} ?(1)*\dir{<} ?(.3)*\dir{}+(2,2)*{\scs i};
  "t2";"t1" **\crv{(12,-15) & (-12,-15)};
  ?(.7)*\dir{}+(0,0)*{\bullet}+(2,-4)*{\scs a+1};
  \endxy \; + t_{ji}\;\;  \xy
 (-6,0)*{};
  (6,0)*{};
   (0,2)*{\lambda};(0,-2)*{ \hackcenter{\begin{tikzpicture}
    \path[draw, blue, very thick, fill=white]
    (-0.15,0) .. controls ++(0,.2) and ++(0,.2) .. (0.15,0)
            .. controls ++(0,-.2) and ++(0,-.2) .. (-0.15,0);
\end{tikzpicture}}};
  (-4,0)*{}="t1";  (4,0)*{}="t2";
 "t2";"t1" **\crv{(4,6) & (-4,6)}; ?(.02)*\dir{<}
   ?(1)*\dir{<} ?(.3)*\dir{}+(2,2)*{\scs j};
  "t2";"t1" **\crv{(4,-6) & (-4,-6)};
  ?(.7)*\dir{}+(0,0)*{\bullet}+(2,-3)*{\scs b+1};
  (-12,0)*{}="t1";  (12,0)*{}="t2";
  "t2";"t1" **\crv{(12,15) & (-12,15)}; ?(0)*\dir{<} ?(1)*\dir{<} ?(.3)*\dir{}+(2,2)*{\scs i};
  "t2";"t1" **\crv{(12,-15) & (-12,-15)};
  ?(.7)*\dir{}+(0,0)*{\bullet}+(2,-4)*{\scs a};
  \endxy
  \;\; \refequal{\eqref{eq_r2_ij-gen}}\;\;
    \xy 0;/r.18pc/:
     (0,-1)*{
 };
 (0,6)*{\lambda};(0,0)*{ \hackcenter{\begin{tikzpicture}
    \path[draw, blue, very thick, fill=white]
    (-0.15,0) .. controls ++(0,.2) and ++(0,.2) .. (0.15,0)
            .. controls ++(0,-.2) and ++(0,-.2) .. (-0.15,0);
\end{tikzpicture}}};
 (-8,-4)*{\ecross};(-8,4)*{\naecross};
  (4,0)*{\mline};(12,0)*{\mline};(0,10)*{\scap};(0,12)*{\lcap};
  (0,-10)*{\scupef};(0,-14)*{\lcupef};
  (6,-3)*{\scs j};(15,-3)*{\scs i};
   (-12,8)*{\bullet}+(-2.5,1)*{\scs a};
 (-4,8)*{\bullet}+(-2.5,1)*{\scs b};
  \endxy
\;\; \refequal{\eqref{eq_dot_slide_ij-gen}} \;\;
    \xy 0;/r.18pc/:
     (0,-1)*{
 };
(0,6)*{\lambda};(0,0)*{ \hackcenter{\begin{tikzpicture}
    \path[draw, blue, very thick, fill=white]
    (-0.15,0) .. controls ++(0,.2) and ++(0,.2) .. (0.15,0)
            .. controls ++(0,-.2) and ++(0,-.2) .. (-0.15,0);
\end{tikzpicture}}};
 (-8,-4)*{\ecross};(-8,4)*{\naecross};
  (4,0)*{\mline};(12,0)*{\mline};(0,10)*{\scap};(0,12)*{\lcap};
  (0,-10)*{\scupef};(0,-14)*{\lcupef};
  (6,-3)*{\scs i};(15,-3)*{\scs j};
   (-12,8)*{\bullet}+(-2.5,1)*{\scs b};
 (-4,8)*{\bullet}+(-2.5,1)*{\scs a};
  \endxy
\]
\[\; \refequal{\eqref{eq_r2_ij-gen}} \;
  t_{ji}\;\; \xy
 (-6,0)*{};
  (6,0)*{};
   (0,2)*{\lambda};(0,-2)*{ \hackcenter{\begin{tikzpicture}
    \path[draw, blue, very thick, fill=white]
    (-0.15,0) .. controls ++(0,.2) and ++(0,.2) .. (0.15,0)
            .. controls ++(0,-.2) and ++(0,-.2) .. (-0.15,0);
\end{tikzpicture}}};  (-4,0)*{}="t1";  (4,0)*{}="t2";
 "t2";"t1" **\crv{(4,6) & (-4,6)}; ?(.02)*\dir{<}
   ?(1)*\dir{<} ?(.3)*\dir{}+(2,2)*{\scs i};
  "t2";"t1" **\crv{(4,-6) & (-4,-6)};
  ?(.7)*\dir{}+(0,0)*{\bullet}+(2,-3)*{\scs a};
  (-12,0)*{}="t1";  (12,0)*{}="t2";
  "t2";"t1" **\crv{(12,15) & (-12,15)}; ?(0)*\dir{<} ?(1)*\dir{<} ?(.3)*\dir{}+(2,2)*{\scs j};
  "t2";"t1" **\crv{(12,-15) & (-12,-15)};
  ?(.7)*\dir{}+(0,0)*{\bullet}+(2,-4)*{\scs b+1};
  \endxy \; + t_{ij}\;\;  \xy
 (-6,0)*{};
  (6,0)*{};
   (0,2)*{\lambda};(0,-2)*{ \hackcenter{\begin{tikzpicture}
    \path[draw, blue, very thick, fill=white]
    (-0.15,0) .. controls ++(0,.2) and ++(0,.2) .. (0.15,0)
            .. controls ++(0,-.2) and ++(0,-.2) .. (-0.15,0);
\end{tikzpicture}}};
  (-4,0)*{}="t1";  (4,0)*{}="t2";
 "t2";"t1" **\crv{(4,6) & (-4,6)}; ?(.02)*\dir{<}
   ?(1)*\dir{<} ?(.3)*\dir{}+(2,2)*{\scs i};
  "t2";"t1" **\crv{(4,-6) & (-4,-6)};
  ?(.7)*\dir{}+(0,0)*{\bullet}+(2,-3)*{\scs a+1};
  (-12,0)*{}="t1";  (12,0)*{}="t2";
  "t2";"t1" **\crv{(12,15) & (-12,15)}; ?(0)*\dir{<} ?(1)*\dir{<} ?(.3)*\dir{}+(2,2)*{\scs j};
  "t2";"t1" **\crv{(12,-15) & (-12,-15)};
  ?(.7)*\dir{}+(0,0)*{\bullet}+(2,-4)*{\scs b};
  \endxy \; \; .\]
  Dividing both sides of this equation by \(t_{ij}\), and by choosing \(v_{ij}=1\), we get
  \[ [ \mathsf{E}_{i,a+1},\mathsf{E}_{j,b} ]1_\lambda=-[\mathsf{E}_{i,a},\mathsf{E}_{j,b+1} ]1_\lambda .\]
  If we reverse the arrows in the preceding equation, they still hold:
   \[ [ \mathsf{F}_{i,a+1},\mathsf{F}_{j,b} ]1_\lambda=-[\mathsf{F}_{i,a},\mathsf{F}_{j,b+1} ]1_\lambda .\]
The choice of signs  in \eqref{homomorp} correspond to
  \[ [x^\pm_{i, a+1}, x^\pm_{j,b}]=[x^\pm_{i,a}, x^\pm_{j,b+1}].\]
The relation \textbf{C5} for the case \(i\neq j\) can be checked the following way:  \[
\mathsf{E}_{i,a}\mathsf{F}_{j,b} 1_\lambda
 \;\; = \;\;
  \xy
 (-6,0)*{};
  (6,0)*{};
  (0,-2)*{ \hackcenter{\begin{tikzpicture}
    \path[draw, blue, very thick, fill=white]
    (-0.15,0) .. controls ++(0,.2) and ++(0,.2) .. (0.15,0)
            .. controls ++(0,-.2) and ++(0,-.2) .. (-0.15,0);
\end{tikzpicture}}};
(0,2)*{\lambda};
  (-4,0)*{}="t1";  (4,0)*{}="t2";
 "t2";"t1" **\crv{(4,6) & (-4,6)}; ?(.02)*\dir{>}
   ?(1)*\dir{>} ?(.3)*\dir{}+(2,2)*{\scs j};
  "t2";"t1" **\crv{(4,-6) & (-4,-6)};
  ?(.7)*\dir{}+(0,0)*{\bullet}+(2,-3)*{\scs b};
  (-12,0)*{}="t1";  (12,0)*{}="t2";
  "t2";"t1" **\crv{(12,15) & (-12,15)}; ?(0)*\dir{<} ?(1)*\dir{<} ?(.3)*\dir{}+(2,2)*{\scs i};
  "t2";"t1" **\crv{(12,-15) & (-12,-15)};
  ?(.7)*\dir{}+(0,0)*{\bullet}+(2,-4)*{\scs a};
  \endxy
  \;\; \refequal{\eqref{mixed_rel}} \;
   t^{-1}_{ji}\;  \xy 0;/r.18pc/:
     (0,-1)*{
 };(0,-2)*{ \hackcenter{\begin{tikzpicture}
    \path[draw, blue, very thick, fill=white]
    (-0.15,0) .. controls ++(0,.2) and ++(0,.2) .. (0.15,0)
            .. controls ++(0,-.2) and ++(0,-.2) .. (-0.15,0);
\end{tikzpicture}}};
(0,6)*{\lambda};
(-8,-4)*{\fecross};(-8,4)*{\naecross};
  (4,0)*{\mline};(12,0)*{\mline};(0,10)*{\scap};(0,12)*{\lcap};
  (0,-10)*{\scupfe};(0,-14)*{\lcupef};
  (6,-3)*{\scs j};(15,-3)*{\scs i};
   (-12,8)*{\bullet}+(-2.5,1)*{\scs a};
 (-4,8)*{\bullet}+(-2.5,1)*{\scs b};
  \endxy
  \;\; \refequal{\eqref{eq_dot_slide_ij-gen}} \;
\]  \[
\;\; \refequal{\eqref{eq_dot_slide_ij-gen}} \;
  t^{-1}_{ji}\;  \xy 0;/r.18pc/:
     (0,-1)*{ };
(0,-2)*{ \hackcenter{\begin{tikzpicture}
    \path[draw, blue, very thick, fill=white]
    (-0.15,0) .. controls ++(0,.2) and ++(0,.2) .. (0.15,0)
            .. controls ++(0,-.2) and ++(0,-.2) .. (-0.15,0);
\end{tikzpicture}}};
(0,6)*{\lambda};     (-8,-4)*{\efcross};(-8,4)*{\naecross};
  (4,0)*{\mline};(12,0)*{\mline};(0,10)*{\scap};(0,12)*{\lcap};
  (0,-10)*{\scupef};(0,-14)*{\lcupfe};
  (6,-3)*{\scs i};(15,-3)*{\scs j};
   (-12,8)*{\bullet}+(-2.5,1)*{\scs b};
 (-4,8)*{\bullet}+(-2.5,1)*{\scs a};
  \endxy
 \;\; \;\; \refequal{\eqref{mixed_rel}} \;
   \xy
 (-6,0)*{};
  (6,0)*{};
(0,-2)*{ \hackcenter{\begin{tikzpicture}
    \path[draw, blue, very thick, fill=white]
    (-0.15,0) .. controls ++(0,.2) and ++(0,.2) .. (0.15,0)
            .. controls ++(0,-.2) and ++(0,-.2) .. (-0.15,0);
\end{tikzpicture}}};
(0,2)*{\lambda};
  (-4,0)*{}="t1";  (4,0)*{}="t2";
 "t2";"t1" **\crv{(4,6) & (-4,6)}; ?(.02)*\dir{<}
   ?(1)*\dir{<} ?(.3)*\dir{}+(2,2)*{\scs i};
  "t2";"t1" **\crv{(4,-6) & (-4,-6)};
  ?(.7)*\dir{}+(0,0)*{\bullet}+(2,-3)*{\scs a};
  (-12,0)*{}="t1";  (12,0)*{}="t2";
  "t2";"t1" **\crv{(12,15) & (-12,15)}; ?(0)*\dir{>} ?(1)*\dir{>} ?(.3)*\dir{}+(2,2)*{\scs j};
  "t2";"t1" **\crv{(12,-15) & (-12,-15)};
  ?(.7)*\dir{}+(0,0)*{\bullet}+(2,-4)*{\scs b};
  \endxy  \;\; \;\; = \;
\mathsf{F}_{j,b} \mathsf{E}_{i,a}1_\lambda.
\]
 %To prove this proposition we need to verify
%the current algebra relations  using the relations in the 2-category $\U^\ast(\fsl_n)$.
%\begin{exercise}
%Verify axiom C3 using exercise~\ref{exercise_power-slides}.
%\end{exercise}
We now prove the Serre relation \textbf{C6} in the case when $i \cdot j = -1$.
Let us prove the identity $$\mathsf{E}_{i,a}\mathsf{E}_{j,b}\mathsf{E}_{i,c}1_{\lambda} +\mathsf{E}_{i,c}\mathsf{E}_{j,b}\mathsf{E}_{i,a}1_{\lambda} = \mathsf{E}_{i,a}\mathsf{E}_{i,c}\mathsf{E}_{j,b} 1_{\lambda}+\mathsf{E}_{j,b}\mathsf{E}_{i,a}\mathsf{E}_{i,c} 1_{\lambda}.$$
\begin{equation} \label{eq_vv1}
\mathsf{E}_{i,a}\mathsf{E}_{j,b}\mathsf{E}_{i,c}1_{\lambda}  =\;\;
 \xy
 (-6,0)*{};
  (6,0)*{};
  (0,-2)*{ \hackcenter{\begin{tikzpicture}
    \path[draw, blue, very thick, fill=white]
    (-0.15,0) .. controls ++(0,.2) and ++(0,.2) .. (0.15,0)
            .. controls ++(0,-.2) and ++(0,-.2) .. (-0.15,0);
\end{tikzpicture}}};
(0,2)*{\lambda};
  (-4,0)*{}="t1";  (4,0)*{}="t2";
 "t2";"t1" **\crv{(4,6) & (-4,6)}; ?(.02)*\dir{<}
   ?(1)*\dir{<} ?(.3)*\dir{}+(2,2)*{\scs i};
  "t2";"t1" **\crv{(4,-6) & (-4,-6)};
  ?(.7)*\dir{}+(0,0)*{\bullet}+(2,-3)*{\scs c};
  (-12,0)*{}="t1";  (12,0)*{}="t2";
  "t2";"t1" **\crv{(12,15) & (-12,15)}; ?(0)*\dir{<} ?(1)*\dir{<} ?(.3)*\dir{}+(2,2)*{\scs j};
  "t2";"t1" **\crv{(12,-15) & (-12,-15)};
  ?(.7)*\dir{}+(0,0)*{\bullet}+(2,-4)*{\scs b};
   (-20,0)*{}="t1";  (20,0)*{}="t2";
  "t2";"t1" **\crv{(20,24) & (-20,24)}; ?(0)*\dir{<} ?(1)*\dir{<} ?(.3)*\dir{}+(2,2)*{\scs i};
  "t2";"t1" **\crv{(20,-24) & (-20,-24)};
  ?(.7)*\dir{}+(0,0)*{\bullet}+(2,-4)*{\scs a};
  \endxy
\;\; \refequal{\eqref{eq_r3_hard-gen}}\;\;
  t^{-1}_{ij} \vcenter{
 \xy 0;/r.18pc/:
    (-4,-4)*{};(4,4)*{} **\crv{(-4,-1) & (4,1)}?(1)*\dir{>};
    (4,-4)*{};(-4,4)*{} **\crv{(4,-1) & (-4,1)}?(1)*\dir{>};
    (4,4)*{};(12,12)*{} **\crv{(4,7) & (12,9)}?(1)*\dir{>};
    (12,4)*{};(4,12)*{} **\crv{(12,7) & (4,9)}?(1)*\dir{>};
    (-4,12)*{};(4,20)*{} **\crv{(-4,15) & (4,17)}?(1)*\dir{>};
    (4,12)*{};(-4,20)*{} **\crv{(4,15) & (-4,17)}?(1)*\dir{>};
    (-4,4)*{}; (-4,12) **\dir{-};
    (12,-4)*{}; (12,4) **\dir{-};
    (12,12)*{}; (12,20) **\dir{-};
 (16,4)*{ \hackcenter{\begin{tikzpicture}
    \path[draw, blue, very thick, fill=white]
    (-0.15,0) .. controls ++(0,.2) and ++(0,.2) .. (0.15,0)
            .. controls ++(0,-.2) and ++(0,-.2) .. (-0.15,0);
\end{tikzpicture}}};
(16,12)*{\lambda};
  (-6,-3)*{\scs i}; (6,-3)*{\scs j};(15,-3)*{\scs i};  (20,5)*{\bullet}+(3,3)*{\scs c};(28,5)*{\bullet}+(3,3)*{\scs b}; (36,5)*{\bullet}+(3,3)*{\scs a};
  (20,8)*{\lfline};(28,8)*{\lfline};(36,8)*{\lfline};
  (16,22)*{\scap};(16,24)*{\lcap};(16,30)*{\xlcap};
  (16,-6)*{\scupef};(16,-10)*{\lcupef};(16,-15)*{\xlcupef};
\endxy}\;
 \;\; -\;  t^{-1}_{ij} \;
 \vcenter{
 \xy 0;/r.18pc/:
    (12,-4)*{};(4,4)*{} **\crv{(12,-1) & (4,1)}?(1)*\dir{>};
    (4,-4)*{};(12,4)*{} **\crv{(4,-1) & (12,1)}?(1)*\dir{>};
    (4,4)*{};(-4,12)*{} **\crv{(-,7) & (-4,9)}?(1)*\dir{>};
    (-4,4)*{};(4,12)*{} **\crv{(-4,7) & (4,9)}?(1)*\dir{>};
    (12,12)*{};(4,20)*{} **\crv{(12,15) & (4,17)}?(1)*\dir{>};
    (4,12)*{};(12,20)*{} **\crv{(4,15) & (12,17)}?(1)*\dir{>};
    (12,4)*{}; (12,12) **\dir{-};
    (-4,-4)*{}; (-4,4) **\dir{-};
    (-4,12)*{}; (-4,20) **\dir{-};
   (16,4)*{ \hackcenter{\begin{tikzpicture}
    \path[draw, blue, very thick, fill=white]
    (-0.15,0) .. controls ++(0,.2) and ++(0,.2) .. (0.15,0)
            .. controls ++(0,-.2) and ++(0,-.2) .. (-0.15,0);
\end{tikzpicture}}};
(16,12)*{\lambda};
  (15,-3)*{\scs i};
  (2,-3)*{\scs j};
  (-6,-3)*{\scs i};(20,5)*{\bullet}+(3,3)*{\scs c};(28,5)*{\bullet}+(3,3)*{\scs b}; (36,5)*{\bullet}+(3,3)*{\scs a};
   (20,8)*{\lfline};(28,8)*{\lfline};(36,8)*{\lfline};
  (16,22)*{\scap};(16,24)*{\lcap};(16,30)*{\xlcap};
  (16,-6)*{\scupef};(16,-10)*{\lcupef};(16,-15)*{\xlcupef};
\endxy}\;\;.
\end{equation}
The first term on the right-hand-side of \eqref{eq_vv1} can be simplified
\begin{eqnarray}
t^{-1}_{ij} \vcenter{
 \xy 0;/r.18pc/:
    (-4,-4)*{};(4,4)*{} **\crv{(-4,-1) & (4,1)}?(1)*\dir{>};
    (4,-4)*{};(-4,4)*{} **\crv{(4,-1) & (-4,1)}?(1)*\dir{>};
    (4,4)*{};(12,12)*{} **\crv{(4,7) & (12,9)}?(1)*\dir{>};
    (12,4)*{};(4,12)*{} **\crv{(12,7) & (4,9)}?(1)*\dir{>};
    (-4,12)*{};(4,20)*{} **\crv{(-4,15) & (4,17)}?(1)*\dir{>};
    (4,12)*{};(-4,20)*{} **\crv{(4,15) & (-4,17)}?(1)*\dir{>};
    (-4,4)*{}; (-4,12) **\dir{-};
    (12,-4)*{}; (12,4) **\dir{-};
    (12,12)*{}; (12,20) **\dir{-};
   (16,4)*{ \hackcenter{\begin{tikzpicture}
    \path[draw, blue, very thick, fill=white]
    (-0.15,0) .. controls ++(0,.2) and ++(0,.2) .. (0.15,0)
            .. controls ++(0,-.2) and ++(0,-.2) .. (-0.15,0);
\end{tikzpicture}}};
(16,12)*{\lambda};
  (-6,-3)*{\scs i}; (6,-3)*{\scs j};(15,-3)*{\scs i};  (20,5)*{\bullet}+(3,3)*{\scs c};(28,5)*{\bullet}+(3,3)*{\scs b}; (36,5)*{\bullet}+(3,3)*{\scs a};
  (20,8)*{\lfline};(28,8)*{\lfline};(36,8)*{\lfline};
  (16,22)*{\scap};(16,24)*{\lcap};(16,30)*{\xlcap};
  (16,-6)*{\scupef};(16,-10)*{\lcupef};(16,-15)*{\xlcupef};
\endxy}\;
\;\; = \;\; t_{ij} \; \vcenter{ \xy 0;/r.18pc/:
 (8,4)*{ \hackcenter{\begin{tikzpicture}
%  \path[draw,blue, very thick, fill=blue!10]
%   (-1.5,-.6) to (-1.5,.6) .. controls ++(0,1.15) and ++(0,1.15) ..
%   (1.5,.6)
%   to (1.5,-.6)  .. controls ++(0,-1.15) and ++(0,-1.15) ..
%   (-1.5,-.6);
    \path[draw, blue, very thick, fill=white]
    (-0.15,0) .. controls ++(0,.2) and ++(0,.2) .. (0.15,0)
            .. controls ++(0,-.2) and ++(0,-.2) .. (-0.15,0);
\end{tikzpicture}}};
(8,12)*{\lambda};
  (8,16)*{\lcap};
  (8,22)*{\xlcap};
 (8,14)*{\scap};
 (8,-6)*{\scupef};(8,-14)*{\xlcupef};
 (0,0)*{\ecross};(-4,8)*{\seline};(4,8)*{\seline};
 (-12,4)*{\meline};
 (24,8)*{\fcross};
 (24,0)*{\fcross};
 (12,4)*{\mfline};
 (8,-10)*{\lcupef};
 (-6,-3)*{\scs i}; (6,-3)*{\scs i};(-15,-3)*{\scs j}; (12,5)*{\bullet}+(3,3)*{\scs c};(-12,5)*{\bullet}+(3,3)*{\scs b}; (-4,5)*{\bullet}+(3,3)*{\scs a};
 \endxy} \;.
\end{eqnarray}
using the biadjoint relations and equations~\eqref{eq_crossl-gen} and \eqref{eq_crossr-gen} to slide crossings around the annulus.  This diagram can be further simplified using the KLR relations
\begin{equation}
   \;\; \refequal{\eqref{eq_r2_ij-gen}}\;\;
   t_{ji}^3 \; \vcenter{ \xy 0;/r.18pc/:
 (8,4)*{ \hackcenter{\begin{tikzpicture}
%  \path[draw,blue, very thick, fill=blue!10]
%   (-1.5,-.6) to (-1.5,.6) .. controls ++(0,1.15) and ++(0,1.15) ..
%   (1.5,.6)
%   to (1.5,-.6)  .. controls ++(0,-1.15) and ++(0,-1.15) ..
%   (-1.5,-.6);
    \path[draw, blue, very thick, fill=white]
    (-0.15,0) .. controls ++(0,.2) and ++(0,.2) .. (0.15,0)
            .. controls ++(0,-.2) and ++(0,-.2) .. (-0.15,0);
\end{tikzpicture}}};
(8,12)*{\lambda}; (8,16)*{\lcap};  (8,22)*{\xlcap}; (8,14)*{\scap};  (8,-6)*{\scupef};
  (8,-14)*{\xlcupef}; (0,0)*{\ecross}; (-4,8)*{\seline}; (4,8)*{\seline}; (-12,4)*{\meline};
 (28,4)*{\mfline}; (20,4)*{\mfline}; (12,4)*{\mfline}; (8,-10)*{\lcupef}; (-6,-3)*{\scs i};
  (6,-3)*{\scs i};  (-15,-3)*{\scs j}; (12,5)*{\bullet}+(3,3)*{\scs c};(-12,5)*{\bullet}+(4,3)*{\scs {b+1}}; (-4,5)*{\bullet}+(3,3)*{\scs a};
 \endxy}
 + \; t_{ji}^2 t_{ij} \;
\vcenter{ \xy 0;/r.18pc/:
 (8,4)*{ \hackcenter{\begin{tikzpicture}
%  \path[draw,blue, very thick, fill=blue!10]
%   (-1.5,-.6) to (-1.5,.6) .. controls ++(0,1.15) and ++(0,1.15) ..
%   (1.5,.6)
%   to (1.5,-.6)  .. controls ++(0,-1.15) and ++(0,-1.15) ..
%   (-1.5,-.6);
    \path[draw, blue, very thick, fill=white]
    (-0.15,0) .. controls ++(0,.2) and ++(0,.2) .. (0.15,0)
            .. controls ++(0,-.2) and ++(0,-.2) .. (-0.15,0);
\end{tikzpicture}}};
(8,12)*{\lambda};
  (8,16)*{\lcap};
  (8,22)*{\xlcap};
 (8,14)*{\scap};
 (8,-6)*{\scupef};(8,-14)*{\xlcupef};
 (0,0)*{\ecross};(-4,8)*{\seline};(4,8)*{\seline};
 (-12,4)*{\meline};
 (28,4)*{\mfline};
 (20,4)*{\mfline};
 (12,4)*{\mfline};
 (8,-10)*{\lcupef};
 (-6,-3)*{\scs i}; (6,-3)*{\scs i};(-15,-3)*{\scs j};
(12,5)*{\bullet}+(3,3)*{\scs c};(-12,5)*{\bullet}+(3,3)*{\scs b}; (-4,5)*{\bullet}+(4,3)*{\scs {a+1}};
 \endxy}
\end{equation}
where we freely slide dots through caps and cups using the dot cyclicity relation~\eqref{eq_cyclic_dot}.
The second term on the right-hand-side of \eqref{eq_vv1} can also be simplified
\[t^{-1}_{ij} \;
 \vcenter{
 \xy 0;/r.18pc/:
    (12,-4)*{};(4,4)*{} **\crv{(12,-1) & (4,1)}?(1)*\dir{>};
    (4,-4)*{};(12,4)*{} **\crv{(4,-1) & (12,1)}?(1)*\dir{>};
    (4,4)*{};(-4,12)*{} **\crv{(-,7) & (-4,9)}?(1)*\dir{>};
    (-4,4)*{};(4,12)*{} **\crv{(-4,7) & (4,9)}?(1)*\dir{>};
    (12,12)*{};(4,20)*{} **\crv{(12,15) & (4,17)}?(1)*\dir{>};
    (4,12)*{};(12,20)*{} **\crv{(4,15) & (12,17)}?(1)*\dir{>};
    (12,4)*{}; (12,12) **\dir{-};
    (-4,-4)*{}; (-4,4) **\dir{-};
    (-4,12)*{}; (-4,20) **\dir{-};
    (16,4)*{ \hackcenter{\begin{tikzpicture}
    \path[draw, blue, very thick, fill=white]
    (-0.15,0) .. controls ++(0,.2) and ++(0,.2) .. (0.15,0)
            .. controls ++(0,-.2) and ++(0,-.2) .. (-0.15,0);
\end{tikzpicture}}};
(16,12)*{\lambda};
  (15,-3)*{\scs i};
  (2,-3)*{\scs j};
  (-6,-3)*{\scs i};(20,5)*{\bullet}+(3,3)*{\scs c};(28,5)*{\bullet}+(3,3)*{\scs b}; (36,5)*{\bullet}+(3,3)*{\scs a};
   (20,8)*{\lfline};(28,8)*{\lfline};(36,8)*{\lfline};
  (16,22)*{\scap};(16,24)*{\lcap};(16,30)*{\xlcap};
  (16,-6)*{\scupef};(16,-10)*{\lcupef};(16,-15)*{\xlcupef};
\endxy} \;
 = \;  t_{ij} \;\vcenter{ \xy 0;/r.18pc/:
 (8,4)*{ \hackcenter{\begin{tikzpicture}
    \path[draw, blue, very thick, fill=white]
    (-0.15,0) .. controls ++(0,.2) and ++(0,.2) .. (0.15,0)
            .. controls ++(0,-.2) and ++(0,-.2) .. (-0.15,0);
\end{tikzpicture}}};
(8,12)*{\lambda};
  (8,16)*{\lcap};
  (8,22)*{\xlcap};
 (8,14)*{\scap};
 (8,-6)*{\scupef};(8,-14)*{\xlcupef};
 (-8,0)*{\ecross};(-12,8)*{\seline};(-4,8)*{\seline};
 (4,4)*{\meline};
 (16,8)*{\fcross};
 (16,0)*{\fcross};
 (28,4)*{\mfline};
 (8,-10)*{\lcupef};
 (-2,-3)*{\scs i}; (6,-3)*{\scs j};(-15,-3)*{\scs i};(12,5)*{\bullet}+(3,1)*{\scs c};(20,5)*{\bullet}+(3,2)*{\scs b}; (28,5)*{\bullet}+(3,2)*{\scs a};
 \endxy}
=\]\[
   \refequal{\eqref{eq_r2_ij-gen}} \;\;
   t_{ij}^2  \;\; \vcenter{ \xy 0;/r.18pc/:
 (8,-3)*{ \hackcenter{\begin{tikzpicture}
    \path[draw, blue, very thick, fill=white]
    (-0.15,0) .. controls ++(0,.2) and ++(0,.2) .. (0.15,0)
            .. controls ++(0,-.2) and ++(0,-.2) .. (-0.15,0);
\end{tikzpicture}}};
(8,4)*{\lambda};
  (8,8)*{\lcap};
  (8,14)*{\xlcap};
 (8,6)*{\scap};
 (8,-6)*{\scupef};(8,-14)*{\xlcupef};
 (-8,0)*{\ecross};
 (4,0)*{\seline};
 (12,0)*{\sfline};
 (20,0)*{\sfline};
 (28,0)*{\sfline};
 (8,-10)*{\lcupef};
 (-2,-3)*{\scs i}; (15,-3)*{\scs j};(-15,-3)*{\scs i};
  (12,0)*{\bullet}+(4,3)*{\scs b+1};(20,0)*{\bullet}+(3,3)*{\scs c}; (28,0)*{\bullet}+(3,3)*{\scs a};
  \endxy}
 \;\; + \;\;t_{ij}t_{ji}
\vcenter{ \xy 0;/r.18pc/:
 (8,-3)*{ \hackcenter{\begin{tikzpicture}
    \path[draw, blue, very thick, fill=white]
    (-0.15,0) .. controls ++(0,.2) and ++(0,.2) .. (0.15,0)
            .. controls ++(0,-.2) and ++(0,-.2) .. (-0.15,0);
\end{tikzpicture}}};
(8,4)*{\lambda};  (8,8)*{\lcap};
  (8,14)*{\xlcap};
 (8,6)*{\scap};
 (8,-6)*{\scupef};(8,-14)*{\xlcupef};
 (-8,0)*{\ecross};
 (4,0)*{\seline};
 (12,0)*{\sfline};
 (20,0)*{\sfline};
 (28,0)*{\sfline};
 (8,-10)*{\lcupef};
 (-2,-3)*{\scs i}; (15,-3)*{\scs j};(-15,-3)*{\scs i};
   (12,0)*{\bullet}+(3,3)*{\scs b};(20,0)*{\bullet}+(4,3)*{\scs c+1}; (28,0)*{\bullet}+(3,3)*{\scs a};
 \endxy} \nn
\]
Combining terms and simplifying using that $t_{ij}^2=t_{ji}^2=1, \, v_{ij}=1$ for our choice of scalars, we have
\begin{equation} \label{fin_eq}
 \xy
 (-6,0)*{};
  (6,0)*{};
 (0,-2)*{ \hackcenter{\begin{tikzpicture}
    \path[draw, blue, very thick, fill=white]
    (-0.15,0) .. controls ++(0,.2) and ++(0,.2) .. (0.15,0)
            .. controls ++(0,-.2) and ++(0,-.2) .. (-0.15,0);
\end{tikzpicture}}};
(0,2)*{\lambda};  (-4,0)*{}="t1";  (4,0)*{}="t2";
 "t2";"t1" **\crv{(4,6) & (-4,6)}; ?(.02)*\dir{<}
   ?(1)*\dir{<} ?(.3)*\dir{}+(2,2)*{\scs i};
  "t2";"t1" **\crv{(4,-6) & (-4,-6)};
  ?(.7)*\dir{}+(0,0)*{\bullet}+(2,-3)*{\scs c};
  (-12,0)*{}="t1";  (12,0)*{}="t2";
  "t2";"t1" **\crv{(12,15) & (-12,15)}; ?(0)*\dir{<} ?(1)*\dir{<} ?(.3)*\dir{}+(2,2)*{\scs j};
  "t2";"t1" **\crv{(12,-15) & (-12,-15)};
  ?(.7)*\dir{}+(0,0)*{\bullet}+(2,-4)*{\scs b};
   (-20,0)*{}="t1";  (20,0)*{}="t2";
  "t2";"t1" **\crv{(20,24) & (-20,24)}; ?(0)*\dir{<} ?(1)*\dir{<} ?(.3)*\dir{}+(2,2)*{\scs i};
  "t2";"t1" **\crv{(20,-24) & (-20,-24)};
  ?(.7)*\dir{}+(0,0)*{\bullet}+(2,-4)*{\scs a};
  \endxy
 \;\; = \;\;
   \vcenter{ \xy 0;/r.18pc/:
 (8,4)*{ \hackcenter{\begin{tikzpicture}
%  \path[draw,blue, very thick, fill=blue!10]
%   (-1.5,-.6) to (-1.5,.6) .. controls ++(0,1.15) and ++(0,1.15) ..
%   (1.5,.6)
%   to (1.5,-.6)  .. controls ++(0,-1.15) and ++(0,-1.15) ..
%   (-1.5,-.6);
    \path[draw, blue, very thick, fill=white]
    (-0.15,0) .. controls ++(0,.2) and ++(0,.2) .. (0.15,0)
            .. controls ++(0,-.2) and ++(0,-.2) .. (-0.15,0);
\end{tikzpicture}}};
(8,12)*{\lambda}; (8,16)*{\lcap};  (8,22)*{\xlcap}; (8,14)*{\scap};  (8,-6)*{\scupef};
  (8,-14)*{\xlcupef}; (0,0)*{\ecross}; (-4,8)*{\seline}; (4,8)*{\seline}; (-12,4)*{\meline};
 (28,4)*{\mfline}; (20,4)*{\mfline}; (12,4)*{\mfline}; (8,-10)*{\lcupef}; (-6,-3)*{\scs i};
  (6,-3)*{\scs i};  (-15,-3)*{\scs j}; (12,5)*{\bullet}+(3,3)*{\scs c};(-12,5)*{\bullet}+(4,3)*{\scs {b+1}}; (-4,5)*{\bullet}+(3,3)*{\scs a};
 \endxy}\;
 + \;
\vcenter{ \xy 0;/r.18pc/:
 (8,4)*{ \hackcenter{\begin{tikzpicture}
%  \path[draw,blue, very thick, fill=blue!10]
%   (-1.5,-.6) to (-1.5,.6) .. controls ++(0,1.15) and ++(0,1.15) ..
%   (1.5,.6)
%   to (1.5,-.6)  .. controls ++(0,-1.15) and ++(0,-1.15) ..
%   (-1.5,-.6);
    \path[draw, blue, very thick, fill=white]
    (-0.15,0) .. controls ++(0,.2) and ++(0,.2) .. (0.15,0)
            .. controls ++(0,-.2) and ++(0,-.2) .. (-0.15,0);
\end{tikzpicture}}};
(8,12)*{\lambda};
  (8,16)*{\lcap};
  (8,22)*{\xlcap};
 (8,14)*{\scap};
 (8,-6)*{\scupef};(8,-14)*{\xlcupef};
 (0,0)*{\ecross};(-4,8)*{\seline};(4,8)*{\seline};
 (-12,4)*{\meline};
 (28,4)*{\mfline};
 (20,4)*{\mfline};
 (12,4)*{\mfline};
 (8,-10)*{\lcupef};
 (-6,-3)*{\scs i}; (6,-3)*{\scs i};(-15,-3)*{\scs j};
(12,5)*{\bullet}+(3,3)*{\scs c};(-12,5)*{\bullet}+(3,3)*{\scs b}; (-4,5)*{\bullet}+(4,3)*{\scs {a+1}};
 \endxy}
   \end{equation}
   \[ - \;\; \vcenter{ \xy 0;/r.18pc/:
 (8,-3)*{ \hackcenter{\begin{tikzpicture}
    \path[draw, blue, very thick, fill=white]
    (-0.15,0) .. controls ++(0,.2) and ++(0,.2) .. (0.15,0)
            .. controls ++(0,-.2) and ++(0,-.2) .. (-0.15,0);
\end{tikzpicture}}};
(8,4)*{\lambda};
  (8,8)*{\lcap};
  (8,14)*{\xlcap};
 (8,6)*{\scap};
 (8,-6)*{\scupef};(8,-14)*{\xlcupef};
 (-8,0)*{\ecross};
 (4,0)*{\seline};
 (12,0)*{\sfline};
 (20,0)*{\sfline};
 (28,0)*{\sfline};
 (8,-10)*{\lcupef};
 (-2,-3)*{\scs i}; (15,-3)*{\scs j};(-15,-3)*{\scs i};
  (12,0)*{\bullet}+(4,3)*{\scs b+1};(20,0)*{\bullet}+(3,3)*{\scs c}; (28,0)*{\bullet}+(3,3)*{\scs a};
  \endxy}
 \;\; - \;\;
\vcenter{ \xy 0;/r.18pc/:
 (8,-3)*{ \hackcenter{\begin{tikzpicture}
    \path[draw, blue, very thick, fill=white]
    (-0.15,0) .. controls ++(0,.2) and ++(0,.2) .. (0.15,0)
            .. controls ++(0,-.2) and ++(0,-.2) .. (-0.15,0);
\end{tikzpicture}}};
(8,4)*{\lambda};  (8,8)*{\lcap};
  (8,14)*{\xlcap};
 (8,6)*{\scap};
 (8,-6)*{\scupef};(8,-14)*{\xlcupef};
 (-8,0)*{\ecross};
 (4,0)*{\seline};
 (12,0)*{\sfline};
 (20,0)*{\sfline};
 (28,0)*{\sfline};
 (8,-10)*{\lcupef};
 (-2,-3)*{\scs i}; (15,-3)*{\scs j};(-15,-3)*{\scs i};
   (12,0)*{\bullet}+(3,3)*{\scs b};(20,0)*{\bullet}+(4,3)*{\scs c+1}; (28,0)*{\bullet}+(3,3)*{\scs a};
 \endxy} \nn.\]
By interchanging \(a\) and \(c\) in the equation  \eqref{fin_eq}, we get
 \begin{equation} \label{fin_eq2}
 \xy
 (-6,0)*{};
  (6,0)*{};
 (0,-2)*{ \hackcenter{\begin{tikzpicture}
    \path[draw, blue, very thick, fill=white]
    (-0.15,0) .. controls ++(0,.2) and ++(0,.2) .. (0.15,0)
            .. controls ++(0,-.2) and ++(0,-.2) .. (-0.15,0);
\end{tikzpicture}}};
(0,2)*{\lambda};  (-4,0)*{}="t1";  (4,0)*{}="t2";
 "t2";"t1" **\crv{(4,6) & (-4,6)}; ?(.02)*\dir{<}
   ?(1)*\dir{<} ?(.3)*\dir{}+(2,2)*{\scs i};
  "t2";"t1" **\crv{(4,-6) & (-4,-6)};
  ?(.7)*\dir{}+(0,0)*{\bullet}+(2,-3)*{\scs a};
  (-12,0)*{}="t1";  (12,0)*{}="t2";
  "t2";"t1" **\crv{(12,15) & (-12,15)}; ?(0)*\dir{<} ?(1)*\dir{<} ?(.3)*\dir{}+(2,2)*{\scs j};
  "t2";"t1" **\crv{(12,-15) & (-12,-15)};
  ?(.7)*\dir{}+(0,0)*{\bullet}+(2,-4)*{\scs b};
   (-20,0)*{}="t1";  (20,0)*{}="t2";
  "t2";"t1" **\crv{(20,24) & (-20,24)}; ?(0)*\dir{<} ?(1)*\dir{<} ?(.3)*\dir{}+(2,2)*{\scs i};
  "t2";"t1" **\crv{(20,-24) & (-20,-24)};
  ?(.7)*\dir{}+(0,0)*{\bullet}+(2,-4)*{\scs c};
  \endxy
 \;\; = \;\;
   \vcenter{ \xy 0;/r.18pc/:
 (8,4)*{ \hackcenter{\begin{tikzpicture}
%  \path[draw,blue, very thick, fill=blue!10]
%   (-1.5,-.6) to (-1.5,.6) .. controls ++(0,1.15) and ++(0,1.15) ..
%   (1.5,.6)
%   to (1.5,-.6)  .. controls ++(0,-1.15) and ++(0,-1.15) ..
%   (-1.5,-.6);
    \path[draw, blue, very thick, fill=white]
    (-0.15,0) .. controls ++(0,.2) and ++(0,.2) .. (0.15,0)
            .. controls ++(0,-.2) and ++(0,-.2) .. (-0.15,0);
\end{tikzpicture}}};
(8,12)*{\lambda}; (8,16)*{\lcap};  (8,22)*{\xlcap}; (8,14)*{\scap};  (8,-6)*{\scupef};
  (8,-14)*{\xlcupef}; (0,0)*{\ecross}; (-4,8)*{\seline}; (4,8)*{\seline}; (-12,4)*{\meline};
 (28,4)*{\mfline}; (20,4)*{\mfline}; (12,4)*{\mfline}; (8,-10)*{\lcupef}; (-6,-3)*{\scs i};
  (6,-3)*{\scs i};  (-15,-3)*{\scs j}; (12,5)*{\bullet}+(3,3)*{\scs a};(-12,5)*{\bullet}+(4,3)*{\scs {b+1}}; (-4,5)*{\bullet}+(3,3)*{\scs c};
 \endxy}\;
 + \;
\vcenter{ \xy 0;/r.18pc/:
 (8,4)*{ \hackcenter{\begin{tikzpicture}
%  \path[draw,blue, very thick, fill=blue!10]
%   (-1.5,-.6) to (-1.5,.6) .. controls ++(0,1.15) and ++(0,1.15) ..
%   (1.5,.6)
%   to (1.5,-.6)  .. controls ++(0,-1.15) and ++(0,-1.15) ..
%   (-1.5,-.6);
    \path[draw, blue, very thick, fill=white]
    (-0.15,0) .. controls ++(0,.2) and ++(0,.2) .. (0.15,0)
            .. controls ++(0,-.2) and ++(0,-.2) .. (-0.15,0);
\end{tikzpicture}}};
(8,12)*{\lambda};
  (8,16)*{\lcap};
  (8,22)*{\xlcap};
 (8,14)*{\scap};
 (8,-6)*{\scupef};(8,-14)*{\xlcupef};
 (0,0)*{\ecross};(-4,8)*{\seline};(4,8)*{\seline};
 (-12,4)*{\meline};
 (28,4)*{\mfline};
 (20,4)*{\mfline};
 (12,4)*{\mfline};
 (8,-10)*{\lcupef};
 (-6,-3)*{\scs i}; (6,-3)*{\scs i};(-15,-3)*{\scs j};
(12,5)*{\bullet}+(3,3)*{\scs a};(-12,5)*{\bullet}+(3,3)*{\scs b}; (-4,5)*{\bullet}+(4,3)*{\scs {c+1}};
 \endxy}
   \end{equation}
   \[ - \;\; \vcenter{ \xy 0;/r.18pc/:
 (8,-3)*{ \hackcenter{\begin{tikzpicture}
    \path[draw, blue, very thick, fill=white]
    (-0.15,0) .. controls ++(0,.2) and ++(0,.2) .. (0.15,0)
            .. controls ++(0,-.2) and ++(0,-.2) .. (-0.15,0);
\end{tikzpicture}}};
(8,4)*{\lambda};
  (8,8)*{\lcap};
  (8,14)*{\xlcap};
 (8,6)*{\scap};
 (8,-6)*{\scupef};(8,-14)*{\xlcupef};
 (-8,0)*{\ecross};
 (4,0)*{\seline};
 (12,0)*{\sfline};
 (20,0)*{\sfline};
 (28,0)*{\sfline};
 (8,-10)*{\lcupef};
 (-2,-3)*{\scs i}; (15,-3)*{\scs j};(-15,-3)*{\scs i};
  (12,0)*{\bullet}+(4,3)*{\scs b+1};(20,0)*{\bullet}+(3,3)*{\scs a}; (28,0)*{\bullet}+(3,3)*{\scs c};
  \endxy}
 \;\; - \;\;
\vcenter{ \xy 0;/r.18pc/:
 (8,-3)*{ \hackcenter{\begin{tikzpicture}
    \path[draw, blue, very thick, fill=white]
    (-0.15,0) .. controls ++(0,.2) and ++(0,.2) .. (0.15,0)
            .. controls ++(0,-.2) and ++(0,-.2) .. (-0.15,0);
\end{tikzpicture}}};
(8,4)*{\lambda};
  (8,8)*{\lcap};
  (8,14)*{\xlcap};
 (8,6)*{\scap};
 (8,-6)*{\scupef};(8,-14)*{\xlcupef};
 (-8,0)*{\ecross};
 (4,0)*{\seline};
 (12,0)*{\sfline};
 (20,0)*{\sfline};
 (28,0)*{\sfline};
 (8,-10)*{\lcupef};
 (-2,-3)*{\scs i}; (15,-3)*{\scs j};(-15,-3)*{\scs i};
   (12,0)*{\bullet}+(3,3)*{\scs b};(20,0)*{\bullet}+(4,3)*{\scs a+1}; (28,0)*{\bullet}+(3,3)*{\scs c};
 \endxy} \nn.\]
 Using the relation \eqref{eq_bubble_rel} and induction, one can easily prove the following equation:
 \begin{equation} \label{ac}
  \xy 0;/r.16pc/:
   (8,-3)*{
 \hackcenter{\begin{tikzpicture}
  \path[draw, blue, very thick, fill=white]
    (-0.15,0) .. controls ++(0,.2) and ++(0,.2) .. (0.15,0)
            .. controls ++(0,-.2) and ++(0,-.2) .. (-0.15,0);
  \end{tikzpicture}}
 };
  (8,8)*{\lcap};
 (8,6)*{\scap};
 (8,-6)*{\scup};
 (0,0)*{\naecross};
 (12,0)*{\sline};
 (20,0)*{\sline};
 (8,-10)*{\lcup};
  (-3.3,6)*{\bullet}+(-4,0)*{\scs a}; (4.7,6)*{\bullet}+(-2.7,0)*{\scs c};
  (22,0)*{\scs i};
 (8,3)*{\lambda};
 \endxy\; \; = - \;\;
 \xy 0;/r.16pc/:
   (8,-3)*{
 \hackcenter{\begin{tikzpicture}
  \path[draw, blue, very thick, fill=white]
    (-0.15,0) .. controls ++(0,.2) and ++(0,.2) .. (0.15,0)
            .. controls ++(0,-.2) and ++(0,-.2) .. (-0.15,0);
  \end{tikzpicture}}
 };
  (8,8)*{\lcap};
 (8,6)*{\scap};
 (8,-6)*{\scup};
 (0,0)*{\naecross};
 (12,0)*{\sline};
 (20,0)*{\sline};
 (8,-10)*{\lcup};
 (-3.3,6)*{\bullet}+(-4,0)*{\scs c}; (4.7,6)*{\bullet}+(-2.7,0)*{\scs a};
  (22,0)*{\scs i}; (8,3)*{\lambda};
 \endxy.
 \end{equation}
 Also, the following equation directly follows from \eqref{eq_nil_dotslide} and \eqref{ac}
 \begin{equation} \label{ac2}
 \xy
 (-6,0)*{};
  (6,0)*{};
   (0,2)*{\lambda};(0,-2)*{ \hackcenter{\begin{tikzpicture}
    \path[draw, blue, very thick, fill=white]
    (-0.15,0) .. controls ++(0,.2) and ++(0,.2) .. (0.15,0)
            .. controls ++(0,-.2) and ++(0,-.2) .. (-0.15,0);
\end{tikzpicture}}};
  (-4,0)*{}="t1";  (4,0)*{}="t2";
 "t2";"t1" **\crv{(4,6) & (-4,6)}; ?(.02)*\dir{<}
   ?(1)*\dir{<} ?(.3)*\dir{}+(2,2)*{\scs i};
  "t2";"t1" **\crv{(4,-6) & (-4,-6)};
  ?(.7)*\dir{}+(0,0)*{\bullet}+(2,-3)*{\scs c};
  (-12,0)*{}="t1";  (12,0)*{}="t2";
  "t2";"t1" **\crv{(12,15) & (-12,15)}; ?(0)*\dir{<} ?(1)*\dir{<} ?(.3)*\dir{}+(2,2)*{\scs i};
  "t2";"t1" **\crv{(12,-15) & (-12,-15)};
  ?(.7)*\dir{}+(0,0)*{\bullet}+(2,-4)*{\scs a};
  \endxy \;\;=\;\;
 \xy 0;/r.16pc/:
   (8,-3)*{
 \hackcenter{\begin{tikzpicture}
  \path[draw, blue, very thick, fill=white]
    (-0.15,0) .. controls ++(0,.2) and ++(0,.2) .. (0.15,0)
            .. controls ++(0,-.2) and ++(0,-.2) .. (-0.15,0);
  \end{tikzpicture}}
 };
  (8,8)*{\lcap};
 (8,6)*{\scap};
 (8,-6)*{\scup};
 (0,0)*{\naecross};
 (12,0)*{\sline};
 (20,0)*{\sline};
 (8,-10)*{\lcup};
 (-3.3,6)*{\bullet}+(-5,0)*{\scs a+1}; (4.7,6)*{\bullet}+(-2.7,0)*{\scs c};
  (22,0)*{\scs i}; (8,3)*{\lambda};
 \endxy\;\; + \;\;
 \xy 0;/r.16pc/:
   (8,-3)*{
 \hackcenter{\begin{tikzpicture}
  \path[draw, blue, very thick, fill=white]
    (-0.15,0) .. controls ++(0,.2) and ++(0,.2) .. (0.15,0)
            .. controls ++(0,-.2) and ++(0,-.2) .. (-0.15,0);
  \end{tikzpicture}}
 };
  (8,8)*{\lcap};
 (8,6)*{\scap};
 (8,-6)*{\scup};
 (0,0)*{\naecross};
 (12,0)*{\sline};
 (20,0)*{\sline};
 (8,-10)*{\lcup};
 (-3.3,6)*{\bullet}+(-5,0)*{\scs c+1}; (4.7,6)*{\bullet}+(-3.5,0)*{\scs a};
  (22,0)*{\scs i}; (8,3)*{\lambda};
 \endxy.
 \end{equation}
We add  the right and left hand sides of the equations \eqref{fin_eq} and \eqref{fin_eq2}, and after eliminating terms using \eqref{ac}, we get
\begin{equation}\label{fin_eq3}
 \xy
 (-6,0)*{};
  (6,0)*{};
(0,-2)*{ \hackcenter{\begin{tikzpicture}
    \path[draw, blue, very thick, fill=white]
    (-0.15,0) .. controls ++(0,.2) and ++(0,.2) .. (0.15,0)
            .. controls ++(0,-.2) and ++(0,-.2) .. (-0.15,0);
\end{tikzpicture}}};
(0,2)*{\lambda};   (-4,0)*{}="t1";  (4,0)*{}="t2";
 "t2";"t1" **\crv{(4,6) & (-4,6)}; ?(.02)*\dir{<}
   ?(1)*\dir{<} ?(.3)*\dir{}+(2,2)*{\scs i};
  "t2";"t1" **\crv{(4,-6) & (-4,-6)};
  ?(.7)*\dir{}+(0,0)*{\bullet}+(2,-3)*{\scs a};
  (-12,0)*{}="t1";  (12,0)*{}="t2";
  "t2";"t1" **\crv{(12,15) & (-12,15)}; ?(0)*\dir{<} ?(1)*\dir{<} ?(.3)*\dir{}+(2,2)*{\scs j};
  "t2";"t1" **\crv{(12,-15) & (-12,-15)};
  ?(.7)*\dir{}+(0,0)*{\bullet}+(2,-4)*{\scs b};
   (-20,0)*{}="t1";  (20,0)*{}="t2";
  "t2";"t1" **\crv{(20,24) & (-20,24)}; ?(0)*\dir{<} ?(1)*\dir{<} ?(.3)*\dir{}+(2,2)*{\scs i};
  "t2";"t1" **\crv{(20,-24) & (-20,-24)};
  ?(.7)*\dir{}+(0,0)*{\bullet}+(2,-4)*{\scs c};
  \endxy
 \;\; +\;\;
 \xy
 (-6,0)*{};
  (6,0)*{};
(0,-2)*{ \hackcenter{\begin{tikzpicture}
    \path[draw, blue, very thick, fill=white]
    (-0.15,0) .. controls ++(0,.2) and ++(0,.2) .. (0.15,0)
            .. controls ++(0,-.2) and ++(0,-.2) .. (-0.15,0);
\end{tikzpicture}}};
(0,2)*{\lambda};   (-4,0)*{}="t1";  (4,0)*{}="t2";
 "t2";"t1" **\crv{(4,6) & (-4,6)}; ?(.02)*\dir{<}
   ?(1)*\dir{<} ?(.3)*\dir{}+(2,2)*{\scs i};
  "t2";"t1" **\crv{(4,-6) & (-4,-6)};
  ?(.7)*\dir{}+(0,0)*{\bullet}+(2,-3)*{\scs c};
  (-12,0)*{}="t1";  (12,0)*{}="t2";
  "t2";"t1" **\crv{(12,15) & (-12,15)}; ?(0)*\dir{<} ?(1)*\dir{<} ?(.3)*\dir{}+(2,2)*{\scs j};
  "t2";"t1" **\crv{(12,-15) & (-12,-15)};
  ?(.7)*\dir{}+(0,0)*{\bullet}+(2,-4)*{\scs b};
   (-20,0)*{}="t1";  (20,0)*{}="t2";
  "t2";"t1" **\crv{(20,24) & (-20,24)}; ?(0)*\dir{<} ?(1)*\dir{<} ?(.3)*\dir{}+(2,2)*{\scs i};
  "t2";"t1" **\crv{(20,-24) & (-20,-24)};
  ?(.7)*\dir{}+(0,0)*{\bullet}+(2,-4)*{\scs a};
  \endxy
 \;\; =
   \end{equation}
   \[ =\; \;
\vcenter{ \xy 0;/r.18pc/:
 (8,4)*{ \hackcenter{\begin{tikzpicture}
%  \path[draw,blue, very thick, fill=blue!10]
%   (-1.5,-.6) to (-1.5,.6) .. controls ++(0,1.15) and ++(0,1.15) ..
%   (1.5,.6)
%   to (1.5,-.6)  .. controls ++(0,-1.15) and ++(0,-1.15) ..
%   (-1.5,-.6);
    \path[draw, blue, very thick, fill=white]
    (-0.15,0) .. controls ++(0,.2) and ++(0,.2) .. (0.15,0)
            .. controls ++(0,-.2) and ++(0,-.2) .. (-0.15,0);
\end{tikzpicture}}};
(8,12)*{\lambda};
  (8,16)*{\lcap};
  (8,22)*{\xlcap};
 (8,14)*{\scap};
 (8,-6)*{\scupef};(8,-14)*{\xlcupef};
 (0,0)*{\ecross};(-4,8)*{\seline};(4,8)*{\seline};
 (-12,4)*{\meline};
 (28,4)*{\mfline};
 (20,4)*{\mfline};
 (12,4)*{\mfline};
 (8,-10)*{\lcupef};
 (-6,-3)*{\scs i}; (6,-3)*{\scs i};(-15,-3)*{\scs j};
(12,5)*{\bullet}+(3,3)*{\scs c};(-12,5)*{\bullet}+(3,3)*{\scs b}; (-4,5)*{\bullet}+(4,3)*{\scs {a+1}};
 \endxy} - \;\;
\vcenter{ \xy 0;/r.18pc/:
(8,-3)*{ \hackcenter{\begin{tikzpicture}
    \path[draw, blue, very thick, fill=white]
    (-0.15,0) .. controls ++(0,.2) and ++(0,.2) .. (0.15,0)
            .. controls ++(0,-.2) and ++(0,-.2) .. (-0.15,0);
\end{tikzpicture}}};
(8,4)*{\lambda};  (8,8)*{\lcap};
  (8,14)*{\xlcap};
 (8,6)*{\scap};
 (8,-6)*{\scupef};(8,-14)*{\xlcupef};
 (-8,0)*{\ecross};
 (4,0)*{\seline};
 (12,0)*{\sfline};
 (20,0)*{\sfline};
 (28,0)*{\sfline};
 (8,-10)*{\lcupef};
 (-2,-3)*{\scs i}; (15,-3)*{\scs j};(-15,-3)*{\scs i};
   (12,0)*{\bullet}+(3,3)*{\scs b};(20,0)*{\bullet}+(4,3)*{\scs c+1}; (28,0)*{\bullet}+(3,3)*{\scs a};
 \endxy} \nn \;\;+
\]
\[+\;\;
\vcenter{ \xy 0;/r.18pc/:
 (8,4)*{ \hackcenter{\begin{tikzpicture}
%  \path[draw,blue, very thick, fill=blue!10]
%   (-1.5,-.6) to (-1.5,.6) .. controls ++(0,1.15) and ++(0,1.15) ..
%   (1.5,.6)
%   to (1.5,-.6)  .. controls ++(0,-1.15) and ++(0,-1.15) ..
%   (-1.5,-.6);
    \path[draw, blue, very thick, fill=white]
    (-0.15,0) .. controls ++(0,.2) and ++(0,.2) .. (0.15,0)
            .. controls ++(0,-.2) and ++(0,-.2) .. (-0.15,0);
\end{tikzpicture}}};
(8,12)*{\lambda};
  (8,16)*{\lcap};
  (8,22)*{\xlcap};
 (8,14)*{\scap};
 (8,-6)*{\scupef};(8,-14)*{\xlcupef};
 (0,0)*{\ecross};(-4,8)*{\seline};(4,8)*{\seline};
 (-12,4)*{\meline};
 (28,4)*{\mfline};
 (20,4)*{\mfline};
 (12,4)*{\mfline};
 (8,-10)*{\lcupef};
 (-6,-3)*{\scs i}; (6,-3)*{\scs i};(-15,-3)*{\scs j};
(12,5)*{\bullet}+(3,3)*{\scs a};(-12,5)*{\bullet}+(3,3)*{\scs b}; (-4,5)*{\bullet}+(4,3)*{\scs {c+1}};
 \endxy} \;\; - \;\;
\vcenter{ \xy 0;/r.18pc/:
(8,-3)*{ \hackcenter{\begin{tikzpicture}
    \path[draw, blue, very thick, fill=white]
    (-0.15,0) .. controls ++(0,.2) and ++(0,.2) .. (0.15,0)
            .. controls ++(0,-.2) and ++(0,-.2) .. (-0.15,0);
\end{tikzpicture}}};
(8,4)*{\lambda};  (8,8)*{\lcap};
  (8,14)*{\xlcap};
 (8,6)*{\scap};
 (8,-6)*{\scupef};(8,-14)*{\xlcupef};
 (-8,0)*{\ecross};
 (4,0)*{\seline};
 (12,0)*{\sfline};
 (20,0)*{\sfline};
 (28,0)*{\sfline};
 (8,-10)*{\lcupef};
 (-2,-3)*{\scs i}; (15,-3)*{\scs j};(-15,-3)*{\scs i};
   (12,0)*{\bullet}+(3,3)*{\scs b};(20,0)*{\bullet}+(4,3)*{\scs a+1}; (28,0)*{\bullet}+(3,3)*{\scs c};
 \endxy} \nn.\]
 Combining the respective terms using the relation \eqref{ac2} in the equation \eqref{fin_eq3}, we have
 \begin{equation}
 \xy
 (-6,0)*{};
  (6,0)*{};
(0,-2)*{ \hackcenter{\begin{tikzpicture}
    \path[draw, blue, very thick, fill=white]
    (-0.15,0) .. controls ++(0,.2) and ++(0,.2) .. (0.15,0)
            .. controls ++(0,-.2) and ++(0,-.2) .. (-0.15,0);
\end{tikzpicture}}};
(0,2)*{\lambda};   (-4,0)*{}="t1";  (4,0)*{}="t2";
 "t2";"t1" **\crv{(4,6) & (-4,6)}; ?(.02)*\dir{<}
   ?(1)*\dir{<} ?(.3)*\dir{}+(2,2)*{\scs i};
  "t2";"t1" **\crv{(4,-6) & (-4,-6)};
  ?(.7)*\dir{}+(0,0)*{\bullet}+(2,-3)*{\scs a};
  (-12,0)*{}="t1";  (12,0)*{}="t2";
  "t2";"t1" **\crv{(12,15) & (-12,15)}; ?(0)*\dir{<} ?(1)*\dir{<} ?(.3)*\dir{}+(2,2)*{\scs j};
  "t2";"t1" **\crv{(12,-15) & (-12,-15)};
  ?(.7)*\dir{}+(0,0)*{\bullet}+(2,-4)*{\scs b};
   (-20,0)*{}="t1";  (20,0)*{}="t2";
  "t2";"t1" **\crv{(20,24) & (-20,24)}; ?(0)*\dir{<} ?(1)*\dir{<} ?(.3)*\dir{}+(2,2)*{\scs i};
  "t2";"t1" **\crv{(20,-24) & (-20,-24)};
  ?(.7)*\dir{}+(0,0)*{\bullet}+(2,-4)*{\scs c};
  \endxy
 \;\; +\;\;
 \xy
 (-6,0)*{};
  (6,0)*{};
(0,-2)*{ \hackcenter{\begin{tikzpicture}
    \path[draw, blue, very thick, fill=white]
    (-0.15,0) .. controls ++(0,.2) and ++(0,.2) .. (0.15,0)
            .. controls ++(0,-.2) and ++(0,-.2) .. (-0.15,0);
\end{tikzpicture}}};
(0,2)*{\lambda};   (-4,0)*{}="t1";  (4,0)*{}="t2";
 "t2";"t1" **\crv{(4,6) & (-4,6)}; ?(.02)*\dir{<}
   ?(1)*\dir{<} ?(.3)*\dir{}+(2,2)*{\scs i};
  "t2";"t1" **\crv{(4,-6) & (-4,-6)};
  ?(.7)*\dir{}+(0,0)*{\bullet}+(2,-3)*{\scs c};
  (-12,0)*{}="t1";  (12,0)*{}="t2";
  "t2";"t1" **\crv{(12,15) & (-12,15)}; ?(0)*\dir{<} ?(1)*\dir{<} ?(.3)*\dir{}+(2,2)*{\scs j};
  "t2";"t1" **\crv{(12,-15) & (-12,-15)};
  ?(.7)*\dir{}+(0,0)*{\bullet}+(2,-4)*{\scs b};
   (-20,0)*{}="t1";  (20,0)*{}="t2";
  "t2";"t1" **\crv{(20,24) & (-20,24)}; ?(0)*\dir{<} ?(1)*\dir{<} ?(.3)*\dir{}+(2,2)*{\scs i};
  "t2";"t1" **\crv{(20,-24) & (-20,-24)};
  ?(.7)*\dir{}+(0,0)*{\bullet}+(2,-4)*{\scs a};
  \endxy
 \;\; =
 \end{equation}
 \[=\;\;\xy
 (-6,0)*{};
  (6,0)*{};
(0,-2)*{ \hackcenter{\begin{tikzpicture}
    \path[draw, blue, very thick, fill=white]
    (-0.15,0) .. controls ++(0,.2) and ++(0,.2) .. (0.15,0)
            .. controls ++(0,-.2) and ++(0,-.2) .. (-0.15,0);
\end{tikzpicture}}};
(0,2)*{\lambda};   (-4,0)*{}="t1";  (4,0)*{}="t2";
 "t2";"t1" **\crv{(4,6) & (-4,6)}; ?(.02)*\dir{<}
   ?(1)*\dir{<} ?(.3)*\dir{}+(2,2)*{\scs i};
  "t2";"t1" **\crv{(4,-6) & (-4,-6)};
  ?(.7)*\dir{}+(0,0)*{\bullet}+(2,-3)*{\scs c};
  (-12,0)*{}="t1";  (12,0)*{}="t2";
  "t2";"t1" **\crv{(12,15) & (-12,15)}; ?(0)*\dir{<} ?(1)*\dir{<} ?(.3)*\dir{}+(2,2)*{\scs i};
  "t2";"t1" **\crv{(12,-15) & (-12,-15)};
  ?(.7)*\dir{}+(0,0)*{\bullet}+(2,-4)*{\scs a};
   (-20,0)*{}="t1";  (20,0)*{}="t2";
  "t2";"t1" **\crv{(20,24) & (-20,24)}; ?(0)*\dir{<} ?(1)*\dir{<} ?(.3)*\dir{}+(2,2)*{\scs j};
  "t2";"t1" **\crv{(20,-24) & (-20,-24)};
  ?(.7)*\dir{}+(0,0)*{\bullet}+(2,-4)*{\scs b};
  \endxy
 \;\; +\;\;
 \xy
 (-6,0)*{};
  (6,0)*{};
(0,-2)*{ \hackcenter{\begin{tikzpicture}
    \path[draw, blue, very thick, fill=white]
    (-0.15,0) .. controls ++(0,.2) and ++(0,.2) .. (0.15,0)
            .. controls ++(0,-.2) and ++(0,-.2) .. (-0.15,0);
\end{tikzpicture}}};
(0,2)*{\lambda};   (-4,0)*{}="t1";  (4,0)*{}="t2";
 "t2";"t1" **\crv{(4,6) & (-4,6)}; ?(.02)*\dir{<}
   ?(1)*\dir{<} ?(.3)*\dir{}+(2,2)*{\scs j};
  "t2";"t1" **\crv{(4,-6) & (-4,-6)};
  ?(.7)*\dir{}+(0,0)*{\bullet}+(2,-3)*{\scs b};
  (-12,0)*{}="t1";  (12,0)*{}="t2";
  "t2";"t1" **\crv{(12,15) & (-12,15)}; ?(0)*\dir{<} ?(1)*\dir{<} ?(.3)*\dir{}+(2,2)*{\scs i};
  "t2";"t1" **\crv{(12,-15) & (-12,-15)};
  ?(.7)*\dir{}+(0,0)*{\bullet}+(2,-4)*{\scs c};
   (-20,0)*{}="t1";  (20,0)*{}="t2";
  "t2";"t1" **\crv{(20,24) & (-20,24)}; ?(0)*\dir{<} ?(1)*\dir{<} ?(.3)*\dir{}+(2,2)*{\scs i};
  "t2";"t1" **\crv{(20,-24) & (-20,-24)};
  ?(.7)*\dir{}+(0,0)*{\bullet}+(2,-4)*{\scs a};
  \endxy\;\;=
 \]
 \[= \mathsf{E}_{j,b}\mathsf{E}_{i,a}\mathsf{E}_{i,c} 1_{\lambda}+  \mathsf{E}_{i,a}\mathsf{E}_{i,c}\mathsf{E}_{j,b} 1_{\lambda}.
 \]
 Reversing the orientation of the arrows, we get the identity $$ \mathsf{F}_{i,a}\mathsf{F}_{j,b}\mathsf{F}_{i,c}1_{\lambda} +\mathsf{F}_{i,c}\mathsf{F}_{j,b}\mathsf{F}_{i,a}1_{\lambda} = \mathsf{F}_{i,a}\mathsf{F}_{i,c}\mathsf{F}_{j,b} 1_{\lambda}+\mathsf{F}_{j,b}\mathsf{F}_{i,a}\mathsf{F}_{i,c} 1_{\lambda}.$$
\end{proof}

% ---------------------------------------------------------------------
%
\section{Action on center of objects in $\Udot$}\label{action}
%
% ---------------------------------------------------------------------

This section is devoted to the proof of Theorem \ref{2-reps}. The action of the current algebra on the trace of any 2-representation follows
directly from the functotiality of the trace. The rest of the section establishes the action on the center.

% - - - - - - - - - - - - - - - - - - - - - - - - - - - - - - - - -
%
\subsubsection{Centers of categories}
%
% - - - - - - - - - - - - - - - - - - - - - - - - - - - - - - - - -

The center $Z(\C)$ of an additive category $\C$ is the endomorphism ring of the identity functor $\Id_{\C}$ on $\C$.  Note that an endomorphism of $\Id_{\C}$ is a natural transformation $\alpha \maps \Id_{\C} \To \Id_{\C}$.  Such a natural transformation consists of a map $\alpha_x \maps x \to x$ for each object $x$ of $\C$ satisfying the requirement that for any map $f \maps x \to y$ in $\C$, we get a commutative square
\[
 \xymatrix{
 x \ar[r]^{f} \ar[d]_{\alpha_x} & y \ar[d]^{\alpha_y} \\ x \ar[r]_{f} & y
 }
\]
\begin{exercise}
Show that $Z(\C)$ is commutative ring. Hint: use the Eckmann--Hilton argument.
\end{exercise}

An additive category is an object of the linear 2-category \cat{AdCat} of additive categories, additive functors, and natural transformations.  For an arbitrary, linear 2-category $\CC$ we define the center $Z(x)$ of an object $x \in \Ob( \CC)$ as the ring of endomorphisms $\CC(1_{x}, 1_{x})$, see \cite{GK}.  Define the {\em center of objects} of the 2-category $\CC$ as the $Z(\CC) = \bigoplus_{x \in \Ob(\CC)} Z(x)$.

The relationship between Hochschild homology and cohomology described in section~\ref{subsec_HHco} can now be recast in a more general framework.

\begin{exercise}
Let $\CC$ be a linear 2-category $\CC$. Show that for each $x \in \Ob(\CC)$, the abelian group $\Tr(\CC(x,x))$ is a module over the ring $Z(x)$.
\end{exercise}

\begin{example} \label{examp_sl2}
In the 2-category $\Ucat(\mf{sl}_2)$, an element in the center $Z(\lambda)$ for an object $\lambda \in \Ob(\Ucat(\mf{sl}_2))$ is given by any closed diagram in the graphical calculus.  It was shown in Section 8 of \cite{Lau1} that any such closed diagram can be reduced to a product of non-nested dotted bubbles with the same orientation.  As explained in section~\ref{sec_sym-bub}, this provides an isomorphism between $Z(\lambda)$ and the ring $\sym$ of symmetric functions.
\end{example}

Given a 2-endomorphism $f \maps \onelp x \onel \To \onelp x \onel$ in $\Ucat^{\ast}(\mf{sl}_n)$, we can interpret the class $[f]$ as a diagram on the annulus.  If we forget that the diagram is on an annulus, we are left with a diagram representing an endomorphism of $\onelp$.
\[
 \vcenter{ \xy 0;/r.16pc/:
 (8,0)*{
 \hackcenter{\begin{tikzpicture}
  \path[draw,blue, very thick, fill=blue!10]
   (-2,-.6) to (-2,.6) .. controls ++(0,1.85) and ++(0,1.85) .. (2,.6)
   to (2,-.6)  .. controls ++(0,-1.85) and ++(0,-1.85) .. (-2,-.6);
    \path[draw, blue, very thick, fill=white]
    (-0.25,0) .. controls ++(0,.35) and ++(0,.35) .. (0.25,0)
            .. controls ++(0,-.35) and ++(0,-.35) .. (-0.25,0);
\end{tikzpicture}}
 };
(10,8)*{\lambda}; (-15,12)*{\lambda'};
  (8,8)*{\lcap};
  (8,14)*{\xlcap};
% (8,6)*{\scap};
 %(8,-6)*{\scupef}; (15,-3)*{\scs j};
 (8,-14)*{\xlcupef};
 (-8,0)*{\ecross};
 (20,0)*{\sfline};
 (28,0)*{\sfline};
 (8,-10)*{\lcupef};
(-2,-3)*{\scs i}; (-15,-3)*{\scs i};
 (-10,1)*{\bullet};
 \endxy} \;\; \in \Tr(\Ucat^{\ast})
 \qquad \rightsquigarrow \qquad
 \vcenter{ \xy 0;/r.18pc/:
(10,8)*{\lambda}; (-15,12)*{\lambda'};
  (8,8)*{\lcap};
  (8,14)*{\xlcap};
% (8,6)*{\scap};
 %(8,-6)*{\scupef}; (15,-3)*{\scs j};
 (8,-14)*{\xlcupef};
 (-8,0)*{\ecross};
 (20,0)*{\sfline};
 (28,0)*{\sfline};
 (8,-10)*{\lcupef};
(-2,-3)*{\scs i}; (-15,-3)*{\scs i};
 (-10,1)*{\bullet};
 \endxy} \;\; \in Z(\lambda')
\]
By choosing scalars $t_ij=t_{ji}=1$ for all $i\neq j$, the 2-category $\Ucat(\mf{sl}_n))=\Ucat_Q(\mf{sl}_n))$ is cyclic and the resulting diagram is independent of the choice of representative for the class $[f] \in \Tr(\Ucat(\mf{sl}_n))$.
In this way, $[f]$ gives rise to a map $Z(\lambda) \to Z(\lambda')$ given by placing an element from $Z(\lambda)$ inside the annulus and regarding the result as an element in $Z(\lambda')$.  Hence, we have the following proposition.

\begin{prop} \label{thm_center-action}
The linear category $\Tr(\Ucat^{\ast}(\mf{sl}_n))$ acts by endomorphisms on the commutative ring
$Z\left(\Ucat^{\ast}(\mf{sl}_n) \right)$.  Under this action the class $[f]$ of endomorphisms $f \maps \onelp x \onel \To \onelp x \onel$ in $\Ucat^{\ast}(\mf{sl}_n)$ is sent to the linear operator sending a closed diagram $D$ representing an element of $Z(\lambda)$ to the element in $Z(\lambda')$ obtained by removing the annulus from the diagram $[f.D]$.
\end{prop}
\vspace*{3mm}

\noindent
{\bf Remark for experts}.
Being able to define an action of the trace on the center of objects in a 2-category essentially amounts to having enough ``coherent" duality in the 2-category so that the trace takes
values in the center of objects.  This idea is captured by the notion of a pivotal 2-category~\cite{Muger}.  This can be seen as a many object version of a pivotal monoidal category, see \cite{GMV} where
traces in this context are studied.  M\"{u}ger points out in \cite[page 11]{Muger} a strict pivotal 2-category can be defined from Mackaay's work~\cite{MackSphere} on spherical
2-categories by ignoring the monoidal structure.  It is also equivalent to a 2-category in which  every 1-morphism has a specified cyclic biadjoint (see \cite{CKS,Lau1}).

% - - - - - - - - - - - - - - - - - - - - - - - - - - - - - - - - -
%
\subsubsection{A current algebra action on centers}
%
% - - - - - - - - - - - - - - - - - - - - - - - - - - - - - - - - -

In what follows we represent a generic 2-morphism $f \in Z(\lambda)$, $f \maps
\onel \To \onel$, by the diagram
\[
%% \xy
%%  (0,0)*\xycircle(3,3){-};
%%  (0,0)*{f};
%%  (8,6)*{\lambda};
%% \endxy
%% \qquad {\rm or} \qquad
 \xy
  (0,0)*{*};(0,-4)*{};
  (8,6)*{\lambda};
 \endxy
\]
where we use the $*$ to represent any linear combination of closed diagrams
describing $f$.

%In \cite{BGHL} we will prove the following theorem.
\begin{thm} \label{thm_main}
The vector space
\begin{equation}
 Z := \bigoplus_{\lambda \in \UcatD} Z(\lambda)
\end{equation}
is a $\bfU(\fsl_n[t])$-module with
\begin{eqnarray}
  1_{\lambda} \maps Z &\lra & Z(\lambda ) \\
 \xy
  (0,0)*{\ast};
  (8,6)*{\lambda'};
 \endxy & \mapsto &
  \delta_{\lambda,\lambda'} \left(\;\;   \xy
  (0,0)*{\ast};
  (8,6)*{\lambda};
 \endxy\;\;\;\right)
   \nn\\
  x^+_{i,r}  \maps Z(\lambda) &\lra & Z(\lambda+\alpha_i) \\
 \xy
  %(0,0)*\xycircle(3,3){-};
  (0,0)*{\ast};
  (8,6)*{\lambda};
 \endxy
& \mapsto & (-1)^{(i+1)r}\;\;
  \xy
(0,0)*{\lcbub{i}};
  (-8,-8)*{\bullet}+(-3,-2)*{r};
  (0,0)*{\ast};
  (4,6)*{\lambda};
  (22,6)*{\lambda+\alpha_i};
 \endxy\;\;
   \nn\\
 x^-_{i,s} \maps Z(\lambda) &\lra & Z(\lambda-\alpha_i) \\
 \xy
  %(0,0)*\xycircle(3,3){-};
  (0,0)*{\ast};
  (8,6)*{\lambda};
 \endxy
& \mapsto & (-1)^{(i+1)s} \;\;
 \xy
(0,0)*{\lccbub{i}};
(-8,-8)*{\bullet}+(-3,-2)*{s};
  %(0,0)*\xycircle(3,3){-};
  (0,0)*{\ast};
  (4,6)*{\lambda};
  (22,6)*{\lambda-\alpha_i};
 \endxy
 \nn \\
   \xi_{i,r} \maps Z(\lambda)& \lra & Z(\lambda)  \\
   \xy
  (0,0)*{\ast};
  (8,6)*{\lambda};
 \endxy & \mapsto & (-1)^{(i+1)r}
     \xy
     (0,-6)*{};
  (-9,0)*{p_{i,r}(\lambda)};
  (0,0)*{\ast};
  (8,6)*{\lambda};
 \endxy
\end{eqnarray}

\end{thm}

\begin{proof}
This theorem follows immediately from Proposition ~\ref{prop-cur}.
In particular, the fact that the current algebra relations hold in the trace $\Tr(\Ucat^\ast(\mf{sl}_n))$ imply that these relations hold under the action described above.
\end{proof}

% - - - - - - - - - - - - - - - - - - - - - - - - - - - - - - - - -
%
\subsubsection{An action of the current algebra on centers of 2-representations}
%
% - - - - - - - - - - - - - - - - - - - - - - - - - - - - - - - - -

\begin{defn}
A 2-representation of $\UcatD_Q(\mf{sl}_n)$ is a graded additive $\Bbbk$-linear 2-functor $\Psi \maps \UcatD_Q(\mf{sl}_n)\to \cal{K}$ for some graded, additive 2-category $\cal{K}$.
\end{defn}

When all of the Hom categories $\cal{K}(x,y)$ between objects $x$ and $y$ of $\cal{K}$ are idempotent complete, in other words
$\Kar(\cal{K}) \cong \cal{K}$, then any graded additive $\Bbbk$-linear 2-functor $\Ucat_{Q}(\mf{sl}_n) \to \cal{K}$ extends uniquely to a 2-representation of $\UcatD_Q(\mf{sl}_n)$.

Any 2-representation $\Psi$ induces a map on traces
\begin{align}
\Tr(\UcatD_Q(\mf{sl}_n)) &\to \Tr(\cal{K}) \\ \nn
[f] & \mapsto [\Psi(f)].
\end{align}
When looking at $\U^\ast$ as
\[
\Ucat^{\ast}(x,y) := \bigoplus_{t \in \Z} \Ucat(x, y \la t\ra),
\]
we can easily extend this map to
\begin{align*}
\Tr(\U^\ast(\mf{sl}_n)) &\to \Tr(\cal{K}).
\end{align*}

Hence by the same procedure described in Proposition~\ref{thm_center-action}, the trace $\Tr(\U^\ast(\mf{sl}_n))$ acts on the center $Z(\cal{K})$ of any 2-representation.  Since
Proposition~\ref{prop-cur} shows that the current algebra is contained in the trace, it follows that the current algebra acts on any 2-representation.

%\addcontentsline{toc}{section}{References}

% ==============================================================================
% REFERENCES
%\bibliographystyle{plain}
%\bibliography{bib-survey}

%
% ==============================================================================

\end{document}